\documentclass[dvipsnames]{amsart}
\usepackage{amsmath,amssymb,amsthm,mathtools}
\usepackage[T1]{fontenc}
\usepackage{fullpage}
\usepackage{tikz-cd}
\usepackage{tikz} 
\usepackage{mathrsfs}
\usepackage{anyfontsize}
\usepackage{amstext}
\usepackage{amscd}
\usepackage[mathscr]{euscript}
\usepackage{xspace}
\usepackage{scalerel}
\usepackage{bbm}
\usepackage{enumerate}
\usepackage{url}
\usepackage{enumitem}
\usepackage{color}
\allowdisplaybreaks

\usepackage[pagebackref]{hyperref}
\hypersetup{linktoc=all,} 

\hypersetup{
  colorlinks   = true, 
  urlcolor     = BrickRed, 
  linkcolor    = NavyBlue, 
  citecolor    = DarkOrchid 
}

\usepackage{xparse}

\tikzcdset{arrow style=math font}

\numberwithin{equation}{section} 
\numberwithin{figure}{section}
\usepackage[nameinlink,capitalise,noabbrev]{cleveref}

\newcommand{\newrefformat}[2]{}

\theoremstyle{plain}
\newtheorem{theorem}[equation]{Theorem}
\newtheorem{corollary}[equation]{Corollary}
\newtheorem{proposition}[equation]{Proposition}
\newtheorem{lemma}[equation]{Lemma}

\newtheorem{introtheorem}{Theorem}
 
\crefname{introtheorem}{Theorem}{Theorems}

\crefname{lemma}{Lemma}{Lemmas}
\crefname{theorem}{Theorem}{Theorems}
\crefname{definition}{Definition}{Definitions}
\crefname{proposition}{Proposition}{Propositions}
\crefname{remark}{Remark}{Remarks}
\crefname{corollary}{Corollary}{Corollaries}
\crefname{equation}{Equation}{Equations}
\crefname{construction}{Construction}{Constructions}
\crefname{example}{Example}{Examples}
\crefname{warning}{Warning}{Warnings}
\crefname{appsec}{Appendix}{Appendices}
\crefname{subsection}{Subsection}{Subsections}
\crefname{question}{Question}{Questions}
\crefname{recollection}{Recollection}{Recollections}

\theoremstyle{definition}
\newtheorem{definition}[equation]{Definition}
\newtheorem{example}[equation]{Example}
\newtheorem{construction}[equation]{Construction}
\newtheorem{warning}[equation]{Warning}
\newtheorem{remark}[equation]{Remark}

\usepackage{microtype}

\usepackage[all]{xy}

\usepackage[textwidth=25mm, textsize=tiny]{todonotes}

\newcommand{\category}{{$\infty$-cat\-e\-go\-ry}\xspace}
\newcommand{\categories}{$\infty$-cat\-e\-gories\xspace}
\newcommand{\categorical}{{$\infty$-cat\-e\-gor\-i\-cal}\xspace}

\newcommand{\operad}{{$\infty$-op\-er\-ad}\xspace}
\newcommand{\operads}{{$\infty$-op\-er\-ads}\xspace}

\newcommand{\op}{\mathrm{op}}
\newcommand{\act}{\mathrm{act}}

\newcommand{\cC}{\mathscr{C}}
\newcommand{\cS}{\mathscr{S}}
\newcommand{\cM}{\mathscr{M}}

\newcommand{\cD}{\mathscr{D}}
\newcommand{\oO}{\mathscr{O}}

\newcommand{\cE}{\mathscr{E}}
\newcommand{\zz}{\mathbb{Z}}

\newcommand{\nn}{\mathbb{N}}

\newcommand{\ee}[1]{\mathbb{E}_{#1}}

\newcommand{\m}[1]{\mathscr{#1}}
\renewcommand{\r}[1]{\mathrm{#1}}

\newcommand{\gpd}[1]{\lVert #1 \rVert}

\newcommand{\Ar}{\operatorname{Ar}}

\newcommand{\sAb}{\mathrm{sAb}}
\newcommand{\choo}{\Ch_{\geq 0}^\mathbb{Z}}
\newcommand{\ssSet}{\mathrm{ssSet}}

\newcommand{\brak}[1]{\langle #1 \rangle}

\newcommand{\unit}{\mathbbm{1}}

\newcommand{\Set}{\mathrm{Set}}
\newcommand{\sSet}{\mathrm{sSet}}

\newcommand{\PoSet}{\operatorname{PoSet}}

\newcommand{\usk}{\mathbf{sk}}

\newcommand{\Deltaleq}[1]{\Delta_{\leq #1}}
\newcommand{\nnDelta}{\mathbb{N}\ltimes\Delta^{\op}_{\leq \ast}}

\newcommand{\sleq}[1]{\leq #1}

\DeclareMathOperator{\Hom}{Hom}
\DeclareMathOperator{\Map}{Map}

\DeclareMathOperator{\End}{End}

\DeclareMathOperator{\Fun}{Fun}
\newcommand{\uFun}{\underline{\Fun}}

\newcommand{\Sp}{\mathrm{Sp}}
\DeclareMathOperator{\Fin}{Fin}
\newcommand{\Spaces}{\mathscr{S}}
\DeclareMathOperator{\Ch}{Ch}

\DeclareMathOperator{\Alg}{Alg}

\DeclareMathOperator{\Mul}{Mul}
\newcommand{\Cat}{\mathrm{Cat}}

\DeclareMathOperator{\ev}{ev}
\DeclareMathOperator{\Act}{act}

\DeclareMathOperator{\Mod}{Mod}

\DeclareMathOperator{\id}{id}
\DeclareMathOperator{\N}{N}
\newcommand{\pr}{\mathrm{pr}}

\DeclareMathOperator{\PrL}{\operatorname{Pr^{L}}}

\newcommand{\pull}{\arrow[dr, phantom, "\lrcorner", very near start]}

\NewDocumentCommand{\pullb}{e{_^}}{%
  \mathbin{\mathop{\times}\displaylimits
    \IfValueT{#1}{_{#1}}
    \IfValueT{#2}{^{#2}}
  }%
}

\DeclareMathOperator*{\colim}{colim}
\DeclareMathOperator*{\hocolim}{hocolim}

\tikzset{shorten <>/.style={shorten >=#1,shorten <=#1}}

\DeclareMathOperator{\Fil}{Fil}

\DeclareMathOperator{\Assoc}{Assoc}

\DeclareMathOperator{\sk}{sk}

\newcommand{\ch}{\mathrm{Ch}}
\newcommand{\chp}{\mathrm{Ch}_{\geq 0}}

\newcommand{\norm}{\mathscr{N}}

\newcommand{\Day}{\circledast}

\makeatletter
\newcommand\bigCircledast{\mathop{\mathpalette\b@gCircledast\relax}}
\newcommand\b@gCircledast[2]{%
\vcenter{\hbox{\m@th
\scalebox{\ifx#1\displaystyle 2.6\else1.2\fi}{$#1\circledast$}%
}}\vcenter{\hbox{\rule{0pt}{15pt}}}%
}
\makeatother

\title{On products of skeleta}
\author[Keenan]{Liam Keenan}
\author[P\'eroux]{Maximilien P\'eroux}

\address{Department of Mathematics, Brown University, 151 Thayer Street, Providence RI 02912, USA}
\email{liam\_keenan@brown.edu}

\address{Department of Mathematics, Michigan State University, 619 Red Cedar Road, East Lansing, MI 48823, USA}
\email{peroux@msu.edu}

\usepackage{natbib}
\usepackage{graphicx}

\begin{document}
\renewcommand{\subjclassname}{\textup{2020} Mathematics Subject Classification}


\begin{abstract}
    Given a symmetric monoidal $\infty$-category $\cE$, compatible with finite colimits, we show that the functor sending a simplicial object in $\cE$ to its skeletal filtration is canonically lax symmetric monoidal. This monoidal structure is the analogue of the one induced by the Eilenberg--Zilber homomorphism from the Dold--Kan correspondence.
    To accomplish this, we establish some new results around $\oO$-promonoidal \categories for any $\infty$-operad $\oO$; most notably, we show that it is possible to localize $\oO$-promonoidal \categories in the same way one localizes symmetric monoidal \categories.  
\end{abstract}

\maketitle

\setcounter{tocdepth}{1}
\tableofcontents

\section{Introduction}

Given topological spaces $X$ and $Y$, the singular chain complex of the product $C_*(X\times Y; \zz)$ is not, in general, isomorphic to the tensor product $C_*(X; \zz)\otimes C_*(Y; \zz)$.
Instead, these chain complexes are homotopy equivalent, via the Eilenberg--Zilber homomorphism:
\[
\nabla\colon C_p(X; \zz) \otimes C_q(Y; \zz) \longrightarrow C_{p+q}(X\times Y; \zz).
 \]
Essentially, the assignment $\nabla$ sends a $p$-chain in $X$ and a $q$-chain in $Y$ to a signed sum of $(p+q)$-chains in $X\times Y$, determined by the collection of nondegenerate $(p+q)$ simplices of $\Delta^{p} \times \Delta^{q}$ \cite{EZ, EMcL}. 
These simplices are in bijection with strictly monotone maps of posets $[p+q] \rightarrow [p] \times [q]$, which can be interpreted as $(p,q)$-shuffles in the symmetric group $\Sigma_{p+q}$. 
The classical K\"{u}nneth theorem, relating the homologies of $X$ and $Y$ with the homology of $X\times Y$, can be deduced from the Eilenberg--Zilber homomorphism via a standard homological algebra argument.

In fact, Eilenberg--Zilber produced these maps not just for the singular chain complexes associated to $X$ and $Y$, but for the chain complexes associated to any pair of simplicial abelian groups $A$ and $B$
\[
\nabla\colon C_\ast(A) \otimes C_\ast(B) \longrightarrow C_\ast(A\otimes B).
\]
Furthermore, using the formula informally described above, $\nabla$ induces a canonical chain map on the \textit{normalized chain complexes} associated to the simplicial abelian groups $A$ and $B$
\[
\nabla\colon\norm_*(A)\otimes \norm_*(B) \longrightarrow \norm_*(A\otimes B), 
\]
which remains a chain homotopy equivalence. 
It was independently discovered by Dold and Kan \cite{doldkan1, doldkan2}, that the assignment $A \mapsto \norm_{\ast}(A)$, now called the \textit{Dold--Kan correspondence}, is an equivalence of categories between simplicial abelian groups and non-negatively graded chain complexes. 
More generally, this equivalence of categories holds after replacing abelian groups with any abelian category $\m{A}$, and, if $\m{A}$ is symmetric monoidal (compatibly with the abelian structure) there is an analogous Eilenberg--Zilber map. 

In the language of category theory, the natural map $\nabla$ endows the functor $\norm_{\ast}$ with the structure of a lax symmetric monoidal functor, which can be leveraged to show that the homotopy theory of simplicial algebras is equivalent to the homotopy theory of differential graded rings \cite{DoldKanSS}. 
Both of which are algebraic models for the homotopy theory of connective $\mathbb{E}_1$-algebras over the Eilenberg--Mac Lane spectrum $H\zz$, see \cite{Shipley-HZDGAS} and \cite[4.1.8.4]{HA}. 

There are extensions of the Dold--Kan correspondence to a myriad of other contexts \cite{Shipley-HZDGAS, cosimp-DK, Sore-DK, PerDK, PerDoldKan, DK101, DK102}, and for us the most relevant extension is the \textit{$\infty$-categorical Dold--Kan correspondence} due to Lurie \cite[1.2.4.1]{HA}, which we now recall.
When $\m{E}$ is an \category that admits finite colimits, and $X \colon \Delta^{\op} \rightarrow \m{E}$ is a simplicial object, there is an associated \textit{skeletal filtration}
\[
\sk_{0}(X) \rightarrow \cdots \rightarrow \sk_{n-1}(X) \rightarrow \sk_{n}(X) \rightarrow \cdots 
\]
which determines a functor $\sk_{\ast}^{\m{E}}(X) \colon \nn \rightarrow \m{E}$. 
By $\sk_{n}(X)$, we mean the (homotopy) colimit of $X$ restricted to the full subcategory $\Delta^{\op}_{\leq n} \subseteq \Delta^{\op}$, spanned by nonempty finite linearly ordered sets $[k]$ with $k \leq n$.
The $\infty$-categorical Dold--Kan correspondence is the assertion that the functor 
\[
\sk^{\m{E}}_{\ast} \colon \Fun(\Delta^{\op},\m{E}) \rightarrow \Fun(\nn,\m{E})
\]
is an equivalence of categories whenever $\m{E}$ is stable. 
In this setting, we have exchanged the abelian category $\m{A}$ for a stable \category $\m{E}$, the category of non-negative chain complexes in $\m{A}$ for the category of non-negatively filtered objects in $\m{E}$, and the normalized chain complex functor for the skeletal filtration functor. 

In this article, we construct an analogue of the Eilenberg--Zilber homomorphism for the skeletal filtration functor defined above. 
In more detail, for any pair of simplicial objects $X$ and $Y$ in an \category $\cE$ with finite colimits and a compatible symmetric monoidal product $\otimes$, we show there are natural maps $\cE$:
\begin{equation}\label{eq: intro-sk-ez}
\nabla\colon \sk_p^\cE(X)\otimes \sk_q^\cE(Y)\longrightarrow \sk_{p+q}^\cE(X\otimes Y)
\end{equation}
for all $p,q \geq 0$, which determine a natural map of filtered objects
\[
\nabla \colon \sk_*^\cE(X)\Day \sk_*^\cE(Y)\longrightarrow \sk_*^\cE(X\otimes Y). 
\]
Here, $\Day$ denotes the Day convolution monoidal product on filtered objects in $\cE$. 
We shall say $\cE$ is \textit{finitely cocomplete} symmetric monoidal \category if $\cE$ is an \category with finite colimits and a symmetric monoidal structure for which the bifunctor $-\otimes -\colon \cE\times \cE\to \cE$ preserves finite colimits separately in each variable. 
The above maps determine the following result.

\begin{introtheorem}[\cref{theorem: main theorem on EZ for sk}]\label{introtheorem: EZ structure}
Let $\cE$ be a finitely cocomplete symmetric monoidal \category. The skeletal filtration functor
\[\sk_*^\cE\colon \Fun(\Delta^\op, \cE)\longrightarrow \Fun(\nn, \cE)\] 
admits a canonical lax symmetric monoidal structure, with respect to the pointwise tensor product on simplicial objects in $\cE$ and the Day convolution product on filtered objects in $\cE$.
\end{introtheorem}

In the special case where $\cE$ is the \category of spectra, we show that the lax symmetric monoidal structure above recovers the usual Eilenberg--Zilber homomorphism. 
This result, combined with work of Hedenlund--Krause--Nikolaus \cite{hedenlund-krause-nikolaus}, shows that the spectral sequence associated to a simplicial $\ee{\infty}$-ring has a (graded) commutative multiplicative structure.  

One of technical challenges in establishing \cref{introtheorem: EZ structure} is that it is not enough to simply specify the maps $\nabla$ in \cref{eq: intro-sk-ez}: we must also specify all the homotopy coherences. 
A standard way to address this issue is to utilize the universal properties of the objects under consideration.
However, the symmetric monoidal structures on the source and target of $\sk_{\ast}^{\m{E}}$ have seemingly different origins: one is a pointwise tensor product, and the other is a convolution product arising from a symmetric monoidal structure on $\nn$. 
Because $\Delta^{\op}$ is not a symmetric monoidal category, the standard methods for manipulating convolution products do not apply. 
Nevertheless, this apparent difference can be remedied by working with \textit{symmetric promonoidal categories}: essentially, this is the minimal necessary structure on a category  required to make sense of a Day convolution product.

Informally, a symmetric monoidal structure on a category $\m{C}$ is encoded by a functor $\m{C} \times \m{C} \rightarrow \m{C}$, while a symmetric promonoidal structure on a category $\m{C}$ is encoded by a \textit{profunctor} $\m{C} \times \m{C} \nrightarrow \m{C}$, which is another name for a functor of the form $\m{C}^{\op}\times \m{C}^{\op} \times \m{C} \rightarrow {\rm Set}$. 
From this point of view, $\Delta^{\op}$ does admit a symmetric promonoidal structure, given by the profunctor 
\(
\Delta\times \Delta\times \Delta^\op\to \Set
\)
sending a $([p], [q] ; [n])$ to the collection of poset maps of the form $[n]\to [p]\times [q]$. 
Note that any such poset map factors through a $(p,q)$-shuffle, which crucially appear in the construction of the classical Eilenberg--Zilber homomorphism.

By studying this symmetric promonoidal structure on $\Delta^{\op}$, we deduce \cref{introtheorem: EZ structure} from a more fundamental result. 
To explain this, note that for $n \geq 0$, we may identify $\Fun(\Delta^{\op}_{\leq n},\m{E})$ with the full subcategory of ``$n$-skeletal'' objects in $\Fun(\Delta^{\op},\m{E})$, which induces an ascending filtration of the \category $\Fun(\Delta^{\op},\m{E})$: 
\[
\Fun(\Delta_{\leq 0}^{\op},\cE) \rightarrow \cdots \rightarrow \Fun(\Delta_{\leq n-1}^{\op},\cE) \rightarrow \Fun(\Delta_{\leq n}^{\op},\cE) \rightarrow \cdots. 
\]
We prove that the (pointwise) tensor product of an $p$-skeletal simplicial object $X$, with an $q$-skeletal simplicial object $Y$, results in an $(p+q)$-skeletal simplicial object $X \otimes Y$. 
This endows the filtration above with a tensor product that respects the filtration degree. 
We make this notion precise and refer to these as \textit{$\nn$-monoidal \categories} (\cref{rem: whay is an N-monoidal cat}).

\begin{introtheorem}[\cref{theorem: products of skeletal objects are skeletal}]\label{introtheorem: products of skeleta}
Let $\cE$ be a finitely cocomplete symmetric monoidal \category. 
The pointwise tensor product of simplicial objects endows the filtered \category 
\[
\Fun(\Delta_{\leq 0}^{\op},\cE) \rightarrow \cdots \rightarrow \Fun(\Delta_{\leq n-1}^{\op},\cE) \rightarrow \Fun(\Delta_{\leq n}^{\op},\cE) \rightarrow \cdots
\]
with a canonical $\nn$-monoidal structure. 
\end{introtheorem}

In order to establish \cref{introtheorem: products of skeleta} and to deduce \cref{introtheorem: EZ structure} from it, we need a toolkit for manipulating $\oO$-promonoidal \categories, $\m{C}$, and the associated Day convolution product on $\Fun(\m{C},\m{E})$, when $\m{E}$ is a finitely cocomplete symmetric monoidal \category. 
These convolution products were originally invented by Day in \cite{day-promonoidal} for ordinary categories, and their \categorical counterparts have been studied in a number of works \cite{spectralmackeyii, shah2021parametrized, linskens-nardin-pol}; see also \cite{Glasman,HA,NikolausOperads} for the case where $\m{C}$ is monoidal instead of promonoidal. 
We draw special attention to the recent paper of Linskens--Nardin--Pol \cite{linskens-nardin-pol}, who establish the sort of toolkit we need, but with the stronger assumption that $\cE$ has all small colimits, meaning their results cannot be directly applied to our situation. 

We proceed by reviewing the theory of $\m{O}$-promonoidal \categories and establishing a number of fundamental technical results about the Day convolution product on $\Fun(\m{C},\m{E})$, where crucially we do not assume that $\m{E}$ admits all small colimits. 
We then prove the following result classifying $\m{O}$-promonoidal structures, extending \cite[3.37]{linskens-nardin-pol} which only covers the symmetric promonoidal case. 

\begin{introtheorem}[\cref{proposition: Day conv recovers presentably O-monoidal str}]
    Let $\oO$ be an \operad and let $p\colon \m{C} \rightarrow \m{O}$ be a coCartesian fibration. 
    Then, any $\m{O}$-promonoidal structure extending $p$, classifies, and is classified by, a presentably $\m{O}$-monoidal structure on the relative copresheaf category $\uFun_{\m{O}}(\m{C},\Spaces_{\oO})$.
    In other words, every presentably $\m{O}$-monoidal structure on $\uFun_{\m{O}}(\m{C},\Spaces_{\oO})$ arises from Day convolution. 
\end{introtheorem}

With the theorem above in mind, the symmetric promonoidal structure on $\Delta^\op$ discussed above is the one which classifies the symmetric monoidal structure on $\Fun(\Delta^\op, \cS)$ given by the Cartesian product of simplicial spaces. 
Our main application of the classification result above is the construction of $\m{O}$-promonoidal localizations. 

Recall that in \categories, we can universally invert a chosen collection $W$ of arrows in an \category $\cC$, to obtain a new \category $\cC[W^{-1}]$ called the localization of $\cC$ with respect to $W$, sometimes also called the Dwyer--Kan localization \cite[1.3.4.1]{HA}. If $\cC$ is symmetric monoidal, the localization can also be made symmetric monoidal, see \cite[4.1.7.4]{HA}, \cite[A.4, A.5]{nikolaus-scholze} and \cite{hindk}. We establish the following promonoidal analogue.

\begin{introtheorem}[\cref{theorem: promonoidal structures exist}]
    Let $\m{O}$ be an $\infty$-operad and $\cC$ be an $\m{O}$-promonoidal \category.
    Given $W$, an $\m{O}$-saturated collection of arrows in $\cC$ compatible with promonoidal structure on $\cC$, there is an $\m{O}$-promonoidal structure on the localization $\cC[W^{-1}]$, such that the universal functor $\cC\to \cC[W^{-1}]$ is $\m{O}$-promonoidal. 
\end{introtheorem} 

We use this result to describe $\sk^\cE_*$ in terms of a left Kan extension along a promonoidal localization, which is a key ingredient deducing that $\sk^\cE_*$ is lax symmetric monoidal. 

Recall that any presentably symmetric monoidal \category is equivalent to an underlying \category of a combinatorial symmetric monoidal model category \cite{Nikolaus-Sagave}.
We can therefore prove \cref{introtheorem: EZ structure}  using model categories in the case where $\cE$ is presentable, and we do so in the appendices (\cref{cor: EZ for presentable case}). This approach has the advantage of giving a point-set approach in the definition of the Eilenberg--Zilber homomorphism. We also relate the $\infty$-categorical filtration $\sk^\cE_*$ with the usual skeletal filtration on ordinary categories.

The classical Eilenberg--Zilber homomorphism $\nabla\colon \norm_*(A)\otimes \norm_*(B)\to \norm_*(A\otimes B)$ for simplicial abelian groups $A$ and $B$ has a homotopy inverse $\Delta\colon \norm_*(A\otimes B)\to\norm_*(A)\otimes \norm_*(B)$ called the \textit{Alexander--Whitney homomorphism}, which provides a (non-symmetric) oplax monoidal structure on the normalization functor $\norm_*\colon \sAb\to \choo$. 
In forthcoming work with Tyler Lawson, we construct an analogous oplax monoidal structure on the skeletal filtration.

\subsection*{Overview}
In \cref{section: promonoidal cat}, we review promonoidal structures, both in ordinary categories and $\infty$-categories. In \cref{section: Day convolution and props} we review the Day convolution induced by promonoidal structures and its general properties.
We prove our classification and localization results for promonoidal \categories in \cref{section: promonoidal localization}.
We review the skeletal filtration and its relationship with normalization in \cref{section: the skeletal filtration}, as well as the classical Eilenberg--Zilber homomorphism. In \cref{section: EZ map on skeleton} is where we prove the the skeletal filtration is lax symmetric monoidal, and recovers the usual Eilenberg--Zilber homomorphism on the hearts. 
In \cref{section: base-change and Kan extensions}, we collect general results about compatibilities between Kan extensions and base-changes.
In \cref{sec: Appendix Day convo agree}, we show that the Day convolution product for model categories agrees with the Day convolution on their underlying \categories. In \cref{section: appendix EZ model cat}, we prove \cref{introtheorem: EZ structure} for presentably symmetric monoidal \categories, using model categories.

\subsection*{Conventions and notation}
In this article, we will make frequent use of the language of higher category theory and higher algebra as developed in the books \cite{HTT} and \cite{HA}. 
In particular, we will make use of symmetric monoidal \categories, \operads, and algebras over \operads. 
We also opt to use the word (co)limit in place of the more appropriate phrase ``homotopy (co)limit'' throughout. 
Throughout the bulk of this article, we commit the standard practice of writing $\m{C}$ in place of $N(\m{C})$ for the \category associated to a $1$-category. 
We make an exception in \cref{sec: Appendix Day convo agree,section: appendix EZ model cat}, where we need to carefully track such things. 

Throughout this article, we consider \categories of varying sizes, following the conventions in \cite{HTT} and \cite{kerodon}. 
That is, we assume that for every cardinal $\lambda_0$, there exists a strongly inaccessible cardinal $\lambda \geq \lambda_0$. 
We let $U(\lambda)$ denote the collection of all sets which are $\lambda$-small, and we declare a set to be small provided that it is $\lambda$-small. 
We write $\Spaces$ and $\Cat_\infty$ for the \categories of small spaces and \categories, respectively, which are themselves \textit{large} \categories; that is, they are $\lambda'$-small for some suitable choice of strongly inaccessible cardinal $\lambda' > \lambda$.
By choosing yet another strongly inaccessible cardinal $\lambda'' > \lambda'$, we can consider the \textit{huge} \category $\PrL$ of presentable \categories and colimit perserving functors between them; see \cite[4.8.1]{HA} for a construction. 
Recall that $\PrL$ is a closed symmetric monoidal \category, with the Lurie tensor product, whose monoidal unit is the $\infty$-category of spaces $\cS$, and whose internal hom is the presentable $\infty$-category $\Fun^L(\cC, \m{D})$ of colimit-preserving functors from $\m{C}$ to $\m{D}$ \cite[4.8.1.24]{HA}.

Throughout the text, we make repeated usage of the following $1$-categories: the category $\Fin_\ast$ of finite pointed sets and basepoint preserving maps between them, the category $\Delta$ of (nonempty) finite linearly ordered sets and monotone maps between them, and the category $\nn = \{0 \leq 1 \leq 2 \leq 3 \leq \cdots \}$ of nonnegative integers with its poset structure. 

\subsection*{Acknowledgements}
The authors would like to thank David White for some helpful pointers to the literature.
The present work has benefited from conversations with Maxine Calle, David Chan, Tom Goodwillie, Peter Haine, Asaf Horev, Shai Keidar, Cary Malkiewich, and Arpon Raksit.  
The authors warmly thank Tyler Lawson for his mathematical insights, encouragement, and interest in this project.

\section{Promonoidal categories}\label{section: promonoidal cat}

In this section, we review the basics of profunctors and $\m{O}$-promonoidal \categories, as well as fix some notation for the sequel. 
For a more in-depth treatment of these ideas, we refer the reader to \cite{day-promonoidal} and \cite{linskens-nardin-pol}.

\subsection{Symmetric promonoidal categories}
Let $\cC$ and $\cD$ be small categories.
We denote by $\m{P}(\cC)=\Fun(\cC^\op, \Set)$ the category of presheaves of $\cC$.
A profunctor $F\colon \cC \nrightarrow \cD$ is a functor $F\colon \cC^\op\times \cD\to \Set$, or equivalently a colimit preserving functor $\m{P}(\cD)\to \m{P}(\cC)$.
Given two profunctors $F \colon \m{C} \nrightarrow \m{D}$ and $G \colon \m{D} \nrightarrow \m{E}$, we define their composition, $G\circ F$, to be the profunctor specified by the coend
    \[
    (G\circ F)(-,-) = \int^{d \in \m{D}} F(-,d) \times G(d,-) \colon \m{C}^\op \times \m{E} \rightarrow \Set. 
    \]
    Profunctors may also equivalently be seen as correspondences, see more details in \cite[1.1.7]{ayala-francis-fibrations} for instance.
    We can consider the profunctorial analogue of symmetric monoidal categories, known as symmetric \textit{promonoidal} categories.
Informally, a promonoidal category consists of the following structure: 
\begin{enumerate}
    \item a small category $\m{C}$;
    \item a ``symmetric monoidal product'' profunctor $\mu\colon \m{C} \times \m{C} \nrightarrow \m{C}$ and a ``unit'' profunctor $\eta\colon \{\ast\} \nrightarrow \m{C}$;  
    \item natural isomorphisms expressing associativity, commutativity, and unitality constraints, e.g., the pentagon axiom. 
\end{enumerate}
Just as a symmetric monoidal category $\cC$ is equivalent to a coCartesian fibration $\cC^\otimes\to \Fin_*$ that satisfy the Segal's condition, we may do a similar construction in the promonoidal case.
To facilitate the following construction, we write $\mu(c_{1},c_{2};c')$ for the value of $\mu$ on a triple.
By associativity, there is a profunctor $\cC\times \cC\times \cC\nrightarrow \cC$ that can be written either as $\mu\circ (1\times \mu)$ or $\mu\circ (\mu\times 1)$ up to natural isomorphism.
We denote the value of this profunctor still as $\mu$ and is defined on a $4$-tuple $(c_1, c_2, c_3, c')$ in $\cC$ as the coend in $\Set$:
\begin{align*}
    \mu(c_1, c_2, c_3; c')&= (\mu \circ (1\times \mu))((c_1, c_2, c_3), c')\\
    &=\int^{(d_1,d_2)\in \cC\times \cC} \Hom_{\cC}(c_1, d_1)\times \mu(c_2, c_3; d_2)\times \mu(d_1, d_2; c')\\
    &\cong \int^{(d_1,d_2)\in \cC\times \cC} \mu(c_1, c_2; d_1)\times \Hom_\cC(c_3, d_2)\times \mu(d_1, d_2 ; c').
\end{align*}
We can then inductively define for $n\geq 3$:
\[
\mu(\{c_{i}\}_{1\leq i \leq n};c')
= (\mu \circ (1\times \mu))(\{c_{i}\}_{1\leq i \leq n}, c'). 
\]
\begin{construction}\label{construction: operator category of a symmetric promonoidal category}
Let $\m{C}$ be a symmetric promonoidal category.
Define its category of operators $\m{C}^{\otimes}$ as follows.  
\begin{enumerate}
\item An object in $\m{C}^{\otimes}$ is a finite sequence $c_{1},\dots,c_{n} \in \m{C}$, which we denote by $[c_{1},\dots,c_{n}]$. 
\item A morphism in $\m{C}^{\otimes}$, displayed as
$
[c_{1},\dots,c_{m}] \rightarrow [c'_{1},\dots,c'_{n}]$,
is given by a map $f \colon \brak{m} \rightarrow \brak{n}$ in $\Fin_{\ast}$ together with a collection of elements 
\[
\left\lbrace \alpha_{j} \in \mu(\{c_{i}\}_{i \in f^{-1}\{j\}};c'_{j}) \mid 1\leq j \leq n\right\rbrace.
\]
\item Composition in $\m{C}^{\otimes}$ is determined by composition in $\Fin_{\ast}$ together with the natural map
\[
\left(\prod_{j \in g^{-1}(k)} \mu(\{c_{i}\}_{i\in f^{-1}\{j\}};c'_{j})\right) \times \mu(\{c'_{j}\}_{j\in g^{-1}\{k\}};c''_{k}) \rightarrow \mu(\{c_{i}\}_{i\in (g\circ f)^{-1}\{k\}};c''_{k}).
\]
This map is obtained by viewing the set 
$\mu(\{c_{i}\}_{i\in (g\circ f)^{-1}\{k\}};c''_{k})$
as the value of a composition of profunctors, and using the fact that the source appears in the coend which calculates the target. 
\end{enumerate}
We obtain an induced functor $\m{C}^{\otimes} \rightarrow \Fin_{\ast}$ sending an object $[c_1, \cdots, c_n]$ to $\brak{n}$.
\end{construction}

As in the case of symmetric monoidal categories, a symmetric promonoidal structure can be encoded in terms of its category of operators; see \cite[2.1.1]{HA} for details. 

\begin{proposition}\label{prop: ordinary promonoidal category as flat inner fibration}
Let $\m{C}$ be a symmetric promonoidal category, and let $p\colon \m{C}^{\otimes} \rightarrow \Fin_\ast$ be the symmetric multicategory determined by \cref{construction: operator category of a symmetric promonoidal category}. 
Additionally, let $\m{C}^{\otimes}_{\act} \rightarrow \Fin$ be the pullback of $p$ along the full subcategory spanned by the active maps. 
Then, $p$ satisfies the following condition: 
\begin{description}
    \item[$(\ast)$] 
    for each map $\{0<1<2\} \rightarrow \Fin$, the following induced map is an equivalence of categories:
\[
\m{C}_{\act}^{\otimes}\pullb_{\Fin}\{0<1\} \coprod_{\m{C}_{\act}^{\otimes}\pullb_{\Fin}\{1\}} \m{C}_{\act}^{\otimes}\pullb_{\Fin}\{1<2\} \longrightarrow \m{C}_{\act}^{\otimes}\pullb_{\Fin}\{0<1<2\}.
\]
\end{description}
In fact, any symmetric promonoidal category is uniquely determined by a symmetric multicategory satisfying the condition above.
\end{proposition}

The proposition is recorded here for context. Since it will not be used in what follows, we omit the proof.
The construction of $\m{C}^{\otimes}$ and the result above implies that $\m{C}^{\otimes}_{\brak{m}} \simeq \prod_{i=1}^{m} \m{C}$, and for every map $f \colon \brak{m} \rightarrow \brak{n}$, we have a profunctor $H_{f}\colon \m{C}^{\otimes}_{\brak{m}} \nrightarrow \m{C}^{\otimes}_{\brak{n}}$, and a natural isomorphism
\[
H_{g}\circ H_{f} \xrightarrow{\cong} H_{g\circ f}. 
\]
In other words, we have a commutative pseudo-monoid in the bicategory $\Cat^{\mathrm{pro}}$ of small categories with promonoidal functors. 
If, additionally, we knew that for each (unique) active map $\phi\colon \brak{m} \rightarrow \brak{1}$, and for each tuple $[c_{1},\dots,c_{n}]$, the induced functor
\[
H_{\phi}(\{c_{i}\}_{1\leq i \leq n};-) \colon \m{C} \rightarrow \Set, 
\]
was corepresentable, then the commutative pseudo-monoid in $\Cat^{\mathrm{pro}}$ would give a commutative pseudo-monoid in $\Cat$. 

One reason promonoidal structures are desirable is that they give rise to Day convolution structures on presheaves and functor categories as we recall in next section, see \cref{construction: Day convolution in ordinary categories}.

\subsection{Exponentiable fibrations and promonoidal structures}\label{subsection: flat inner fibrations and symmetric promonoidal categories}
In the case of \categories, profunctors, and by extension, promonoidal structures, can be conveniently encoded by \textit{exponentiable fibrations}. 
These have been studied extensively by Ayala--Francis in \cite{ayala-francis-fibrations}. 

\begin{definition}
    A functor of \categories $p\colon \m{C} \rightarrow \m{O}$ is an \textit{exponentiable fibration}\footnote{These are also sometimes referred to in the literature as \textit{Condouch\'{e} fibrations}.}
    provided the functor 
    \[
    p^{\ast}=\m{C}\times_{\m{O}} (-) \colon \Cat_{\infty/\m{O}} \rightarrow \Cat_{\infty/\m{C}}
    \]
    admits a right adjoint, that we denote $p_{\ast}$. 
\end{definition}

\begin{remark}
    Thanks to \cite[Lemma 2.2.8]{ayala-francis-fibrations}, we can think of an exponentiable fibration as encoding: 
    \begin{enumerate}
        \item for each $x \in \m{O}$ an \category $\m{C}_{x} = \m{C}\times_{\m{O}} \{x\}$;
        \item for each $\phi \colon x \rightarrow y$ in $\m{O}$, a profunctor $\m{C}_{\phi} \colon \m{C}_{x}^{\op}\times\m{C}_{y} \rightarrow \Spaces$;
        \item for each composite $\psi\circ \phi \colon x \rightarrow y \rightarrow z$, a natural equivalence of profunctors 
        \(
        \m{C}_{\psi} \circ \m{C}_{\phi} \xrightarrow{\sim} \m{C}_{\psi\circ \phi};
        \)
        \item plus higher coherences.
    \end{enumerate}
\end{remark}

Let $\m{O}^{\otimes}$ be an \operad and let $I$ be a finite set. 
An object in $\m{O}^{\otimes}$ lies over a finite pointed set $I_+$, and by the Segal condition, such an object can be identified with a tuple of the form $\{x_i\}_{i\in I}$. 
For brevity, we use the notation $\{x_i\}$ when there is no chance of confusion. 
Furthermore, we write $\m{O}^{\otimes}_{\Act} = \m{O}^{\otimes} \times_{\Fin_\ast} \Fin$ for the wide subcategory of active morphisms in $\m{O}^{\otimes}$. 

\begin{definition}
    Let $p\colon \m{O}^{\otimes} \rightarrow \Fin_{\ast}$ be an $\infty$-operad. 
    We say that $\m{O}^{\otimes}$ is a \textit{symmetric promonoidal \category} if the
    induced functor $p_\act\colon \m{O}^{\otimes}_{\act} \rightarrow \Fin$ is an exponentiable fibration.
\end{definition}

\noindent A straightforward application of \cref{prop: ordinary promonoidal category as flat inner fibration} yields the following result. 

\begin{corollary}\label{cor: promonoidal nerve}
    Let $\m{C}$ be a symmetric promonoidal category. 
    Then $p\colon \N(\m{C}^{\otimes}) \rightarrow \Fin_{\ast}$ is a symmetric promonoidal \category.
\end{corollary}

Just as with the theory of $\m{O}$-monoidal \categories in \cite[2.1.2]{HA}, there is an analogous notion of an $\m{O}$-promonoidal \category. 

\begin{definition}
    Let $p\colon \m{C}^{\otimes} \rightarrow \m{O}^{\otimes}$ be a fibration of $\infty$-operads. 
    We say that $p$ is an \textit{$\m{O}$-promonoidal \category} provided that $p_{\Act}\colon \m{C}^{\otimes}_{\Act} \rightarrow \m{O}^{\otimes}_{\Act}$ is an exponentiable fibration. 
    We say that $p$ is an \textit{$\m{O}$-monoidal \category} provided $p$ is a coCartesian fibration. 
\end{definition}

\begin{remark}
    Any $\m{O}$-promonoidal \category determines an exponentiable fibration of \categories $\m{C}^{\otimes}_{\brak{1}} \rightarrow \m{O}^{\otimes}_{\brak{1}}$, which we denote by $\m{C} \rightarrow \m{O}$.  
    Additionally, if $x \in \m{O} = \m{O}^{\otimes}_{\brak{1}}$, we use the notation $\m{C}_{x}$ in place of $\m{C}^{\otimes}_{\{x\}}$. 
    The authors find it helpful to think about an $\m{O}$-promonoidal \category $\m{C}^{\otimes}$ as endowing an exponentiable fibration $\m{C} \rightarrow \m{O}$ with an $\m{O}$-promonoidal structure.
\end{remark}

\begin{example}
    Any $\m{O}$-monoidal \category is an $\m{O}$-promonoidal \category. 
    For instance, when $\m{O} = \Assoc$, we see that monoidal \categories can be viewed as $\Assoc$-promonoidal \categories.
\end{example}

\begin{example}
    Let $I$ be small \category, and consider the presheaf category $\m{P}(I^{\op}\times I) = \Fun(I\times I^{\op},\Spaces)$, and recall there is a canonical identification 
    \[
    \m{P}(I^{\op}\times I) \simeq \m{P}(I^{\op}) \otimes \m{P}(I) \simeq \End^{L}(\m{P}(I))
    \]
    in $\PrL$, by the dualizability of presheaf categories. 
    As $\End^{L}(\m{P}(I))$ is an associative algebra in $\PrL$, this equivalence endows $\m{P}(I^{\op}\times I)$ with the structure of a presentably monoidal \category, which, by \cref{proposition: closed O-monoidal structures are O-promonoidal}, endows $I \times I^{\op}$ with the structure of an $\Assoc$-promonoidal \category; classically, this observation is due to Day--Street \cite[Section 7, Example 8]{DayRoss}.
\end{example}

\begin{definition}
Let $p\colon \m{C}^{\otimes} \rightarrow \m{O}^{\otimes}$ be a fibration of \operads, and let $\phi \colon \{x_i\} \rightarrow y$ be an active morphism in $\m{O}^{\otimes}$. 
Given $\{c_i\} \in \m{C}^{\otimes}_{\{x_i\}}$ and $c\in \m{C}^{\otimes}_{y}$, define \textit{the space of $\phi$-multimorphisms in $\m{C}^{\otimes}$} to be 
\[
\Mul^{\phi}_{\m{O}}(\{c_i\},c) = \Map_{\m{C}^{\otimes}}(\{c_i\},c) \times_{\Map_{\m{O}}(\{x_i\},y)} \{\phi\}.
\]
We say that $p$ is an \textit{$\m{O}$-corepresentable \operad} provided that 
\(
\Mul^{\phi}_{\m{O}}(\{c_i\},-) \colon \m{C}^{\otimes}_{y} \rightarrow \Spaces
\)
is a corepresentable functor for all active maps $\phi$ and all objects $\{c_i\}$.
\end{definition}

The following lemma provides a useful method for recognizing $\m{O}$-monoidal \categories among $\m{O}$-promonoidal \categories and $\m{O}$-corepresentable \operads.

\begin{lemma}\label{lem: corepresentable operad fibration is locally coCartesian}
    Let $p\colon \m{C}^{\otimes} \rightarrow \m{O}^{\otimes}$ be a fibration of \operads.
    \begin{enumerate}
        \item The fibration $p$ is locally coCartesian if and only if $p$ is $\m{O}$-corepresentable.  
        \item The fibration $p$ is $\m{O}$-monoidal if and only if it is $\m{O}$-corepresentable and $\m{O}$-promonoidal. 
    \end{enumerate}
\end{lemma}

\begin{proof}
The first claim follows from the observation in \cite[6.2.4.5]{HA} extended to the $\m{O}$-monoidal case, which can be done using \cite[2.4.2.9]{HTT}. 

For the second claim, one direction is clear, so assume $p$ is $\m{O}$-corepresentable and $\m{O}$-promonoidal. 
By \cite[1.1.8.2(a)]{ayala-francis-fibrations}, $p_{\Act}$ is coCartesian as it is both exponentiable and locally coCartesian by assumption. 
However, as $p$ is a fibration of \operads the fact that $p_{\Act}$ is coCartesian implies $p$ is also coCartesian. 
\end{proof}

\begin{definition}\label{definition: (lax) O-promonoidal functors}
    Let $\m{C}$ and $\m{D}$ be $\m{O}$-promonoidal \categories. 
    A \textit{lax $\m{O}$-promonoidal functor} from $\m{C}$ to $\m{D}$ is a commutative diagram of \operad maps
    \[
    \begin{tikzcd}
	{\m{C}^{\otimes}} && {\m{D}^{\otimes}} \\
	& {\m{O}^{\otimes}}
	\arrow["f", from=1-1, to=1-3]
	\arrow["p"', from=1-1, to=2-2]
	\arrow["q", from=1-3, to=2-2]
    \end{tikzcd}
    \]
    We say that $f$ is \textit{$\m{O}$-promonoidal} if, additionally, for all active maps $\phi\colon\{x_{i}\} \rightarrow y$ in $\m{O}^{\otimes}$ and for all $\{c_{i}\}\in \m{C}^{\otimes}_{\{x_{i}\}}$, the natural map
    \[
    (f_{y})_{!}\Mul_{\m{C}}^{\phi}(\{c_{i}\},-) \rightarrow \Mul_{\m{D}}(\{f_{x_{i}}(c_{i})\},-)
    \]
    is an equivalence of functors $\m{D}^{\otimes}_{y} \rightarrow \Spaces$. 
\end{definition}

\begin{remark}
If $\m{C}$ and $\m{D}$ are $\m{O}$-monoidal \categories, then a map of \operads $F\colon \m{C}\rightarrow \m{D}$ is $\m{O}$-promonoidal if and only if it is $\m{O}$-monoidal in the sense of \cite[2.1.2.15]{HA}.
The reader unfamiliar with $\m{O}$-promonoidal categories may find this an instructive exercise. 
\end{remark}

We conclude this section by recording a checkable criterion for when a sub-\operad of an $\m{O}$-promonoidal \category is itself $\m{O}$-promonoidal. 

\begin{proposition}\label{proposition: sub operad of promonoidal is promonoidal}
    Let $p\colon \m{C}^{\otimes} \rightarrow \m{O}^{\otimes}$ be an $\m{O}$-promonoidal \category, let $\m{D}$ be a full subcategory of $\m{C} = \m{C}^{\otimes}_{\brak{1}}$, and let $\m{D}^{\otimes}$ be the full sub \operad of $\m{C}^{\otimes}$ spanned by the objects of $\m{D}$.
    Subsequently, $\iota\colon \m{D}^{\otimes} \rightarrow \m{C}^{\otimes}$, is a map of \operads over $\m{O}^\otimes$ and the following are equivalent:
    \begin{enumerate}
        \item The map of \operads $p\circ \iota \colon \m{D}^{\otimes} \rightarrow \m{O}$ exhibits $\m{D}^{\otimes}$ as an $\m{O}$-promonoidal category and $\iota$ is an $\m{O}$-promonoidal functor. 
        \item For each active morphism $\phi\colon \{x_i\} \rightarrow y$ in $\m{O}^{\otimes}$, and for each $\{d_{i}\} \in \m{D}_{\{x_i\}}^{\otimes}$, the induced map 
        \[
        (\iota_{y})_{!}\Mul^{\phi}_{\m{D}}(\{d_i\},-) \rightarrow \Mul_{\m{C}}^{\phi}(\{\iota(d_{i})\},-)
        \]
        is an equivalence of functors $\m{C}_{y} \rightarrow \Spaces$. 
    \end{enumerate}
\end{proposition}

\begin{proof}
    That (1) implies (2) is immediate from the definitions, so we demonstrate (2) implies (1). 
    Assuming we have shown that $\m{D}^{\otimes}$ is $\m{O}$-promonoidal via $p\circ \iota$, condition (2) is precisely the statement that $\iota$ is an $\m{O}$-promonoidal functor.
    Therefore, it will suffice to show that $p\circ \iota$ exhibits $\m{P}^{\otimes}$ as a promonoidal \category. 
    To this end, note that by \cite[Lemma 2.2.8]{ayala-francis-fibrations}, the induced map $\m{D}^{\otimes}_{\act} \rightarrow \m{O}^{\otimes}_{\act}$ is expontentiable if and only if for all active morphisms in $\m{O}^{\otimes}$ of the form $\phi \colon \{x_i\}_{i\in I} \rightarrow \{y_j\}_{j\in J}$ and $\psi\colon {y_j}_{j\in J} \rightarrow z$, the induced map  
    \[
    \int^{\{d_{j}\}_{j\in J} \in \m{D}^{\otimes}_{y_J}} \prod_{j \in J} \Mul^{\phi_j}_{\m{D}}(\{d_i\}_{i \in f^{-1}\{j\}},d_{j}) \times \Mul_{\m{D}}^{\psi}(\{d_{j}\}_{j\in J},-) \rightarrow \Mul^{\psi\circ \phi}_{\m{D}}(\{d_i\}_{i \in I}, -)
    \]
    is an equivalence of functors $\m{D}^{\otimes}_{z} \rightarrow \Spaces$, for all $\{d_i\}_{i\in I} \in \m{D}^{\otimes}_{\{x_i\}_{i\in I}}$. 
    As $\m{D}^{\otimes}$ is a sub-\operad of $\m{C}^{\otimes}$, we can identify the relevant multimorphism spaces of $\m{D}^{\otimes}$ with multimorphism spaces of $\m{C}^{\otimes}$. 
    Combining the aforementioned identification of multimorphism spaces with the natural equivalences
    \[
    (\iota_{y_j})_{!}\Mul_{\m{D}}^{\phi_{j}}(\{d_i\}_{i \in f^{-1}\{j\}},-) \simeq \Mul_{\m{C}}^{\phi_j}(\{\iota(d_i)\}_{i\in f^{-1}\{j\}},-), 
    \]
    guaranteed by (2), the projection formula \cref{lem: projection formula} implies the desired transformation is an equivalence. 
\end{proof}

\section{Day convolution and its properties}\label{section: Day convolution and props}
We begin by recalling the definition of the Day convolution \operad due to Shah \cite{shah2021parametrized}, and then adapt a number of technical results from \cite{linskens-nardin-pol} which we will need later on. 
Specifically, there is a blanket cocompleteness assumption made in \cite{linskens-nardin-pol} which is not strictly necessary for many of their results; this should be compared with the treatment of the Day convolution \operad in \cite[2.2.6.15, 2.2.6.16]{HA}.
This technical work occupies \cref{subsection: existence of convolution products} and \cref{subsection: functoriality of convolution operad}.

We begin by recalling the Day convolution in ordinary categories.
The following construction was first made by Day in \cite{day-promonoidal}.

\begin{construction}[{\cite[\S 3]{day-promonoidal}}]\label{construction: Day convolution in ordinary categories}
If $\cC$ is a small symmetric promonoidal category and $\m{E}$ is a cocomplete symmetric monoidal category, then the functor category $\Fun(\cC, \m{E})$ is endowed with a symmetric monoidal structure, with tensor product given by the Day convolution.
Given functors $F,G\colon \cC\rightarrow \m{E}$, their Day convolution product $F\Day G\colon \cC\rightarrow \m{E}$ is defined as the coend in $\m{E}$:
\[
(F\Day G)(c)= \int^{(c_1, c_2)\in \cC^{\times 2}} \mu(c_1, c_2; c)\odot F(c_1)\otimes G(c_2)
\]
for all $c\in \cC$.
Here $\mu\colon \cC\times \cC \nrightarrow \cC$ is the promonoidal product in $\cC$, while $\otimes\colon \m{E}\times \m{E}\rightarrow \m{E}$ is the monoidal product in $\m{E}$, and $\odot\colon \Set\times \m{E}\rightarrow \m{E}$ denotes the tensoring of $\m{E}$ over $\Set$.
If $\eta\colon \{* \}\nrightarrow \cC$ is the promonoidal unit of $\m{E}$, and $\unit_\m{E}$ is the monoidal unit of $\m{E}$, then the monoidal unit of the Day convolution is the functor $\unit\colon \cC\rightarrow \m{E}$ defined as
\(
\unit(c)= \eta^\op(c)\odot \unit_\m{E}.
\)
If $\m{E}$ is moreover closed and complete, with internal hom given by $[-,-]\colon \m{E}^\op\times \m{E}\rightarrow \m{E}$, then $\Fun(\cC, \m{E})$ is also closed with internal hom given by the end in $\m{E}$:
\[
\{F, G\}(c) = \int_{(c', c'')\in \cC^\op\times \cC}\left[ \mu(c,c';c'')\odot F(c'),\, G(c'')\right]
\]
for all $c\in \cC$ and all functors $F,G\colon \cC\rightarrow \m{E}$.
\end{construction}

\begin{remark}\label{remark: Day convolution in ordinaty symmetric monoidal case}
    In \cref{construction: Day convolution in ordinary categories}, if $\cC$ is a symmetric monoidal category, with tensor product $\wedge\colon \cC\times \cC\rightarrow \cC$, instead of a promonoidal category, then the Day convolution of functors $F,G\colon \cC\rightarrow \m{E}$ is given by:
    \begin{align*}
          (F\Day G)(c) &\cong  \int^{(c_1, c_2)\in \cC^{\times 2}} \cC(c_1\wedge c_2, c)\odot F(c_1)\otimes G(c_2)
           \cong \colim_{c_1\wedge c_2\rightarrow c} F(c_1)\otimes G(c_2)
    \end{align*}
   for all $c\in \cC$.
   The monoidal unit $\unit\colon \cC\rightarrow \m{E}$ is then the functor
    \(
    c\longmapsto \coprod_{\cC(\unit_\cC,\, c)} \unit_\m{E}
    \)
    if $\unit_\cC\in \cC$ is the monoidal unit, while the internal hom will be given by:
    \[
    \{F, G\}(c)=\int_{c'\in \cC}[F(c'), G(c\wedge c')]
    \]
    for all $c\in \cC$. See more details in \cite[\S 4]{day-promonoidal}.
\end{remark}

\subsection{Constructing the Day convolution operad}\label{subsection: constructing Day convolution operad}

\begin{definition}
Let $\m{O}^{\otimes}$ be an \operad, and let $p \colon \m{C}^\otimes \rightarrow \m{O}^{\otimes}$ be an $\m{O}$-promonoidal \category. 
By the discussion in \cref{subsection: flat inner fibrations and symmetric promonoidal categories},  $p_{\r{act}}$ is an exponentiable fibration, and the pullback functor 
\[
p^{\ast} \colon (\mathrm{Op}_{\infty})_{/\m{O}^{\otimes}} \rightarrow (\mathrm{Op}_{\infty})_{/\m{C}^{\otimes}},
\]
admits a right adjoint, $p_{\ast}$, called the \textit{norm along $p$}.
The Day convolution functor is 
\[
\Fun_{\m{O}}(\m{C},-)^{\otimes} = p_{\ast} p^{\ast} \colon (\mathrm{Op}_{\infty})_{/\m{O}^{\otimes}} \rightarrow (\mathrm{Op}_{\infty})_{/\m{O}^{\otimes}}.
\]
Note that this is the right adjoint of the functor $p_{!} p^{\ast}$, where $p_{!}$ is given by restriction along $p$.
\end{definition}

\begin{remark}
For each $x \in \m{O}^{\otimes}$, there is a natural equivalence 
\(
\Fun_{\m{O}}(\m{C},\m{E})^{\otimes}_{x} \simeq \Fun(\m{C}_{x}^{\otimes},\m{E}_{x}^{\otimes}), 
\)
but it is not the case that $\Fun_{\m{O}}(\m{C},\m{E})^{\otimes}_{\brak{1}}$ is equivalent to $\Fun_{/\m{O}}(\m{C},\m{E})$. 
Rather, it is the sections of the fibration
\(
\Fun_{\m{O}}(\m{C},\m{E})^{\otimes}_{\brak{1}} \rightarrow \m{O}
\)
which are equivalent to $\Fun_{/\m{O}}(\m{C},\m{E})$; for more details see \cite[2.2.6.5]{HA}. 
\end{remark}

\begin{remark}\label{remark: underlying coCartesian fibration of presentably O-promonoidal}
Let $p\colon \m{C}^{\otimes} \rightarrow \m{O}^{\otimes}$ be a $\m{O}$-promonoidal \category where $p_{\brak{1}}$ is a coCartesian fibration, and let $\m{E}^{\otimes} \rightarrow \Fin_\ast$ be a cocompletely symmetric monoidal \category, so $\m{E}_{\m{O}} = \m{O}\times\m{E}$ is a cocompletely $\m{O}$-monoidal \category. 
Then, by \cite[10.7]{shah2021parametrized} the fibration 
\(
\Fun_{\m{O}}(\m{C},\m{E}_{\m{O}})^{\otimes}_{\brak{1}} \rightarrow \m{O}
\)
is classified by the functor which sends $x$ to $\Fun(\m{C}_{x},\m{E})$ and $\phi$ to $\phi_{!}$, the functor of left Kan extension $\phi_{!} \colon \Fun(\m{C}_x, \m{E}) \rightarrow \Fun(\m{C}_{y},\m{E})$; see also \cite[Remark 3.23]{linskens-nardin-pol} for a discussion of this point. 
\end{remark}

\begin{remark}\label{remark: universal property of day convolution operad}
Let $p \colon \m{C}^{\otimes} \rightarrow \m{O}^{\otimes}$ be a promonoidal \category, and let $\cD^{\otimes} \rightarrow \m{O}^{\otimes}$ and $\m{E}^{\otimes} \rightarrow \m{O}^\otimes$ be maps of \operads. 
Because the formation of the Day convolution \operad is a right adjoint, we immediately obtain the following canonical equivalence of mapping spaces: 
\[
\Map_{(\r{Op}_{\infty})_{/\m{O^{\otimes}}}}(\m{D}^{\otimes}\times_{\m{O}^{\otimes}} \m{C}^{\otimes},\m{E}^{\otimes}) \simeq \Map_{(\r{Op}_{\infty})_{/\m{O}^{\otimes}}}(\m{D}^{\otimes},\Fun_{\m{O}}(\m{C},\m{E})^{\otimes}). 
\]
As noted in \cite[Construction 3.27]{linskens-nardin-pol}, this equivalence can be upgraded to an equivalence of \categories 
\[
\Alg_{\m{D}^{\otimes}/\m{O}^{\otimes}}(\Fun_{\m{O}}(\m{C},\m{D})^{\otimes}) \simeq \Alg_{\m{D}^{\otimes}\times_{\m{O}^{\otimes}}\m{C}^{\otimes}}(\m{D}^{\otimes}). 
\] 
\end{remark}

\subsection{Existence of convolution products}\label{subsection: existence of convolution products}
To determine when the Day convolution \operad is an $\m{O}$-monoidal \category, we recall the following crucial lemma due to Linskens--Nardin--Pol \cite[Lemma 3.21]{linskens-nardin-pol}.

\begin{lemma}\label{equation: multimorphism space formula}
Let $\m{C}^{\otimes} \rightarrow \m{O}^{\otimes}$ be an $\m{O}$-promonoidal \category and $\m{D}^{\otimes} \rightarrow \m{O}^{\otimes}$ an \operad over $\m{O}^{\otimes}$.
For all active morphisms $\phi\colon\{x_i\} \rightarrow y$ in $\m{O}^{\otimes}$, and objects $\{F_i\} \in \Fun_{\m{O}}(\m{C},\m{D})^{\otimes}_{\{x_i\}}$ and $G \in \Fun_{\m{O}}(\m{C},\m{D})^{\otimes}_{y}$, the space of $\phi$-multimorphisms 
\(
\Mul^{\phi}_{\Fun_{\m{O}}(\m{C},\m{D})}(\{F_i\},G)
\)
can be canonically identified with the following end: 
\[
\int_{c'\in \m{C}^{\otimes}_{y}} \int_{\{c_i\} \in (\m{C}^{\otimes}_{\{x_i\}})^{\op}} \Map_{\Spaces}\big(\Mul^{\phi}_{\m{C}}(\{c_i\},c'), \Mul^{\phi}_{\m{D}}(\{F_{i}(c_i)\},G(c'))\big). 
\]
\end{lemma}

In light of the formula above, we can isolate conditions under which the Day convolution \operad is $\m{O}$-corepresentable. 
To this end, we introduce the following notation. 

\begin{definition}
    Let $p\colon \m{C}^{\otimes} \rightarrow \m{O}^{\otimes}$ be an $\m{O}$-promonoidal \category and let $\phi\colon \{x_i\} \rightarrow y$ be an active morphism in $\m{O}$. 
    Define
    \[
    \ev^{\phi} \colon \Ar^{\phi}(\m{C}) \rightarrow \m{C}^{\otimes}_{\{x_i\}} \times \m{C}^{\otimes}_{y}
    \]
    to be the bifibration induced by the following pullback square
    \[
    \begin{tikzcd}
    \Ar^{\phi}(\m{C}) \arrow[r] \arrow[d] & \Ar(\m{C}^{\otimes}_{\Act}) \arrow[d,"\Ar(p_{\Act})"] \\
    \Delta^{0} \arrow[r,"\phi"'] & \Ar(\m{O}^{\otimes}_{\Act}) 
    \end{tikzcd}
    \]
    and let $\ev^{\phi}_{i}$ denote $\ev^{\phi}$ followed by projection onto the $i$-th factor, where $i \in \{0,1\}$. 
    Furthermore, given $c \in \m{C}^{\otimes}_{y}$, we let $\Ar^{\phi}(\m{C})_{c} = \Ar^{\phi}(\m{C})\times_{\m{C}^{\otimes}_{y}} \{c\}$, which we note is classified by the functor 
    \[
    \Mul_{\m{C}}^{\phi}(-,c) \colon (\m{C}^{\otimes}_{\{x_i\}})^{\op} \rightarrow \Spaces. 
    \]
\end{definition}

The next lemma follows immediately by unraveling the definitions.  

\begin{lemma}\label{lemma: Day conv is O-corep}
    In the situation of \cref{equation: multimorphism space formula}, the functor 
    \[
    \Mul^{\phi}_{\Fun_{\m{O}}(\m{C},\m{D})}(\{F_i\},-) \colon \Fun(\m{C}_{y},\m{D}_{y}) \rightarrow \Spaces
    \]
    is corepresentable if and only if either of the following equivalent conditions hold:
    \begin{enumerate}
        \item The left Kan extension of
        \[
        \Ar^{\phi}(\m{C}) \xrightarrow{\{F_i\} \circ \ev_{0}^{\phi}} \m{D}_{\{x_i\}}^{\otimes} \xrightarrow{\otimes^{\phi}_{\m{D}}} \m{D}^{\otimes}_{y}
        \]
        along $\ev^{\phi}_{1} \colon \Ar^{\phi}(\m{C}) \rightarrow \m{C}^{\otimes}_{y}$ exists. 
        \item For all $c \in \m{C}^{\otimes}_{y}$, the following colimit exists
        \[
        \underset{\{c_i\} \in \Ar^{\phi}(\m{C})_{c}}{\colim} \bigotimes^{\phi}_{i \in I} \{F_i(c_i)\}.
        \]
    \end{enumerate}
\end{lemma}

We now turn the condition in the lemma above into a definition. 

\begin{definition}\label{definition: indexing kappa-Day conv prod}
    $\m{O}^{\otimes}$ be an \operad and let $\kappa$ be an infinite regular cardinal. 
    \begin{enumerate}
        \item If $p\colon \m{C}^{\otimes} \rightarrow \m{O}^{\otimes}$ is an $\m{O}$-promonoidal \category, we say $p$ is \textit{$\m{O}$-promonoidally $\kappa$-small} provided that: for each active morphism $\phi\colon \{x_i\} \rightarrow y$ in $\m{O}^{\otimes}$ and for each $c \in \m{C}^{\otimes}_{y}$, the \category $\Ar^{\phi}(\m{C})_{c}$ has an essentially $\kappa$-small cofinal subcategory. 
        In the special case where $\m{O}^{\otimes} = \Fin_\ast$, we say that $\m{C}^{\otimes}$ is \textit{symmetric promonoidally $\kappa$-small}.  
        \item If $q\colon \m{D}^{\otimes} \rightarrow \m{O}^{\otimes}$ is an $\m{O}$-monoidal \category, we say $q$ \textit{is compatible with $\kappa$-small colimits} provided that: for each $x \in \m{O}$, the \category $\m{D}^{\otimes}_{x}$ admits $\kappa$-small colimits and for each active morphism $\phi\colon \{x_i\} \rightarrow y$ in $\m{O}^{\otimes}$, the functor $\otimes^{\phi}_{\m{D}}$ preserves $\kappa$-small colimits separately in each variable.  
    \end{enumerate}
    In the special case where $\kappa = \omega$, we say that $p$ is $\m{O}$-promonoidally finite, and that $q$ is compatible with finite colimits.
\end{definition}

\begin{remark}\label{remark: tensor product arrow categories}
    Let $p\colon \m{C}^{\otimes} \rightarrow \m{O}^{\otimes}$ be a map of \operads and let $\sigma\colon [2] \rightarrow \m{O}^{\otimes}_{\Act}$ represent the commutative triangle
    \[
    \begin{tikzcd}
    & \{y_{j}\}_{j\in J} \arrow[dr,"\psi"] & \\
    \{x_i\}_{i\in I} \arrow[rr,"\xi"'] \arrow[ur,"\phi"] & & z
    \end{tikzcd}
    \]
    where $\phi$ is a lift of a morphism $\alpha \colon I \rightarrow J$ in $\Fin$.
    Let $\Ar^{\sigma}(\m{C}) = \Fun_{/\m{O}^{\otimes}_{\Act}}([2],\m{C}^{\otimes}_{\Act})$ and note that restriction along $\{i<j\} \subseteq [2]$ induces three evaluation functors 
    \[
    \begin{tikzcd}
    & \Ar^{\sigma}(\m{C}) \arrow[dl,"\ev_{\{0<1\}}"'] \arrow[d,"{\ev_{\{1<2\}}}" description] \arrow[dr,"\ev_{\{0<2\}}"] & \\
    \Ar^{\phi}(\m{C}) & \Ar^{\psi}(\m{C}) & \Ar^{\xi}(\m{C}) 
    \end{tikzcd}
    \]
    two of which participate in an equivalence of \categories 
    \[
    (\ev_{\{0<1\}},\ev_{\{1<2\}}) \colon \Ar^{\sigma}(\m{C}) \xrightarrow{\sim} \Ar^{\phi}(\m{C}) \times_{\m{C}^{\otimes}_{\{y_j\}}} \Ar^{\psi}(\m{C}).
    \]
    Furthermore, if $p$ is $\m{O}$-promonoidal with $\sigma$ as above, then by \cite[2.2.8(6)]{ayala-francis-fibrations} the functor $\ev_{\{0<2\}}$ is final, so that $(\ev_{\{0<2\}})_{!}(\ev_{\{0<2\}})^{\ast} \simeq \mathrm{id}$.
\end{remark}

\begin{proposition}\label{prop: existence of Day convolution product}
    Fix an infinite regular cardinal $\kappa$. 
    Let $p\colon \m{C}^{\otimes} \rightarrow \m{O}^{\otimes}$ be an $\m{O}$-promonoidally $\kappa$-small \category, and let $q\colon \m{D}^{\otimes} \rightarrow \m{O}^{\otimes}$ be an $\m{O}$-monoidal category that is compatible with $\kappa$-small colimits.     
    Then, $\pi \colon \Fun_{\m{O}}(\m{C},\m{D})^{\otimes} \rightarrow \m{O}^{\otimes}$ is an $\m{O}$-monoidal \category, compatible with $\kappa$-small colimits, whose convolution product is given by
    \[
    \left(\bigotimes^{\phi}_{i\in I}\{F_{i}\}_{i}\right)(c) \simeq \colim_{\{c_{i}\}_{i} \in \Ar^{\phi}(\m{C})_{c}} \bigotimes^{\phi}_{i\in I} \{F_{i}(c_{i})\}_{i}. 
    \]
    In the case where $\m{O}^{\otimes} = \Fin_{\ast}$, and $q$ is compatible with small colimits, we have 
    \[
    (F\otimes G)(c) \simeq \int^{\{c_1,c_2\} \in \m{C}^{\times 2}} \Mul_{\m{C}}(\{c_1,c_2\},c) \times F(c_1)\otimes F(c_2)
    \]
    and
    \[
    \unit_{\Fun(\m{C},\m{D})} \simeq \Mul_{\m{C}}(\varnothing,-) \times \unit_{\m{D}}. 
    \]
\end{proposition}

\begin{proof}
    By \cref{lem: corepresentable operad fibration is locally coCartesian}, it is enough to prove that $\pi$ is $\m{O}$-corepresentable and $\m{O}$-promonoidal. 
    As the hypotheses of \cref{lemma: Day conv is O-corep} are satisfied, in virtue of our assumptions, we find that $\pi$ is $\m{O}$-corepresentable and the claimed formula for the convolution product holds. 
    To see that $\pi_{\Act}$ is exponentiable, we will prove that $\pi_{\Act}$ has the property described in \cite[2.2.8(4)]{ayala-francis-fibrations}.
    To this end, $\phi \colon \{x_{i}\}_{i\in I} \rightarrow \{y_j\}_{j\in J}$ and $\psi \colon\{y_{j}\}_{j\in J} \rightarrow z$ denote active morphisms in $\m{O}^{\otimes}$, and let $\{F_{i}\}_{i} \in \Fun_{\m{O}}(\m{C},\m{D})^{\otimes}_{\{x_{i}\}}$.
    Furthermore, note that $\psi$ is given by a tuple of active morphisms $\{\psi_{i}\colon \{x_{i}\}_{i \in \alpha^{-1}\{j\}} \rightarrow y_{j}\}_{j\in J}$ and determines a map of finite sets $\alpha\colon I \rightarrow J$.
    By the Yoneda lemma combined with the $\m{O}$-corepresentability of $\pi$, we only need to verify that the canonical transformation    
    \[
    \widetilde{\sigma} \colon \bigotimes^{\xi}_{i \in I} \{F_{i}\} \longrightarrow \bigotimes^{\psi}_{j \in J}\bigg\{ \bigotimes^{\phi_{j}}_{i \in \alpha^{-1}\{j\}} \{F_{i}\} \bigg\}_{j} 
    \]
    is an equivalence of functors $\m{C}^{\otimes}_{z} \rightarrow \m{D}^{\otimes}_{z}$. 
    Because $\m{C}^{\otimes}$ is $\m{O}$-promonoidally $\kappa$-small and $\m{D}^{\otimes}$ is $\kappa$-cocompletely $\m{O}$-monoidal, we find the target of $\widetilde{\sigma}$ is equivalent to
    \[
    (\ev_{1}^{\psi})_{!}(\ev_{0}^{\psi})^{\ast}(\ev_{1}^{\phi})_{!}(\ev_{0}^{\phi})^{\ast} \left(\otimes^{\xi}_{\m{D}} \circ \{F_i\}\right),
    \]
    after unwinding the definitions.
    By the discussion in \cref{remark: tensor product arrow categories}, combined with \cref{proposition: exchange transformation} we have 
    \begin{align*}
    (\ev_{1}^{\psi})_{!}(\ev_{0}^{\psi})^{\ast}(\ev_{1}^{\phi})_{!}(\ev_{0}^{\phi})^{\ast} & \simeq (\ev_{1}^{\psi})_{!} (\ev_{\{1<2\}})_{!} (\ev_{\{0<1\}})^{\ast}(\ev_{0}^{\phi})^{\ast}\\
    & \simeq (\ev_{1}^{\xi})_{!} (\ev_{\{0<2\}})_{!} (\ev_{\{0<2\}})^{\ast}(\ev_{0}^{\xi})^{\ast} \\
    & \simeq (\ev_{1}^{\xi})_{!}(\ev_{0}^{\xi})^{\ast}.
    \end{align*}
    As the source of $\widetilde{\sigma}$ is given by $(\ev_{1}^{\xi})_{!}(\ev_{0}^{\xi})^{\ast}\left(\otimes^{\xi}_{\m{D}} \circ \{F_i\}\right)$, the claim follows
    The explicit description in the case where $\m{O}^{\otimes} = \Fin_\ast$ follows directly from the definitions or from \cite[3.29]{linskens-nardin-pol}
\end{proof}

\subsection{Functoriality of the Day convolution operad}\label{subsection: functoriality of convolution operad} 
We now address the extent to which the construction 
\[
(\m{C}^{\otimes} \xrightarrow{p} \m{O}^{\otimes}, \m{E}^{\otimes} \xrightarrow{q} \m{O}^{\otimes}) \mapsto \Fun_{\m{O}}(\m{C},\m{E})^{\otimes} 
\]
is natural in $p\colon \m{C}^{\otimes} \rightarrow \m{O}^{\otimes}$ and $q: \m{E}^{\otimes} \rightarrow \m{O}^{\otimes}$.
Any map of \operads $f\colon \m{E}^{\otimes} \rightarrow \m{F}^{\otimes}$ over $\m{O}^{\otimes}$ induces a map of Day convolution \operads over $\m{O}^{\otimes}$ which we denote by 
\[
f_{\ast} \colon \Fun_{\m{O}}(\m{C},\m{E})^{\otimes} \rightarrow \Fun_{\m{O}}(\m{C},\m{F})^{\otimes}.
\]
The following lemma is standard, and follows almost immediately from \cref{prop: existence of Day convolution product}.

\begin{lemma}\label{lem: Dau convolution preserves maps of operads}
    Let $p\colon \m{C}^{\otimes} \rightarrow \m{O}^\otimes$ be an $\m{O}$-promonoidally $\kappa$-small \category and let $\m{E}^\otimes$ and $\m{F}^{\otimes}$ be $\kappa$-cocomplete $\m{O}$-monoidal \categories. 
    Then, for $f \colon \m{E}^\otimes \rightarrow \m{F}^\otimes$ a map of \operads over $\m{O}^{\otimes}$, the induced functor 
    \(
    f_{\ast} \colon \Fun_{\m{O}}(\m{C},\m{E})^{\otimes}\rightarrow \Fun_{\m{O}}(\m{C},\m{F})^{\otimes}
    \)
    is a map of \operads over $\m{O}^{\otimes}$. 
    If, additionally, $f$ is $\m{O}$-monoidal, and, for each $x \in \m{O}$, the functor $f_{x}\colon \m{E}_{x} \rightarrow \m{F}_{x}$ preserves $\kappa$-small colimits, then $f_{\ast}$ is $\m{O}$-monoidal.
\end{lemma}

We recall the following construction from \cite[Construction 3.27]{linskens-nardin-pol}. 

\begin{construction}\label{construction: day convolution functoriality in the source}
    Let $p\colon \m{C}^{\otimes} \rightarrow \m{O}^\otimes$ and $q\colon \m{D}^{\otimes} \rightarrow \m{O}^{\otimes}$ be $\m{O}$-promonoidal \categories and let $f\colon \m{C}^{\otimes} \rightarrow \m{D}^\otimes$ be a map of \operads over $\m{O}^{\otimes}$. 
    We let $f^{\ast}$ denote the natural transformation
    \[
    \Fun_{\m{O}}(\m{D},-)^{\otimes} = q_{\ast}q^{\ast} \rightarrow p_{\ast}p^\ast = \Fun_{\m{O}}(\m{C},-)^{\otimes} 
    \]
    which is adjoint to the composition
    \[
    p_{!}p^{\ast}q_{\ast}q^{\ast} \simeq p_{!}f^{\ast}q^{\ast}q_{\ast} q^{\ast} \rightarrow p_{!} f^{\ast}q^{\ast} \simeq p_{!}p^{\ast} \rightarrow \id.
    \]
    Moreover, given a third $\m{O}$-promonoidal \category $r \colon \m{E}^{\otimes} \rightarrow \m{O}^\otimes$ and a map $g \colon \m{D}^{\otimes} \rightarrow \m{E}^{\otimes}$ of \operads over $\m{O}^{\otimes}$, from the definition above, it is not difficult to verify that we have a natural equivalence 
    \(
    (gf)^{\ast} \simeq f^{\ast}g^{\ast}. 
    \)
\end{construction}

It well-known that left Kan extension along a symmetric monoidal functor is itself a symmetric monoidal construction. 
By work of Day and Street, the same observation is true when left Kan extending along symmetric promonoidal functors, \cite[1 Proposition]{day-street}. 
We prove a version of this fact for $\m{O}$-promonoidal \categories, the cocomplete antecedent of which was established by Linskens--Nardin--Pol \cite[3.34]{linskens-nardin-pol}. 
To accomplish this, we need a fibrational interpretation of when a lax $\m{O}$-promonoidal functor is $\m{O}$-promonoidal. 

\begin{remark}
    Let $f\colon \m{C}^{\otimes} \rightarrow \m{D}^{\otimes}$ be a lax $\m{O}$-promonoidal functor. 
    For any active morphism $\phi\colon \{x_i\} \rightarrow y$ in $\m{O}^{\otimes}$, we obtain a commutative square of \categories 
    \[
    \begin{tikzcd}
    \Ar^{\phi}(\m{C}) \arrow[d,"\ev^{\phi}_{\m{C}}"'] \arrow[rr,"f_{\phi}"] & &\Ar^{\phi}(\m{D}) \arrow[d,"\ev^{\phi}_{\m{D}}"] \\
    \m{C}^{\otimes}_{\{x_i\}} \times \m{C}^{\otimes}_{y} \arrow[rr,"{(f_{\{x_i\}},f_{y})}"'] & & \m{D}^{\otimes}_{\{x_i\}} \times \m{D}^{\otimes}_{y} 
    \end{tikzcd}
    \]
    in which the vertical arrows are bifibrations. 
    This commutative square induces a functor of bifibrations
    \[
    \begin{tikzcd}
    (\id,f_{y})_{!}\Ar^{\phi}(\m{C}) \arrow[rr] \arrow[dr] & & (f_{\{x_i\}},\id)^{\ast}\Ar^{\phi}(\m{D}) \arrow[dl] \\
    & \m{C}_{\{x_i\}}^{\otimes} \times \m{D}^{\otimes}_{y} &    
    \end{tikzcd}
    \]
    which is an equivalence precisely when $f$ is $\m{O}$-promonoidal.   
\end{remark}

\begin{proposition}\label{prop: left kan extension along strong promonoidal is strong monoidal}
    Let $q\colon \m{E}^{\otimes} \rightarrow \m{O}^{\otimes}$ be an $\m{O}$-monoidal \category which is compatible with $\kappa$-small colimits and let $f\colon \m{C}^{\otimes} \rightarrow \m{D}^{\otimes}$ be a lax $\m{O}$-promonoidal functor between $\m{O}$-promonoidally $\kappa$-small \categories. 
    Then, precomposition by $f$ induces a lax $\m{O}$-monoidal functor
    \[
    f^{\ast}\colon\Fun_{\m{O}}(\m{D},\m{E})^{\otimes} \rightarrow \Fun_{\m{O}}(\m{C},\m{E})^{\otimes}.
    \]    
    \noindent If we additionally assume that for each $x \in \m{O} = \m{O}^{\otimes}_{\brak{1}}$, the \categories $\m{C}_{x}^{\otimes}$ and $\m{D}_{x}^{\otimes}$ are $\kappa$-small and $f$ is an $\m{O}$-promonoidal functor, then  $f^{\ast}$ admits an $\m{O}$-monoidal operadic left adjoint 
    \[
    f_{!} \colon \Fun_{\m{O}}(\m{C},\m{E})^{\otimes} \rightarrow \Fun_{\m{O}}(\m{D},\m{E})^{\otimes}
    \]    
    given by left Kan extension. 
\end{proposition}

\begin{proof}
    The existence of the lax $\m{O}$-monoidal functor $f^{\ast}$ follows from \cref{construction: day convolution functoriality in the source}. 
    To produce $f_{!}$, it is enough to verify that $f^\ast$ satisfies the hypotheses of \cite[Corollary 7.3.2.12]{HA}. 
    Our first assumption guarantees that for all $\{x_i\} \in \m{O}^{\otimes}$, the induced map 
    \[
    \Fun_{\m{O}}(\m{C},\m{E})^{\otimes}_{\{x_i\}} \rightarrow \Fun_{\m{O}}(\m{D},\m{E})^{\otimes}_{\{x_i\}}
    \]
    admits a left adjoint given by left Kan extension. 
    Fixing an active map, $\phi \colon \{x_{i}\} \rightarrow y$ and an object $\{F_{i}\} \in \Fun_{\m{O}}(\m{C},\m{E})^{\otimes}_{\{x_{i}\}}$, it is enough to show that the oplax transformation
    \begin{equation}\label{equation: oplax transformation}
    (f_{y})_{!}\left(\bigotimes^{\phi}_{i\in I}\{F_{i}\}_{i \in I} \right) \rightarrow \bigotimes_{i\in I}^{\phi}\{(f_{x_{i}})_{!}F_{i}\}_{i\in I},
    \end{equation}
    is an equivalence.
    By definition, the source and target in \cref{equation: oplax transformation} are given by 
    \[
    (f_{y})_!({\ev_{\m{C},1}^{\phi}})_{!}({\ev_{\m{C},0}^{\phi}})^{\ast} \left(\otimes^{\phi}_{\m{E}} \circ \{F_i\} \right) \ \ \text{and} \ \ ({\ev_{\m{D},1}^{\phi}})_{!}({\ev_{\m{D},0}^{\phi}})^{\ast}\left( \otimes_{\m{E}}^{\phi} \circ \{(f_{x_i})_{!}F_{i}\} \right), 
    \]
    respectively, where we have decorated $\ev^{\phi}_{i}$ with a $\m{C}$ or $\m{D}$ to denote the relevant evaluation map. 
    Furthermore, we have a canonical equivalence 
    \[
    ({\ev_{\m{D},1}^{\phi}})_{!}({\ev_{\m{D},0}^{\phi}})^{\ast}\left( \otimes_{\m{E}}^{\phi} \circ \{(f_{x_i})_{!}F_{i}\} \right)\simeq ({\ev_{\m{D},1}^{\phi}})_{!}({\ev_{\m{D},0}^{\phi}})^{\ast}(f_{\{x_i\}})_{!}\left(\otimes_{\m{E}}^{\phi}\circ \{F_{x_i}\}\right)
    \]
    because we assumed $\m{E}^{\otimes}$ is compatible with $\kappa$-small colimits. 
    Next, as $f$ is an $\m{O}$-promonoidal functor, we have a commutative diagram
    \[\begin{tikzcd}
	{\m{C}_{y}^{\otimes}} & {\m{D}^{\otimes}_{y}} & {\m{D}^{\otimes}_{y}} \\
	{\Ar^{\phi}(\m{C})} & {(f_{\{x_i\}},\id)^{\ast}\Ar^{\phi}(\m{D})} & {\Ar^{\phi}(\m{D})} \\
	{\m{C}^{\otimes}_{\{x_i\}}} & {\m{C}^{\otimes}_{\{x_i\}}} & {\m{D}^{\otimes}_{\{x_i\}}}
	\arrow["{f_{y}}", from=1-1, to=1-2]
	\arrow["\id", from=1-2, to=1-3]
	\arrow["{\ev_{\m{C},1}^{\phi}}", from=2-1, to=1-1]
	\arrow["\eta", from=2-1, to=2-2]
	\arrow["{\ev_{\m{C},0}^{\phi}}"', from=2-1, to=3-1]
	\arrow["{\widetilde{\ev}_{1}^{\phi}}", from=2-2, to=1-2]
	\arrow["{\tilde{f}_{\phi}}", from=2-2, to=2-3]
	\arrow["{\widetilde{\ev}_{0}^{\phi}}"', from=2-2, to=3-2]
	\arrow["{\ev_{\m{D},1}^{\phi}}"', from=2-3, to=1-3]
	\arrow["{\ev_{\m{D},0}^{\phi}}", from=2-3, to=3-3]
	\arrow["\id"', from=3-1, to=3-2]
	\arrow["{f_{\{x_i\}}}"', from=3-2, to=3-3]
    \end{tikzcd}\]
    in which the bottom right square is a pullback of Cartesian fibrations. 
    By \cref{proposition: exchange transformation} plus the commutativity of the diagram above, we have: 
    \begin{align*}
    ({\ev_{\m{D},1}^{\phi}})_{!}({\ev_{\m{D},0}^{\phi}})^{\ast}(f_{\{x_i\}})_{!}\left(\otimes_{\m{E}}^{\phi}\circ \{F_{x_i}\}\right) & \simeq ({\ev_{\m{D},1}^{\phi}})_{!} (\tilde{f}_{\phi})_{!} ({\widetilde{\ev}}_{0}^{\phi})^{\ast}\left(\otimes_{\m{E}}^{\phi}\circ \{F_{x_i}\}\right) \\
    & \simeq ({\widetilde{\ev}}_{1}^{\phi})_{!}({\widetilde{\ev}}_{0}^{\phi})^{\ast}\left(\otimes_{\m{E}}^{\phi}\circ \{F_{x_i}\}\right).
    \end{align*}
    The claim now follows from the assertion that $(\widetilde{\ev}_{0}^{\phi})^\ast \simeq \eta_{!} (\ev^{\phi}_{\m{C},0})^{\ast}$, which can be easily deduced from the equivalence
    \[
    (\id,f_{y})_{!}\Ar^{\phi}(\m{C}) \simeq (f_{\{x_i\}},\id)^{\ast}\Ar^{\phi}(\m{D})
    \]
    as we have only extended over the second variable. 
\end{proof}

\begin{example}
Suppose $\cC$ and $\m{E}$ are symmetric monoidal \categories, suppose $\cC$ is small and that $\m{E}$ admits colimits indexed by $\cC$.
Then \cref{prop: left kan extension along strong promonoidal is strong monoidal} in particular shows that if $F\colon \cC\to \m{E}$ is a lax $\m{O}$-monoidal functor, then $\colim_\cC F$ is an $\m{O}$-algebra in $\m{E}$.
The analogue of this claim is false for arbitrary symmetric promonoidal \categories, as the functor $\m{O}^{\otimes} \rightarrow \Fin_\ast$ is not always symmetric promonoidal. 
\end{example}

\begin{lemma}\label{lemma: sections and O-monoidal structures}
    Let $\m{O}^{\otimes}$ and $\m{E}^{\otimes}$ be symmetric monoidal \categories, and let $\m{C}^{\otimes}$ be an $\m{O}$-promonoidal \category. 
    \begin{enumerate}
    \item There is a canonical equivalence of \operads 
    \[
    \Fun(\m{C},\m{E})^{\otimes} \xrightarrow{\sim} \Fun(\m{O},\Fun_{\m{O}}(\m{C},\m{E}_{\m{O}}))^{\otimes}.   
    \]
    Furthermore, whenever $\m{O}^{\otimes}$ and $\m{C}^{\otimes}$ are symmetric promonoidally $\kappa$-small, $\m{C}^{\otimes}$ is $\m{O}$-promonoidally $\kappa$-small, and $\m{E}^{\otimes}$ is compatible with $\kappa$-small colimits, this is an equivalence of symmetric monoidal \categories. 
    \item Let $f\colon \m{C}^{\otimes} \rightarrow \m{D}^{\otimes}$ be an $\m{O}$-promonoidal functor between $\m{O}$-promonoidally $\kappa$-small \categories, which are also symmetric promonoidally $\kappa$-small, and assume that $\m{E}^{\otimes}$ is compatible with $\kappa$-small colimits. 
    Then, the $\m{O}$-monoidal functor 
    \[
    f_{!}\colon \Fun_{\m{O}}(\m{C},\m{E}_{\m{O}})^{\otimes}   \rightarrow  \Fun_{\m{O}}(\m{D},\m{E}_{\m{O}})^{\otimes} 
    \]
    preserves $\kappa$-small colimits, and there is a commutative diagram of symmetric monoidal functors:
    \[
    \begin{tikzcd}
    \Fun(\m{C},\m{E})^{\otimes} \arrow[r,"f_{!}"] \arrow[d,"\sim"'] & \Fun(\m{D},\m{E})^{\otimes}\arrow[d,"\sim"] \\
    \Fun(\m{O},\Fun_{\m{O}}(\m{C},\m{E}_{\m{O}}))^{\otimes}   \arrow[r,"(f_{!})_{\ast}"'] & \Fun(\m{O},\Fun_{\m{O}}(\m{D},\m{E}_{\m{O}}))^{\otimes}  
    \end{tikzcd}
    \]
    \end{enumerate}
\end{lemma}

\begin{proof}
     The canonical equivalence in claim (1) is induced by the map of \operads 
     \[
     \Fun(\m{C},\m{E})^{\otimes} \times_{\Fin_\ast} (\m{O}^{\otimes} \times_{\m{O}^{\otimes}} \m{C}^{\otimes})\simeq \Fun(\m{C},\m{E})^{\otimes} \times_{\Fin_\ast} \m{C}^{\otimes} \rightarrow \m{E}^{\otimes}. 
     \]
     That this is an equivalence of \operads follows from \cref{remark: universal property of day convolution operad}. 
     The only thing to check regarding the second point in claim (1) is that $\m{E}_{\m{O}}^{\otimes}$ is compatible with $\kappa$-small colimits, but this is also clear from the definition. 

     The functor $f_{!}$ is $\m{O}$-monoidal by \cref{prop: left kan extension along strong promonoidal is strong monoidal}, and preserves $\kappa$-small colimits as it is a relative left adjoint. 
     Because $\m{O}^{\otimes}$ is a symmetric monoidal category, $f_{!}\colon \Fun_{\m{O}}(\m{C},\m{E}_{\m{O}})^{\otimes}   \rightarrow  \Fun_{\m{O}}(\m{D},\m{E}_{\m{O}})^{\otimes}$ is a symmetric monoidal functor between symmetric monoidal \categories. 
     This implies that $(f_{!})_{\ast}$ is symmetric monoidal as claimed. 
     To complete the proof, by the uniqueness of relative left adjoints, it is enough to note that the maps $f$ and $f^{\ast} \colon \Fun_{\m{O}}(\m{D},\m{E}_{\m{O}})^{\otimes} \rightarrow \Fun_{\m{O}}(\m{C},\m{E}_{\m{O}})^{\otimes}$ are compatible with the pairing used to establish claim (1).  
\end{proof}

\section{Classification and localization}\label{section: promonoidal localization}
Given a symmetric promonoidal \category $p\colon \m{C}^{\otimes} \rightarrow \Fin_\ast$, it was shown in \cite[Theorem 3.37]{linskens-nardin-pol}, that $p$ determines, and is determined by, the presentably symmetric monoidal structure on $\Fun(\m{C},\Spaces)$ given by the convolution product. 
In this section, we first prove an $\m{O}$-promonoidal analogue of the result above. 
In particular, given an $\m{O}$-promonoidal \category, $p \colon \m{C}^{\otimes} \rightarrow \m{O}^{\otimes}$, whose underlying exponentiable fibration $\m{C} \rightarrow \m{O}$ is coCartesian, we show that $p$ determines, and is determined by, a presentably $\m{O}$-monoidal structure on $\uFun_{\m{O}}(\m{C},\Spaces_{\m{O}}) \rightarrow \m{O}$, which is the appropriate analogue of the \category $\Fun(\m{C},\Spaces)$, relative to $\m{O}$. 
After establishing this classification result, we use it to construct what we call \textit{$\m{O}$-promonoidal localizations}, analogous to the familiar notion of symmetric monoidal localization studied in \cite{hindk} and \cite{HA}. 

\subsection{Fibrational preliminaries} 
In order to classify $\m{O}$-promonoidal structures as indicated above, we need to know how some basic constructions in category theory work relative to an \category $\m{O}$. 
In particular, we now recall the relative analogues of functor categories, opposite categories, and the Yoneda embedding. 

\begin{construction}\label{construction: relative exponential/Cartesian workhorse}
Let $p\colon \cC \rightarrow \m{O}$ be a coCartesian fibration, let $\m{E}$ be an \category, and let $\pr_{1}\colon \m{E}_{\m{O}} = \m{O}\times \m{E} \rightarrow \m{O}$ denote the projection, which is both Cartesian and coCartesian.  
Because $p$ is coCartesian, it is exponentiable, thereby yielding an adjunction
\[
\begin{tikzcd}
    \Cat_{\infty/\oO} \ar[bend left=15, description]{r}{p^*} \ar[phantom, description, xshift=-0.4ex]{r}{\perp}& [3em]\ar[bend left=15]{l}{p_*}\Cat_{\infty/\cC}.
\end{tikzcd}
\]
We let 
\(
q\colon \uFun_{\m{O}}(\m{C},\m{E}_{\m{O}}) \rightarrow \m{O}
\)
denote the object $p_{\ast}p^{\ast}\m{E}_{\m{O}} \in \Cat_{\infty/\oO}$, which we say is the \textit{relative exponential} of $\m{E}_{\oO}$ along $p$. 
Unwinding the definitions, an arrow $\Delta^1 \rightarrow \uFun_{\oO}(\cC, \m{E})$ corresponds to a functor $\eta\colon \m{C}_{\phi} \rightarrow \m{E}_{\m{O}}$, where $p_{\phi} \colon \m{C}_{\phi} = \Delta^{1} \times_{\oO}\cC \rightarrow \Delta^1$ is the coCartesian fibration classified by the functor $\phi_{!} \colon \cC_x \rightarrow \cC_y$. 
The arrow $\eta$ being Cartesian or coCartesian is captured by the behavior of the functor
\[
\m{C}_{\phi} \xrightarrow{\eta} \m{E}_{\m{O}} \xrightarrow{\pr_{2}}\m{E}. 
\]
We have that:
\begin{enumerate} 
\item the arrow $\eta$ is $q$-Cartesian, provided $\pr_{2} \circ \eta$ carries carries $p_{\phi}$-coCartesian arrows to equivalences; 
\item the arrow $\eta$ is $q$-coCartesian, provided $\pr_{2} \circ \eta$ is the left Kan extension of its restriction to the subcategory $\m{C}_{x} \subseteq \m{C}_{\phi}$. 
\end{enumerate}
It is always the case that $q$ is a Cartesian fibration, and if $\m{E}$ admits small colimits, then $q$ is a coCartesian fibration; see \cite[8.6.5.9]{kerodon} and \cite[3.23]{linskens-nardin-pol} for more details. 
\end{construction}

\begin{construction}\label{construction: fiberwise yoneda and map}
If the coCartesian fibration $p \colon \m{C} \rightarrow \m{O}$ is classified by a functor 
\(
F_{p} \colon \m{O} \rightarrow \Cat_\infty, 
\)
then the functor 
\[
(-)^{\op} \circ F_{p} \colon \m{O} \rightarrow \Cat_\infty \xrightarrow{(-)^{\op}} \Cat_\infty, 
\]
classifies a \textit{coCartesian dual fibration}, $p^{\vee} \colon \m{C}^{\vee} \rightarrow \m{O}$,  the construction of which is originally due to Barwick--Glasman--Shah \cite{barwick-glasman-shah-dualizing}. 
We have opted to follow Lurie's treatment of dual coCartesian fibrations from \cite[Section 8.6]{kerodon}, which we now briefly summarize. 

Let $\m{C}^{\vee}$ denote the full subcategory of $\uFun_{\oO}(\cC,\Spaces_{\oO})$ spanned by those objects $(x,F_x)$, where $F_x\colon \cC_x \rightarrow \Spaces$ is a corepresentable functor. 
We denote the inclusion of subcategories by
\[
\cC^{\vee} \xrightarrow{h} \uFun_{\oO}(\cC,\Spaces_{\oO}), 
\]
which is a \textit{contravariant Yoneda embedding}, and it is not hard to see that $h$ is a map of coCartesian fibrations. 
By the universal property of the relative exponential, the identity map of $p_{\ast}p^{\ast}\Spaces_{\oO}$ in $\Cat_{\infty/\m{O}}$ induces a morphism
\[
{\rm ev} \colon \uFun_{\oO}(\cC,\Spaces_{\oO})\times_{\oO}\cC \rightarrow \m{C}\times_{\m{O}} \Spaces_{\oO} \rightarrow \Spaces_{\oO}
\]
in $\Cat_{\infty/\oO}$. 
By precomposing with the inclusion of $\m{C}^{\vee}$, we obtain the \textit{fiberwise mapping space functor}
\begin{equation}\label{equation: fiberwise mapping space}
\Map_{\cC / \oO} \colon \m{C}^{\vee} \times_{\oO}\cC \xrightarrow{(h,{\rm id})} \uFun_{\oO}(\cC,\Spaces_{\oO}) \times_{\oO} \cC \xrightarrow{{\rm ev}} \Spaces_{\oO}.
\end{equation}
This terminology is justified by the fact that over each $x \in \oO$, we have an identification $(\Map_{\cC /\oO})_{x}\simeq \Map_{\cC_{x}} \colon \cC_{x}^{\op} \times \cC_{x} \rightarrow \Spaces$. 
The functor $\Map_{\cC/\oO}$ exhibits $p^{\vee} \colon \m{C}^{\vee} \rightarrow \oO$ as the coCartesian dual of $p$, in the sense of \cite[8.6.4.1]{kerodon}, thereby characterizing uniquely $p^{\vee}$ up to equivalence. 
Furthermore, the inclusion $h$ exhibits the target as the fiberwise cocompletion of the source, by \cite[8.7.2.4]{kerodon}. 
By appropriately handling size issues, as discussed in the introduction, we can form the coCartesian dual of $h$, which gives a morphism of coCartesian fibrations over $\m{O}$:
\[
\m{C} \simeq \m{C}^{\vee \vee} \xrightarrow{h^{\vee}} \uFun_{\oO}(\cC,\Spaces_{\oO})^{\vee}. 
\]
\end{construction}

\subsection{Classifying promonoidal structures}\label{subsection: classifying O-promonoidal structures}
Fibrational preliminaries dispensed with, we now turn to the aforementioned $\m{O}$-promonoidal classification result. 
We first show how presentably $\m{O}$-monoidal structures on fiberwise presheaf categories determine $\m{O}$-promonoidal structures, and then prove these $\m{O}$-promonoidal structures recover the original presentably $\m{O}$-monoidal structures via Day convolution.
For our arguments below, it will be convenient to introduce the following definition.

\begin{definition}
    Let $\m{O}^\otimes$ be an \operad and let $p \colon \m{C} \rightarrow \m{O}$ be a functor of \categories which is a coCartesian fibration. 
    We define an \textit{$\m{O}$-promonoidal structure extending $p$} to be a morphism of \operads $p^{\otimes}\colon \m{C}^{\otimes} \rightarrow \m{O}^{\otimes}$ such that:
    \begin{enumerate}
        \item the morphism of \operads $\m{C}^{\otimes} \rightarrow \m{O}^{\otimes}$ exhibits $\m{C}^{\otimes}$ as $\m{O}$-promonoidal;
        \item the induced exponentiable fibration $p^{\otimes}_{\brak{1}} \colon \m{C}^{\otimes}_{\brak{1}} \rightarrow \m{O}^{\otimes}_{\brak{1}}$ is equivalent to the coCartesian fibration $p$. 
    \end{enumerate}
    When $p^{\otimes}$ is a coCartesian fibration, we say that $p^{\otimes}$ is an \textit{$\m{O}$-monoidal structure extending $p$}. 
\end{definition}

\begin{example}
    The \category of spaces can be viewed as a symmetric monoidal \category via the coproduct or product, which we denote by $\Spaces^{\sqcup}$ and $\Spaces^{\times}$, respectively. 
    Both of these symmetric monoidal \categories are symmetric monoidal structures extending the unique coCartesian fibration $\Spaces \rightarrow \ast$. 
\end{example}

\begin{example}
    Let $\mathbb{CP}^{\ast} \colon \nn \rightarrow \Spaces$ denote the filtered space 
    \[
    \mathbb{CP}^{0} \rightarrow \mathbb{CP}^{1} \rightarrow \mathbb{CP}^{2} \rightarrow \cdots 
    \]
    which classifies a left fibration we denote by $p\colon \nn\ltimes \mathbb{CP}^{\ast} \rightarrow \nn$. 
    By the results of \cite[5.3]{lurie-rotation}, the filtered space above is the two-fold (filtered) bar construction of the filtered commutative monoid
    \[
    \{0\} \subseteq \{0,1\} \subseteq \{0,1,2\} \subseteq \cdots \subseteq \nn
    \]
    which we denote by $\nn_{\leq \ast}$. 
    Because the bar construction always preserves $\ee{\infty}$-algebras, there is an extension of $\mathbb{CP}^{\ast}$ to a lax symmetric monoidal functor $\mathbb{CP}^{\ast,\otimes} \colon \nn^{\otimes} \rightarrow \Spaces^{\times}$, which encodes a filtered commutative algebra in $\Spaces$; for a precise definition of $\nn^{\otimes}$, see the beginning of section \cref{section: EZ map on skeleton}.
    This is classified by a left fibration $p^{\otimes}\colon (\nn \ltimes \mathbb{CP}^{\ast})^{\otimes} \rightarrow \nn^{\otimes}$, which makes $p^{\otimes}$ an $\nn$-monoidal structure extending $p$.
\end{example}

\begin{proposition}\label{proposition: closed O-monoidal structures are O-promonoidal}
    Let $\m{O}^{\otimes}$ be an \operad, let $p\colon \m{C} \rightarrow \m{O}$ a coCartesian fibration, let $q \colon \uFun_{\m{O}}(\m{C},\Spaces_{\m{O}}) \rightarrow \m{O}$ be the relative exponential coCartesian fibration, and let
    \[
    q^{\otimes} \colon \uFun_{\m{O}}(\m{C},\Spaces_{\m{O}})^{\otimes} \rightarrow \m{O}^{\otimes}
    \]
    be a presentably $\m{O}$-monoidal structure extending $q$.
    Define $\m{C}^{\otimes}$ to be the full suboperad of 
    \[
    \uFun_{\m{O}}(\m{C},\Spaces_{\m{O}})^{\otimes,\vee} \rightarrow \m{O}^{\otimes}
    \]
    spanned by the objects of $\m{C}$ and let
    \[
    (h^{\vee})^{\otimes}\colon \m{C}^{\otimes} \rightarrow \uFun_{\m{O}}(\m{C},\Spaces_{\m{O}})^{\otimes,\vee} 
    \]
    denote the inclusion of \operads over $\m{O}^{\otimes}$. 
    Then $\m{C}^{\otimes}$ is an $\m{O}$-promonoidal \category. 
\end{proposition}

\begin{proof}
    It is enough to verify \cite[Lemma 2.2.8(4)]{ayala-francis-fibrations}. 
    Fix a composable pair of active maps 
    \[
    \{x_i\}_{i\in I}\xrightarrow{\phi} \{y_{j}\}_{j\in J} \xrightarrow{\psi} z
    \]
    in $\m{O}^{\otimes}$, let $f \colon I \rightarrow J$ denote the image of $\phi$ in $\Fin$, and for each $j \in J$, write $\phi_{j}\colon \{x_i\}_{i\in f^{-1}\{j\}} \rightarrow y_{j}$ for the component of $\phi$ which lies over $f^{-1}\{j\} \rightarrow \{j\}$.  
    Now, let $\{c_i\} \in\m{C}^{\otimes}_{\{x_i\}}$, and consider the natural transformation
    \[
    \int^{\{c_j\} \in \m{C}^{\otimes}_{\{y_j\}}} \Mul_{\m{C}}^{\psi}(\{c_j\},-) \times \prod_{j \in J} \Mul^{\phi_{j}}_{\m{C}}(\{c_i\}_{i\in f^{-1}\{j\}},c_{j}) \rightarrow \Mul_{\m{C}}^{\psi\circ \phi}(\{c_i\},-)
    \]
    of functors $\m{C}_{y} \rightarrow \Spaces$, which we claim is an equivalence. 
    By definition, we can identify each of these multimorphism spaces with mapping spaces in $\uFun_{\m{O}}(\m{C},\Spaces_{\m{O}})$, so the above transformation is equivalent to
    \[
    \int^{\{c_j\} \in \m{C}^{\otimes}_{\{y_j\}}} \Map_{\Fun(\m{C}_z,\Spaces)}\left(h^{\vee}(-),\bigotimes_{j\in J}^{\psi}h^{\vee}(c_j) \right)\times \prod_{j \in J} \Map_{\Fun(\m{C}_z,\Spaces)}\left(h^{\vee}(c_j),\bigotimes^{\phi_j}_{i\in f^{-1}\{j\}}h^{\vee}(c_i) \right)
    \]
    and
    \[
    \Map_{\Fun(\m{C}_z,\Spaces)}\left(h^{\vee}(-),\bigotimes_{i\in I}^{\psi\circ \phi}h^{\vee}(c_i) \right) 
    \] 
    respectively. 
    By the Yoneda lemma, combined with the presentability of the $\m{O}$-monoidal structure on $\uFun_{\m{O}}(\m{C},\Spaces_{\m{O}})$, the transformation in question is an equivalence if and only if the induced map
    \[
    \bigotimes_{j\in J}^{\psi} \left(\int^{c_{j} \in \m{C}^{\otimes}_{y_{j}}} h^{\vee}(c_{j}) \times \Map_{\Fun(\m{C}_{y_j},\Spaces)}\left(h^{\vee}(c_j),\bigotimes_{i\in f^{-1}\{j\}}^{\phi_{j}}h^{\vee}(c_i)\right) \right) \rightarrow \bigotimes_{i\in I}^{\psi \circ \phi}h^{\vee}(c_i) 
    \]
    is an equivalence in $\Fun(\m{C}_z,\Spaces)$. 
    This follows from the Yoneda lemma for coends, together with the equivalence of functors
    \[
    \bigotimes_{i\in I}^{\psi\circ \phi} \simeq\bigotimes^{\psi}_{j\in J} \bigotimes^{\phi_j}_{i\in f^{-1}\{j\}} \colon \prod_{i\in I} \Fun(\m{C}_{x_i},\Spaces) \rightarrow \Fun(\m{C}_{z},\Spaces). \qedhere
    \]
\end{proof}

To investigate the Day convolution \operads indexed by the $\m{O}$-promonoidal structures produced in \cref{proposition: closed O-monoidal structures are O-promonoidal}, we need an $\m{O}$-monoidal analogue of the relative mapping space functor from \cref{construction: fiberwise yoneda and map}.

\begin{construction}\label{construction: O-monoidal mapping space functor}
Let $\m{C}^{\otimes}$ be an \operad and let $p\colon \m{C}^{\otimes} \rightarrow \m{O}^{\otimes}$ be an $\m{O}$-monoidal \category. 
By \cref{construction: fiberwise yoneda and map}, we have a fiberwise mapping space functor 
\[
\Map_{\m{C}^{\otimes}/\m{O}^{\otimes}} \colon \m{C}^{\otimes,\vee}\times_{\m{O}^{\otimes}} \m{C}^{\otimes} \longrightarrow \cS. 
\]
Given $I \in \Fin$, the inert maps induce an equivalence 
\[
\m{C}^{\otimes,\vee}\times_{\m{O}^{\otimes}} \m{C}^{\otimes}\times_{\Fin_\ast} I_+ \simeq \prod_{i\in I} \m{C}^{\vee} \times_{\m{O}} \m{C}
\]
so that an object over $I_+ \in \Fin_\ast$ can be identified with a set of tuples $\{(c_i,x_i,c_i')\}_{i\in I}$, where $c_i,c_i' \in \m{C}_{x_i}$. 
Subsequently, we have a natural map 
\[
\Map_{\m{C}^{\otimes}/\m{O}^{\otimes}}(\{(c_i,x_i,c_i')\}_{i\in I}) \rightarrow \prod_{i\in I} \Map_{\m{C}_{x_i}}(c_i,c_i')
\]
which is an equivalence, as the composition 
\[
\prod_{i\in I} \m{C}^{\op}_{x_i} \times \m{C}_{x_i} \rightarrow \m{C}^{\otimes,\vee}\times_{\m{O}^{\otimes}} \m{C}^{\otimes}\xrightarrow{\Map_{\m{C}^{\otimes}/\m{O}^{\otimes}}} \Spaces
\]
is equivalent to the mapping space functor for the \category $\prod_{i\in I} \m{C}_{x_i}$. 
Therefore, by \cite[2.4.1.7]{HA}, the functor $\Map_{\m{C}^{\otimes}/\m{O}^{\otimes}}$ (essentially) uniquely extends to a map of \operads 
\[
(\Map_{\m{C}^{\otimes}/\m{O}^{\otimes}})^{\otimes} \colon \m{C}^{\otimes,\vee}\times_{\m{O}^{\otimes}} \m{C}^{\otimes} \rightarrow \Spaces^{\times}.
\]
However, as the source is $\m{O}$-monoidal, we obtain a map of \operads over $\m{O}^{\otimes}$ denoted by
\[
\begin{tikzcd}
\m{C}^{\otimes,\vee}\times_{\m{O}^{\otimes}} \m{C}^{\otimes} \arrow[rr,"\Map_{\m{C}/\m{O}}^{\otimes}"] \arrow[dr] & & \Spaces_{\m{O}}^{\otimes} \arrow[dl] \\
& \m{O}^{\otimes} &
\end{tikzcd}
\]
where $\Spaces_{\m{O}}^{\otimes} = \m{O}^{\otimes} \times_{\Fin_\ast} \Spaces^{\times}$. 
We call the functor $\Map_{\m{C}/\m{O}}^{\otimes}$ the \textit{$\m{O}$-monoidal mapping space of $\m{C}^{\otimes}$}. 
\end{construction}

\begin{theorem}\label{proposition: Day conv recovers presentably O-monoidal str}
    Let $\m{O}^{\otimes}$ be an \operad, let $p\colon \m{C} \rightarrow \m{O}$ a coCartesian fibration, let $q \colon \uFun_{\m{O}}(\m{C},\Spaces_{\m{O}}) \rightarrow \m{O}$ be the relative exponential coCartesian fibration, let
    \(
    q^{\otimes} \colon \uFun_{\m{O}}(\m{C},\Spaces_{\m{O}})^{\otimes} \rightarrow \m{O}^{\otimes}
    \)
    be a presentably $\m{O}$-monoidal structure extending $q$, and let $\m{C}^{\otimes}$ be the $\m{O}$-promonoidal structure associated to $q^{\otimes}$ by \cref{proposition: closed O-monoidal structures are O-promonoidal}. 
    Then, there is a canonical equivalence of presentably $\m{O}$-monoidal \categories 
    \[
    \uFun_{\m{O}}(\m{C},\Spaces_{\m{O}})^{\otimes} \xrightarrow{\sim} \Fun_{\m{O}}(\m{C},\Spaces_{\m{O}})^{\otimes}
    \]
    which is the identity on underlying coCartesian fibrations.     
\end{theorem}

\begin{proof}
    We begin by constructing the desired functor. 
    Precomposing the $\m{O}$-monoidal mapping space for $\uFun_{\m{O}}(\m{C},\Spaces_{\m{O}})$ with the inclusion $\m{C}^{\otimes} \rightarrow \uFun_{\m{O}}(\m{C},\Spaces_{\m{O}})^\otimes$, we obtain a map of \operads over $\m{O}^\otimes$:
    \[
    \m{C}^{\otimes}\times_{\m{O}^{\otimes}}\uFun_{\m{O}}(\m{C},\Spaces_{\m{O}})^{\otimes} \rightarrow \uFun_{\m{O}}(\m{C},\Spaces_{\m{O}})^{\otimes,\vee}\times_{\m{O}^{\otimes}}\uFun_{\m{O}}(\m{C},\Spaces_{\m{O}})^{\otimes} \xrightarrow{\Map_{/\m{O}}^{\otimes}} \Spaces_{\m{O}}^{\otimes}.
    \]
    As $\m{C}^{\otimes}$ is promonoidal, the universal property of the Day convolution \operad guarantees a lax $\m{O}$-monoidal functor 
    \[
    \Phi^{\otimes}\colon \uFun_{\m{O}}(\m{C},\Spaces_{\m{O}})^{\otimes} \rightarrow  \Fun_{\m{O}}(\m{C},\Spaces_{\m{O}})^{\otimes}
    \]   
    which is the identity functor on underlying coCartesian fibrations by \cref{remark: underlying coCartesian fibration of presentably O-promonoidal}.
    Therefore, it remains to show that $\Phi^{\otimes}$ is an $\m{O}$-monoidal functor.
    To prove that the natural transformation 
    \[
    \bigotimes^{\phi}_{i\in I} \circ \ \Phi^{\otimes}_{\{x_i\}_{i\in I}} \rightarrow \Phi^{\otimes}_{y}\circ \bigotimes^{\phi}_{i\in I}
    \]
    is an equivalence for all active maps $\phi\colon \{x_i\}_{i \in I} \rightarrow y$, it is enough to check on the representable objects as both $\m{O}$-monoidal \categories are compatible with small colimits.
    This follows by combining the definition of the $\m{O}$-promonoidal structure on \category $\m{C}^{\otimes}$ together with \cref{prop: existence of Day convolution product} (or \cite[Proof of Corollary 3.29]{linskens-nardin-pol}).
\end{proof}

\begin{remark}
    It seems plausible to the authors that this classification result might hold in even greater generality, e.g., when the underlying fibration is not assumed coCartesian. 
\end{remark}

\subsection{Fiberwise localization, after Hinich}
In this subsection, we review some of the basic theory around fiberwise localization as studied by Hinich in \cite{hindk}, and later by Lurie in \cite{kerodon}.
We also fix some notation and terminology to be used later on. 

\begin{definition}
    Recall that there is a fully faithful functor $\Spaces \hookrightarrow \Cat_\infty$. 
    This inclusion admits a left adjoint called \textit{the groupoid completion functor}
    \[
    \gpd{-}\colon \Cat_\infty \rightarrow \Spaces, 
    \]
    and a right adjoint called \textit{the groupoid core functor}
    \[
    (-)^{\simeq} \colon \Cat_\infty \rightarrow \Spaces. 
    \]
    Given an \category $\m{C}$, there are canonical functors 
    \(
    \m{C} \rightarrow \gpd{\m{C}}
    \)
    and 
    \(
    \m{C}^{\simeq} \rightarrow \m{C}
    \)
    obtained from the unit and counit of the relevant adjunctions. 
    The latter of these exhibits $\m{C}^{\simeq}$ as the wide subcategory containing only invertible morphisms.
\end{definition}

\begin{remark}
    From the point of view of simplicial sets, the groupoid completion of an \category $\m{C}$ can be obtained from its fibrant replacement with respect to the Quillen model structure on simplicial sets. 
    Explicitly, the map of simplicial sets $\m{C} \rightarrow \mathrm{Sing}\lvert\m{C}\rvert$ exhibits the target as the groupoid completion of the source.  
\end{remark}

\begin{definition}
    Let $\m{C}$ be a small \category and let $W$ be a collection of edges in $\m{C}$. 
    We say that $W$ is \textit{saturated} if there is a functor $\m{C} \rightarrow \m{D}$ such that $W$ is the preimage of the collection of equivalences in $\m{C}$, thereby defining a subcategory $W \subseteq \m{C}$. 
    The \textit{Dwyer--Kan localization of $\m{C}$ with respect to $W$}, denoted by $\m{C}[W^{-1}]$, is the \category given by the pushout
    \[
    \begin{tikzcd}
    W \arrow[d] \arrow[r] & \gpd{W} \arrow[d] \\
    \m{C} \arrow[r,"L"'] & \m{C}[W^{-1}] 
    \end{tikzcd}
    \]
    in $\Cat_\infty$. 
    Unwinding the definitions, the functor $L\colon\m{C} \rightarrow \m{C}[W^{-1}]$ induces an equivalence of \categories 
    \[
    \Fun(\m{C}[W^{-1}],\m{E}) \xrightarrow{\sim} \Fun^{W}(\m{C},\m{E}) \subseteq \Fun(\m{C},\m{E}), 
    \]
    where $\Fun^{W}(\m{C},\m{E})$ denotes the full subcategory of functors $F\colon \m{C} \rightarrow \m{E}$ which carry morphisms in $W$ to equivalences.    
    We say that a functor $F$ with this property is \textit{$W$-local}.
\end{definition}

To fix notation for later, we state a well-known lemma that follows from the results of \cite[5.5.2]{HTT}. 

\begin{lemma}\label{lemma: local functors define a reflective subcategory}
    Let $\m{C}$ be a small \category, let $W \subseteq \m{C}$ be a saturated collection of morphisms, and let $\m{E}$ be a presentable \category. 
    Then, the functor 
    \[
    L^{\ast} \colon \Fun(\m{C}[W^{-1}],\m{E}) \rightarrow \Fun(\m{C},\m{E})
    \]
    is fully faithful and admits a left adjoint given by $L_{!}$. 
    Therefore, the composition $L^{\ast}\circ L_{!}$ defines a reflective localization of $\Fun(\m{C},\m{E})$. 

    Let $\widehat{W} \subseteq \Fun(\m{C},\Spaces)$ denote the strongly saturated class generated by the image of $W^{\op} \subseteq \m{C}^{\op} \rightarrow \Fun(\m{C},\Spaces)$, where the second map is the Yoneda embedding. 
    Then, the left adjoint $L_{!}$ induces an equivalence of \categories 
    \[
    \Fun(\m{C},\Spaces)[\widehat{W}^{-1}] \xrightarrow{\sim} \Fun(\m{C}[W^{-1}],\Spaces), 
    \]
    so $\Fun(\m{C}[W^{-1}],\Spaces)$ is an $\omega$-accessible localization in that the functor $L^{\ast}$ preserves filtered colimits. 
\end{lemma}

\begin{remark}
    The conclusion of the preceding lemma holds in many situations where $\m{E}$ is not presentable.
    Indeed, it is enough that $L^{\ast}$ admits a left adjoint, which will be the case whenever $\m{E}$ admits colimits of shape $\m{C}\times_{\m{C}[W^{-1}]} \m{C}[W^{-1}]_{/y}$ for all $y \in \m{C}$. 
    We will study such a situation at the end of this subsection.
\end{remark}

We now introduce some terminology and state the main technical result from \cite{hindk}. 

\begin{definition}
    Let $p\colon \m{C} \rightarrow \m{O}$ be a coCartesian fibration of \categories. 
    We say that $f\colon c \rightarrow d$ in $\m{C}$ is \textit{vertical}, if $p(f)$ is an equivalence in $\m{O}$, and \textit{horizontal} if $f$ is a coCartesian lift of the morphism $p(f)$.
    
    We denote the collection of all vertical edges in $X$ by $W_{p}$. 
    Because both $\m{C}$ and $\m{O}$ are \categories, then $W_{p}$ is a saturated collection of morphisms in $\m{C}$. 
\end{definition}

\begin{definition}
    Let $p\colon \m{C} \rightarrow \m{O}$ be a coCartesian fibration of \categories and let $W \subseteq \m{C}$ be a saturated collection of arrows in $\m{C}$. 
    For each $x \in \m{O}$ let $W_{x} = W\times_{\m{O}}\{x\}$.
    We say that $W$ is \textit{$\m{O}$-saturated} provided:
    \begin{enumerate}
        \item the morphisms in $W$ are vertical; 
        \item for each $\phi\colon x \rightarrow y$ in $\m{O}$, the functor $\phi_{!}\colon \m{C}_x \rightarrow \m{C}_{y}$ carries $W_{x}$ into $W_{y}$.
    \end{enumerate}    
\end{definition}

\begin{example}
    For $p$ a coCartesian fibration of \categories, the collection of all vertical arrows, $W_{p}$, is $\m{O}$-saturated. 
\end{example}

\begin{proposition}[{\cite[2.1.4]{hindk}}]\label{proposition: Hinich fiberwise loc}
    Let $p\colon \m{C} \rightarrow \m{O}$ be a coCartesian fibration of small \categories, let $F_{p} \colon \m{O} \rightarrow \Cat_\infty$ be the essentially unique functor which classifies $p$, and let $W \subseteq \m{C}$ be an $\m{O}$-saturated collection of arrows. 
    Then, the Dwyer--Kan localization, $\m{C}[W^{-1}]$, fits into a commutative diagram 
    \[
    \begin{tikzcd}
    \m{C} \arrow[rr,"L"] \arrow[dr,"p"'] & & \m{C}[W^{-1}] \arrow[dl,"p_W"] \\
    & \m{O} & 
    \end{tikzcd}
    \]
    of \categories such that:
    \begin{enumerate}
        \item the map $p_{W}$ is a coCartesian fibration; 
        \item the assignment $L$ is a map of coCartesian fibrations; and
        \item for each $x \in \m{O}$, the induced functor 
        \[
        L_{x}\colon \m{C}_{x} \rightarrow \m{C}[W^{-1}]_{x}
        \]
        exhibits the target as the localization of the source with respect to $W_{x}$, i.e., we have a canonical equivalence $\m{C}_{x}[W_{x}^{-1}] \xrightarrow{\sim} \m{C}[W^{-1}]_{x}$.
    \end{enumerate}
    In the special case where $W = W_{p}$, the coCartesian fibration $p_{W}$ is a left fibration.
\end{proposition}

\begin{definition}
    Let $p\colon \m{C} \rightarrow \m{O}$ be a coCartesian fibration between \categories and let $W$ be an $\m{O}$-saturated class of arrows. 
    The \textit{fiberwise localization of $p$ with respect to $W$} is defined to be the coCartesian fibration 
    \[
    p_{W}\colon \m{C}[W^{-1}] \rightarrow \m{O}.
    \]
    When $W=W_{p}$, we say that $p_{W}$ is the \textit{fiberwise groupoid completion} of $p$. 
\end{definition}

\begin{example}\label{example: fiberwise gpd of nnDelta}
Consider the coCartesian fibration $p\colon \nnDelta \rightarrow \nn$ which is classified by the filtration $\Deltaleq{\ast}^{\op}$ of $\Delta^{\op}$ as in \cref{const: skeleton functor and the I category}. 
Forming the fiberwise groupoid completion of $p$, which we denote by 
\[
\nn\ltimes\gpd{\Deltaleq{\ast}^{\op}} = (\nnDelta) [W_{p}^{-1}]
\]
we obtain a commutative diagram 
\[
\begin{tikzcd}
\nnDelta \arrow[rr,"L"] \arrow[dr,"p"'] & & \nn\ltimes\gpd{\Deltaleq{\ast}^{\op}} \arrow[dl,"p_W"] \\
& \nn. & 
\end{tikzcd}
\]
In fact, because $\gpd{\Deltaleq{b}^{\op}} \simeq \ast$ for all $b \geq 0$, the functor $p_{W}$ is an equivalence of \categories.
\end{example}

\begin{definition}
    Let $p\colon \m{C} \rightarrow \m{O}$ be a coCartesian fibration, and let $\m{E}$ be an \category. 
    We say that $\m{E}$ admits \textit{fiberwise left Kan extensions along $p$}, provided that, for each arrow $\phi \colon x \rightarrow y$ in $\m{O}$, the functor 
    \[
    \phi_{!}^{\ast} \colon \Fun(\m{C}_{y},\m{E}) \rightarrow \Fun(\m{C}_{x},\m{E}),
    \]
    given by restriction along $\phi_{!} \colon \m{C}_x \rightarrow \m{C}_{y}$, admits a left adjoint. 
    Similarly, given a map of coCartesian fibrations
    \[
    \begin{tikzcd}
    \m{C} \arrow[rr,"f"] \arrow[dr,"p"'] & & \m{D} \arrow[dl,"q"] \\
    & \m{O} & 
    \end{tikzcd}
    \]
    we say that $\m{E}$ admits \textit{fiberwise left Kan extensions along $f$} provided that for each $x \in \m{O}$, the functor 
    \[
    f_{x}^{\ast} \colon  \Fun(\m{D}_{x},\m{E}) \rightarrow \Fun(\m{C}_{x},\m{E})
    \]
    given by restriction along $f_{x} \colon \m{C}_x \rightarrow \m{D}_x$, admits a left adjoint. 
\end{definition}

\begin{construction}\label{construction: fiberwise local functors}
Let $p\colon \m{C} \rightarrow \m{O}$ be a coCartesian fibration of \categories and let $W \subseteq \m{C}$ be an $\m{O}$-saturated collection of arrows, so that by \cref{proposition: Hinich fiberwise loc}, there is a map of coCartesian fibrations 
\[
\begin{tikzcd}
\m{C} \arrow[rr,"L"] \arrow[dr,"p"'] & & \m{C}[W^{-1}] \arrow[dl,"p_W"] \\
& \m{O} & 
\end{tikzcd}
\]
By \cref{construction: relative exponential/Cartesian workhorse}, we have a map of Cartesian fibrations 
\[
\begin{tikzcd}
    \uFun_{\m{O}}(\m{C}[W^{-1}],\m{E}_{\m{O}}) \arrow[rr,"L^{\ast}"] \arrow[dr,"q_{W}"'] & & \uFun_{\m{O}}(\m{C},\m{E}_{\m{O}})  \arrow[dl,"q"] \\
    &\m{O}&
\end{tikzcd}
\]
for any \category $\m{E}$. 
Furthermore, we let $\uFun_{\m{O}}^{W}(\m{C},\m{E}_{\m{O}})$ denote the full subcategory of $\uFun_{\m{O}}(\m{C},\m{E}_{\m{O}})$ spanned by objects $(x,F_x)$ with the property that $F_x\colon \m{C}_{x} \rightarrow \m{E}$ inverts all the edges in $W_{x}$. 
It is not hard to see that the essential image of $L^{\ast}$ is equivalent to the full subcategory $\uFun_{\m{O}}^{W}(\m{C},\m{E}_{\m{O}})$. 
\end{construction}

\begin{proposition}\label{proposition: fiberwise local functors}
    Let $\m{E}$ be an \category which admits fiberwise left Kan extensions along $p$, $p_{W}$, and $L$. 
    Then, the Cartesian fibrations $q\colon \uFun_{\m{O}}(\m{C},\m{E}_{\m{O}}) \rightarrow \m{O}$ and $q_{W} \colon \uFun_{\m{O}}(\m{C},\m{E}_{\m{O}}) \rightarrow \m{O}$ are biCartesian, and $L^{\ast}$ admits a relative left adjoint
    \[\begin{tikzcd}
	{\uFun_{\m{O}}(\m{C},\m{E}_{\m{O}}) } && {\uFun_{\m{O}}(\m{C}[W^{-1}],\m{E}_{\m{O}})} \\
	& {\m{O}}
	\arrow["{L_{!}}", from=1-1, to=1-3]
	\arrow["q"', from=1-1, to=2-2]
	\arrow["{q_{W}}", from=1-3, to=2-2]
    \end{tikzcd}\]
    with the following properties: 
    \begin{enumerate}
        \item For each $x \in \m{O}$, the functor $(L_{!})_{x}$ is equivalent to $(L_{x})_{!}$, the functor of left Kan extension along $L_{x}\colon \m{C}_{x} \rightarrow \m{C}[W^{-1}]_{x} \simeq \m{C}_{x}[W_{x}^{-1}]$. 
        \item The counit of the relative adjunction, $L_{!}\circ L^{\ast} \rightarrow \mathrm{id}$ is an equivalence, so that $L^{\ast}$ is fully faithful. 
        \item $L_{!}$ carries $q$-coCartesian arrows to $q_{W}$-coCartesian arrows.
    \end{enumerate}
    Furthermore, the functor $L_{!}$ induces a relative equivalence of \categories over $\m{O}$
    \[
    L_{!}\colon \uFun_{\m{O}}^{W}(\m{C},\m{E}_{\m{O}}) \xrightarrow{\sim} \uFun_{\m{O}}(\m{C}[W^{-1}],\m{E}_{\m{O}})
    \]  
    thereby exhibiting $\uFun_{\m{O}}^{W}(\m{C},\m{E}_{\m{O}})$ as a relative reflective localization of $\uFun_{\m{O}}(\m{C},\m{E}_{\m{O}})$ over $\m{O}$. 
\end{proposition}

\begin{proof}
    That $q$ and $q_{W}$ are coCartesian follows directly from our assumptions combined with the description of coCartesian arrows of  relative exponentials from \cref{construction: relative exponential/Cartesian workhorse}.
    The existence of the relative left adjoint $L_{!}$ follows from \cite[7.3.2.6]{HA}, as $L^{\ast}$ carries $q_{W}$-Cartesian arrows to $q$-Cartesian arrows and $L_{x}^{\ast}$ admits a left adjoint for each $x \in \m{O}$. 
    
    By \cite[7.3.2.5]{HA}, the natural transformation $\mathrm{id} \rightarrow L^{\ast}\circ L_{!}$ induces, for each $x \in \m{O}$, a transformation
    \(
    \mathrm{id}_{x} \rightarrow L^{\ast}_{x} \circ (L_{x})_{!}
    \)
    which is the unit of an adjunction between $\Fun(\m{C}_x,\m{E})$ and $\Fun(\m{C}_{y},\m{E})$. 
    As $L^{\ast}_{x}$ is given by restriction along $L_{x}$, we can identify $(L_{!})_{x}$ with the left Kan extension along $L_{x}$ as claimed. 
    To see that the counit is an equivalence, it is enough to check that $L^{\ast}\circ L_{!} \rightarrow \mathrm{id}$ is an equivalence fiberwise, which follows from \cref{lemma: local functors define a reflective subcategory}.       
    
    To see that $L_{!}$ preserves coCartesian arrows, let $(x,F_x) \rightarrow (y,F_y)$ be a $q$-coCartesian arrow, with image under $L_{!}$ given by 
    \(
    (x,(L_{x})_!F_x) \rightarrow (y,(L_{y})_{!}F_y), 
    \)
    where $(\phi_{W,!})_{!}(L_{x})_!F_x \rightarrow (L_{y})_{!}F_y$ is given by the composition 
    \[
    (\phi_{W,!})_{!}(L_{x})_!F_x \simeq (L_{y})_{!}(\phi_{!})_{!}F_{x} \rightarrow (L_{y})_{!}F_y.
    \]
    This, however, is an equivalence because $(x,F_x) \rightarrow (y,F_y)$ was assumed to be $q$-coCartesian.    
\end{proof}

\begin{remark}
    Any category $\m{E}$ which admits small colimits satisfies the hypotheses of the preceding lemma. 
\end{remark}

\subsection{Existence of promonoidal localization}
We now turn to the construction of what we call \textit{$\m{O}$-promonoidal localizations.}
The following result is crucial for the proofs of \cref{introtheorem: EZ structure} and \cref{introtheorem: products of skeleta} below.

\begin{theorem}\label{theorem: promonoidal structures exist}
    Let $p\colon \m{C} \rightarrow \m{O}$ be a coCartesian fibration, let $W \subseteq \m{C}$ be an $\m{O}$-saturated collection of arrows, and let $p^{\otimes} \colon \m{C}^{\otimes} \rightarrow \m{O}^{\otimes}$ be an $\m{O}$-promonoidal structure extending $p$. 
    Assume that 
    \begin{description}
    \item[$(\ast)$] for each active morphism $\phi\colon x_{I} \rightarrow y$ in $\m{O}^{\otimes}$, the functor 
    \[
    \bigotimes^{\phi}_{i\in I} \colon \prod_{i\in I} \Fun(\m{C}_{x_i},\Spaces) \rightarrow \Fun(\m{C}_{y},\Spaces)
    \]
    carries $\prod_{i\in I} W_{x_i}$ into $\widehat{W}_{y}$.
    \end{description}
    Then, there exists an $\m{O}$-promonoidal structure extending $p_{W}$, $\m{C}[W^{-1}]^{\otimes} \rightarrow \m{O}^{\otimes}$, together with an $\m{O}$-promonoidal functor
    \(
    L^{\otimes} \colon \m{C}^{\otimes} \rightarrow \m{C}[W^{-1}]^{\otimes}. 
    \)
    Furthermore, there is a commutative diagram of $\m{O}$-promonoidal \categories and $\m{O}$-promonoidal functors
\[\begin{tikzcd}
	{\m{C}^{\otimes}} & {\Fun_{\m{O}}(\m{C},\Spaces_{\m{O}})^{\otimes,\vee}} \\
	{\m{C}[W^{-1}]^{\otimes}} & {\Fun_{\m{O}}(\m{C}[W^{-1}],\Spaces_{\m{O}})^{\otimes,\vee}}
	\arrow[from=1-1, to=1-2]
	\arrow["{L^{\otimes}}"', from=1-1, to=2-1]
	\arrow["{L_{!}^{\otimes,\vee}}", from=1-2, to=2-2]
	\arrow[from=2-1, to=2-2]
\end{tikzcd}\]
    where the horizontal maps are inclusions of full suboperads. 
\end{theorem}

\begin{proof}
    By \cite[Corollary 3.29]{linskens-nardin-pol}, the Day convolution \operad 
    \(
    q^{\otimes}\colon\Fun_{\m{O}}(\m{C},\Spaces_{\m{O}})^{\otimes} \rightarrow \m{O}^{\otimes}
    \)
    is a presentably $\m{O}$-monoidal structure which extends the coCartesian fibration 
    \(
    \uFun_{\m{O}}(\m{C},\Spaces_{\m{O}}) \rightarrow \m{O}
    \)
    obtained from \cref{construction: fiberwise local functors} and \cref{proposition: fiberwise local functors}. 
    Furthermore, again by \cref{proposition: fiberwise local functors}, we have a family of localization functors 
    \[
    \{L_{x}^{\ast}(L_{x})_{!}\colon \Fun(\m{C}_{x},\Spaces) \rightarrow \Fun(\m{C}_{x},\Spaces)\}_{x \in \m{O}}
    \]
    and we claim this family is compatible with the $\m{O}$-monoidal structure on $\uFun_{\m{O}}(\m{C},\Spaces_{\m{O}})$ in the sense of \cite[2.2.1.6]{HA}. 
    To establish the claim, it is enough to show, for each active morphism $\phi\colon x_{I} \rightarrow y$ in $\m{O}^{\otimes}$, the map 
    \[
    \bigotimes_{i\in I}^{\phi}\{F_{x_i}\} \rightarrow \bigotimes_{i\in I}^{\phi}\{L_{x_i}^{\ast}(L_{x_i})_{!}F_{x_i}\}
    \]
    is an $L_{y}^{\ast}(L_{y})_{!}$-local equivalence. 
    To see why, first note that by \cref{lemma: local functors define a reflective subcategory}, the localization functor $L_{x}^{\ast}(L_{x})_{!}$ determines (and is determined by) the reflective subcategory $\Fun(\m{C}_{x}[W_{x}^{-1}],\Spaces) \simeq \Fun(\m{C}_{x},\Spaces)[\widehat{W}_{x}^{-1}]$. 
    Subsequently, the hypothesis above implies there is an essentially unique functor $\widetilde{\bigotimes}^{\phi}_{i\in I}$ and a commutative diagram
    \[\begin{tikzcd}
	{\prod_{i\in I}\Fun(\m{C}_{x_i},\Spaces)} && {\Fun(\m{C}_{y},\Spaces)} \\
	{\prod_{i\in I}\Fun(\m{C}_{x_i}[W^{-1}_{x_i}],\Spaces)} && {\Fun(\m{C}_{y}[W_{y}^{-1}],\Spaces)}
	\arrow["{\bigotimes_{i\in I}^{\phi}}", from=1-1, to=1-3]
	\arrow["{\prod_{i\in I} (L_{x_i})_{!}}"', from=1-1, to=2-1]
	\arrow["{(L_{y})_{!}}", from=1-3, to=2-3]
	\arrow["{\widetilde{\bigotimes}_{i\in I}^{\phi}}"', dashed, from=2-1, to=2-3]
    \end{tikzcd}\]
    Therefore, after applying $(L_y)_{!}$, the morphism in question is given by 
    \[
    \widetilde{\bigotimes}^{\phi}_{i\in I} \{(L_{x_i})_{!} F_{x_i}\} \rightarrow \widetilde{\bigotimes}^{\phi}_{i\in I} \{(L_{x_i})_{!}L_{x_i}^{\ast}(L_{x_i})_{!}F_{x_i}\} 
    \]
    which is an equivalence because $(L_{x_i})_{!}L_{x_i}^{\ast} \simeq \mathrm{id}$. 
    Therefore, writing  
    \[
    q_{W}^{\otimes}\colon \uFun^{W}_{\m{O}}(\m{C},\Spaces_{\m{O}})^{\otimes} \xrightarrow{L^{\ast,\otimes}}\Fun_{\m{O}}(\m{C},\Spaces_{\m{O}})^{\otimes}  \xrightarrow{q^{\otimes}} \m{O}^{\otimes}
    \]
    for the full suboperad spanned by those objects belonging to $\uFun^{W}_{\m{O}}(\m{C},\Spaces_{\m{O}})$, \cite[2.2.1.6]{HA} implies that $q_{W}^{\otimes}$ is a presentably $\m{O}$-monoidal structure extending the coCartesian fibration
    \(
    \uFun^{W}_{\m{O}}(\m{C},\Spaces_{\m{O}}) \rightarrow \m{O}
    \)
    and $L^{\ast,\otimes}$ admits a relative left adjoint, $L_{!}^{\otimes}$, which is $\m{O}$-monoidal.
    
    Under the equivalence of \cref{proposition: fiberwise local functors}, we have produced a presentably $\m{O}$-monoidal structure which extends $q_{W}$ from \cref{proposition: fiberwise local functors}, which, by abuse of notation, we now denote by 
    \[
    q^{\otimes}_{W}\colon \uFun_{\m{O}}(\m{C}[W^{-1}],\Spaces_{\m{O}})^{\otimes} \rightarrow \m{O}^{\otimes}.
    \]    
    Therefore, by \cref{proposition: closed O-monoidal structures are O-promonoidal} and \cref{proposition: Day conv recovers presentably O-monoidal str}, there exists an $\m{O}$-promonoidal structure $p_{W}^{\otimes}\colon\m{C}[W^{-1}]^{\otimes} \rightarrow \m{O}$ extending $p_{W}$ together with an $\m{O}$-promonoidal inclusion 
    \[
    \m{C}[W^{-1}]^{\otimes} \rightarrow \uFun_{\m{O}}(\m{C}[W^{-1}],\Spaces_{\m{O}})^{\otimes,\vee}
    \]
    inducing an equivalence of presentably $\m{O}$-monoidal \categories
    \[
    \uFun_{\m{O}}(\m{C}[W^{-1}],\Spaces_{\m{O}})^{\otimes} \xrightarrow{\sim} \Fun_{\m{O}}(\m{C}[W^{-1}],\Spaces_{\m{O}})^{\otimes}.
    \]
    The existence of the commutative diagram 
    \[\begin{tikzcd}
	{\m{C}^{\otimes}} & {\Fun_{\m{O}}(\m{C},\Spaces_{\m{O}})^{\otimes,\vee}} \\
	{\m{C}[W^{-1}]^{\otimes}} & {\Fun_{\m{O}}(\m{C}[W^{-1}],\Spaces_{\m{O}})^{\otimes,\vee}}
	\arrow[from=1-1, to=1-2]
	\arrow["{L^{\otimes}}"', from=1-1, to=2-1]
	\arrow["{L_{!}^{\otimes,\vee}}", from=1-2, to=2-2]
	\arrow[from=2-1, to=2-2]
    \end{tikzcd}\]
    follows immediately from the fact that left Kan extension preserves representables, and the $\m{O}$-promonoidality of $L^{\otimes}$ follows directly from the $\m{O}$-monoidality of $L_{!}^{\otimes,\vee}$.
\end{proof}

\begin{remark}\label{remark: promonoidal loc existence criteria}
    The assumption $(\ast)$ in \cref{theorem: promonoidal structures exist} admits an equivalent reformulation: 
    \begin{description}
    \item[$(\ast)$] for each active morphism $\phi\colon x_{I} \rightarrow y$ in $\m{O}^{\otimes}$, and for each tuple of morphisms $(f_{i}\colon c_{i} \rightarrow c_{i}')_{i\in I}$ belonging to $\prod_{i\in I} W_{x_i} \subseteq \prod_{i\in I}\m{C}_{x_i}$, the induced transformation 
    \[
    \Mul^{\phi}_{\m{C}}(\{c_{i}'\}_{i\in I},-) \rightarrow \Mul^{\phi}_{\m{C}}(\{c_{i}\}_{i\in I},-)
    \]
    is an equivalence after left Kan extension along $L_{y} \colon \m{C}_{y} \rightarrow \m{C}[W^{-1}]_{y} \simeq \m{C}_{y}[W_{y}^{-1}]$. 
    \end{description}
    This can be deduced directly from the definition combined with \cref{lemma: local functors define a reflective subcategory}. 
\end{remark}

\begin{remark}
    We expect that the $\oO$-promonoidal \category, $\m{C}[W^{-1}]^{\otimes}$, constructed above is the initial $\m{O}$-promonoidal \category equipped with a lax $\m{O}$-promonoidal functor from $\m{C}^{\otimes}$ which inverts the arrows in $W$. 
    However, as we do not require this characterization of $\m{C}[W^{-1}]^{\otimes}$ for our work in this article, we will not pursue such a description here. 
    In spite of this, we introduce the following convenient terminology. 
\end{remark}

\begin{definition}
    Let $p\colon \m{C} \rightarrow \m{O}$ be a coCartesian fibration, let $W \subseteq \m{C}$ be an $\m{O}$-saturated collection of arrows, and let $p^{\otimes} \colon \m{C}^{\otimes} \rightarrow \m{O}^{\otimes}$ be an $\m{O}$-promonoidal structure extending $p$ which satisfies the condition of \cref{theorem: promonoidal structures exist}. 
    We call the resulting $\m{O}$-promonoidal \category 
    \(
    \m{C}[W^{-1}]^{\otimes} \rightarrow \m{O}^{\otimes}
    \)
    the \textit{$\m{O}$-promonoidal localization of $\m{C}^{\otimes} \rightarrow \m{O}^{\otimes}$ with respect to $W$}. 
    If $W=W_{p}$, then we call this the \textit{$\m{O}$-promonoidal groupoid completion of $\m{C}^{\otimes} \rightarrow \m{O}^{\otimes}$}. 
\end{definition}

\section{The skeletal filtration and the Dold--Kan correspondence}\label{section: the skeletal filtration}

We recall in this section the $\infty$-categorical Dold--Kan correspondence established by Lurie in \cite[1.2.4.1]{HA}, as well as its relationship with
the classical Dold--Kan correspondence that asserts an equivalence between simplicial and non-negative chain complexes in an abelian category. 

\subsection{Filtered objects are chain complexes}\label{subsec: filtered obj=chains}
Lurie argues in \cite{HA} that the analogues of non-negative chain complexes in a stable \category $\cE$ are the filtered objects in $\cE$, i.e., functors $\nn\to \cE$.
Informally, given a filtered object $F_\ast$ in $\cE$, we can define a chain complex in its homotopy category $h\cE$, viewed as an additive category:
\[
\begin{tikzcd}
\cdots \ar{r} & \Sigma^{-2}F_2/F_1 \ar{r} & \Sigma^{-1} F_1/F_0 \ar{r} & F_0 \ar{r} & 0.
\end{tikzcd}
\]
The differentials are given using cofiber sequences such as:
\[
\begin{tikzcd}
    F_0\ar{r} & F_1 \ar{r} & F_1/F_0 \ar{r} & \Sigma F_0
\end{tikzcd}
\]
and
\[
\begin{tikzcd}
    F_1/F_0 \ar{r} & F_2/F_0 \ar{r} & F_2/F_1 \ar{r} & \Sigma F_1/F_0
\end{tikzcd}
\]
and so on, induced from the filtration. 
We shall review how one can formalize the construction above so that the \category $\Fun(\nn, \cE)$ of (non-negative) filtered objects in $\cE$  can be regarded as equivalent to (non-negative) homotopy coherent chain complexes in $\cE$. 

\begin{definition}[{\cite[35.1]{Joyal}, \cite[2.1]{stefano}}]
   Denote by $\ch$ the pointed category whose objects are $\zz \cup \{\ast \}$, and whose set of morphisms is defined as:
   \[
   \ch(m,n)= \begin{cases}
       \{\partial, 0\}, & \text{if }m=n+1,\\
       \{\id, 0\}, & \text{if }m=n,\\
       \{0\}, & \text{otherwise}.
   \end{cases}
   \]
   By construction, we have that $\partial \circ \partial= 0$.
   A homotopy coherent chain complex in a pointed \category $\cE$ is a pointed functor $\ch\to \cE$. We denote by $\ch(\cE)$ the full subcategory of $\Fun(\ch, \cE)$ spanned by the pointed functors.
   
   Analogously, we define $\chp$ as the same pointed category whose objects are $\nn \cup \{\ast \}$, and define homotopy coherent non-negative chain complexes in $\cE$ in a similar fashion, that assemble into an \category $\chp(\cE)$.
\end{definition}

\begin{remark}\label{remark: nerve of chain is chain of nerves}
    If $\cE$ is (the nerve of) an abelian category $\m{A}$ then $\ch(\m{A})$ is equivalent to (the nerve of) the category of chain complexes in $\m{A}$. We can observe a similar result for non-negative chain complexes.
\end{remark}

In \cite[4.7, 4.8]{stefano}, Ariotta determines an adjunction of \categories:
\[
\begin{tikzcd}
    \Fun(\zz\,, \cE) \ar[bend left]{r}{} \ar[phantom, description, xshift=-0.8ex]{r}{\perp}& \ar[bend left]{l}{}\ch(\cE)
\end{tikzcd}
\]
for any complete stable \category $\cE$. It becomes an equivalence of \categories if we restrict to the full subcategory of $\Fun(\zz\,, \cE)$ spanned by complete filtered objects, i.e.\ functors $F_*\colon \zz\to \cE$ such that $\lim_{\zz}F_*\simeq 0$. Tracing through the arguments, one obtain a equivalence of $\infty$-categories on non-negative objects:
\begin{align*}
 \Fun(\nn, \cE)&\stackrel{\simeq}\longrightarrow \chp(\cE)\\
 F_*&\longmapsto \Sigma^{-*}F_*/F_{*-1}
\end{align*}
for any complete stable $\infty$-category $\cE$.

\begin{remark}
    In \cite{stefano}, Ariotta proved his results for cochain complexes rather than chain complexes, but his approach is entirely dualizable since the category $\ch$ is self-dual, i.e.\ $\ch^\op\simeq \ch$. In other words, the data of a chain complex:
    \[
    \begin{tikzcd}
        \cdots \ar{r} & C_2 \ar{r} & C_1 \ar{r} & C_0 \ar{r} & C_{-1} \ar{r} & C_{-2} \ar{r} & \cdots
    \end{tikzcd}
    \]
    is equivalent to the following cochain complex where $C_n'=C_{-n}$:
     \[
    \begin{tikzcd}
        \cdots \ar{r} & C_{-2}' \ar{r} & C_{-1}' \ar{r} & C_0' \ar{r} & C_{1}' \ar{r} & C_{2}' \ar{r} & \cdots.
    \end{tikzcd}
    \]
    
\end{remark}

\subsection{The higher categorical Dold--Kan correspondence}
We shall recall first how to obtain a filtered object from any simplicial object.

\begin{construction}\label{const: skeleton functor and the I category}
Denote by $\Deltaleq{n}$ the full subcategory of the simplex category $\Delta$ spanned by the objects $[0],\ldots, [n]$. This defines a filtration of $\Delta^\op$:
\[
\Deltaleq{0}^{\op} \subset \Deltaleq{1}^{\op} \subset \Deltaleq{2}^{\op} \subset \cdots,
\]
and therefore determines a functor $\nn\rightarrow \Cat$. The functor classifies a coCartesian fibration $p\colon\nnDelta\rightarrow \nn$.
The category $\nnDelta$ appears in \cite[1.2.4.17]{HA} (denoted $\m{I}$ therein) and can be described explicitly as the full subcategory of $\nn\times \Delta^\op$ spanned by pairs $(a, [n])$ for which $n\leq a$.
Then the coCartesian fibration $p\colon \nnDelta\rightarrow\nn$ is the projection. Denote by $q\colon \nnDelta\rightarrow \Delta^\op$ the other projection. Given an $\infty$-category $\cE$ that is complete and finitely cocomplete, we obtain the following adjunctions:
\[
\begin{tikzcd}
\Fun(\Delta^\op, \cE) \ar[bend left]{r}{q^*} \ar[phantom, description, xshift=0.8ex]{r}{\perp}& \ar[bend left]{l}{q_*}\Fun(\nnDelta, \cE)\ar[bend left]{r}{p_!}\ar[phantom, description, xshift=-1.5ex]{r}{\perp} & \ar[bend left]{l}{p^*} \Fun(\nn, \cE).
\end{tikzcd}
\]
The skeleton filtration $\sk^\cE_\ast\colon \Fun(\Delta^\op, \cE)\rightarrow \Fun(\nn, \cE)$ is defined as the composition of the above left adjoints: $\sk^\cE_\ast=p_!\circ q^*$.
Of course, the skeleton filtration exists if we only assume $\cE$ to be finitely cocomplete, and is given on a simplicial object $X\colon \Delta^\op\to \cE$ as:
\[
\sk^\cE_n(X) \simeq \colim_{\Delta^\op_{\leq n}} \iota_n^*(X) 
\]
where $\iota_n\colon \Delta^\op_n\hookrightarrow \Delta^\op$ is the full subcategory functor.
If we assume $\cE$ to be also complete, then $\sk^\cE_*$ admits a right adjoint denoted $\Gamma^\cE_\bullet\colon \Fun(\nn, \cE)\to \Fun(\Delta^\op, \cE)$ that is defined as the composition  of the right adjoints $\Gamma^\cE_\bullet=q_*\circ p^*$. When $\cE$ is the \category of spectra $\Sp$, we write simply $\sk_*$.
\end{construction}

In \cref{section: appendix EZ model cat}, we show that the above skeletal filtration on the \category of spaces $\cS$ can be regarded as the geometric realization of the usual skeleton filtration of bisimplicial sets.

\begin{remark}
    Let $(\cE, \otimes, \unit)$ be a  symmetric monoidal \category that is finitely cocomplete.
    Denote by $\unit_*[\Delta^n]$ the filtered object in $\cE$ obtained by the tensoring  $\unit \in \cE$ with the filtered space $\sk_*^\cS(\Delta^n)$ where $\Delta^n$ is the standard $n$-simplex.
    Then we may think of the skeleton filtration $\sk^\cE_*\colon \Fun(\Delta^\op, \cE)\to \Fun(\nn, \cE)$ as given by the coend formula:
    \[
    \sk^\cE_*(X)\simeq \int^{[n]\in \Delta} \unit_*[\Delta^n]\odot X_n.
    \]
    If we further assume that $\cE$ is closed monoidal and complete, there is an explicit description of the right adjoint of the skeleton filtration $\Gamma^\cE_\bullet\colon \Fun(\nn, \cE)\to \Fun(\Delta^\op, \cE)$ given for all filtered object $F_*$ in $\cE$ by the formula:
    \[
\Gamma^\cE_n(F_*)=\Hom(\unit_*[\Delta^n], F_*), 
    \]
    where $\Hom(-,-)$ denotes the induced enrichment of $\Fun(\nn, \cE)$ over $\cE$. 
\end{remark}

While the skeleton filtration can be defined for any finitely cocomplete \category, the functor becomes an equivalence of \categories if we assume stability. 

\begin{theorem}[{\cite[1.2.4.1]{HA}}]\label{theorem: Lurie DK-theorem}
    Let $\cE$ be a stable \category.
    Then the skeleton filtration induces an equivalence of \categories:
    \[
    \sk_*^\cE\colon \Fun(\Delta^\op, \cE) \stackrel{\simeq}\longrightarrow \Fun(\nn, \cE).
    \]
\end{theorem}

Recall from \cite[1.2.2.9, 1.2.2.14]{HA} that every filtered object $F\colon \nn\to \Sp$ defines a spectral sequence 
\[
E_1^{p,q}=\pi_{p+q}(F_p/F_{p-1})\Rightarrow \pi_{p+q}(\colim F_*).
\]
Therefore it is useful to think of $\sk_p(X)/\sk_{p-1}(X)$ as computing the homotopy type of the geometric realization $|X|$ for any simplicial spectrum $X$.

\subsection{Relationship to the classical Dold-Kan correspondence}
The last theorem relates to the classical Dold--Kan correspondence of abelian categories as explained in \cite[1.2.4.3]{HA}. 
The Dold--Kan correspondence $\sAb\simeq \choo$ between simplicial abelian groups and non-negatively graded chain complexes is given by the normalization functor $\norm_*\colon\sAb \rightarrow \choo$ that we now recall.
Let $X$ be a simplicial abelian group. For all $n\geq 0$, define $C_n(X)=X_n$ and its differentials are given by the alternating sum $d=\sum_{i=0}^n (-1)^i d_i$ of the face maps of $X$.
The chain complex $C_*(X)$ has a natural sub chain complex $D_*(X)$ given by the degenerate simplices. We define the normalization of $X$ to be the quotient in $\choo$:
\[
\norm_*(X)=C_*(X) / D_*(X).
\]
It is useful to notice that $D_*(X)$ is acyclic and thus $\norm_*(X)$ is quasi-isomorphic to $C_*(X)$. Importantly, we have $H_*(\norm_*(X))\cong \pi_*(X)$ for any simplicial abelian group $X$. 

There are other definitions of the normalization functor, such as the Moore complex of a simplicial abelian group, but they are all equivalent. 
In the next lemma, which is likely well-known, we prove that any functor $\m{F}_{\ast} \colon \sAb \rightarrow \choo$ which is an equivalence of categories, is naturally isomorphic to the normalization functor. 

\begin{lemma}\label{lemma: uniqueness of DoldKan Equivalence}
If $\m{F}_*\colon \sAb\to \choo$ is an equivalence of categories, then there exists a natural isomorphism from the normalization $\norm_*$ to $\m{F}_{\ast}$.
\end{lemma}

\begin{proof}  
    We begin with a few observations about abelian categories. 
    First, if $S$ is a set of indecomposable projective generators in an abelian category $\m{A}$, then any indecomposable projective object $P \in \m{A}$ must be isomorphic to some object in $S$; indeed, $P$ must be a summand of a sum of objects in $S$, but $P$ is also indecomposable. 
    Second, any equivalence of abelian categories between $\m{A}$ and $\m{B}$ must carry $S$ to a set of indecomposable projective generators of $\m{B}$.
    Third, any additive and right exact functor $\m{A} \rightarrow \m{B}$ is uniquely determined by its restriction to the full subcategory spanned by $S$, and the same is true for natural transformations between such functors.  

    For $n\geq 0$, let $D^n$ denote the non-negative chain complex that is $\zz$ in degrees $n$ and $n-1$ and zero elsewhere, whose only nontrivial differential is the identity map; as $D^{n}$ is concentrated in non-negative degrees, this implies $D^0$ is chain complex given by $\zz$ concentrated in degree $0$. 
    The collection $\{D^{n}\}_{n \geq 0}$ forms a set of indecomposable projective generators of $\choo$.
    Let $\Gamma\colon \choo\to \sAb$ be the inverse of $\norm_*$, where the $n$-simplices are given, for any non-negative chain complex $C_*$, by $\Gamma(C_*)_n=\Hom_{\choo}(\norm_*(\zz[\Delta^n]), C_*)$ for all $n\geq 0$.
    Writing $\mathbb{D}^{n} = \Gamma(D^n)$, the collection $\{\mathbb{D}^{n}\}_{n \geq 0}$ is a set of indecomposable projective generators for $\sAb$, and the collections $\{\norm_{\ast}(\mathbb{D}^{n})\}_{n \geq 0}$ and $\{\m{F}_{\ast}(\mathbb{D}^{n})\}_{n\geq 0}$ are sets of indecomposable projective generators for $\choo$.
    Note that the counit of the equivalence induces isomorphisms $\norm_*(\mathbb{D}^n)\cong D^n$, natural in $\mathbb{D}^{n}$, and by our observations above, for each $n \geq 0$, there is an isomorphism $\m{F}_{\ast}(\mathbb{D}^{n}) \cong D^{n_{\m{F}}}$ for some $n_{\m{F}} \geq 0$.  
    From the isomorphisms 
    \begin{align*}
        \Hom_{\sAb}(\mathbb{D}^{m},\mathbb{D}^{n})\cong\Hom_{\choo}(\m{F}_{\ast}(\mathbb{D}^{m}), \m{F}_*(\mathbb{D}^n)) & \cong \Hom_{\choo}(D^{m_{\m{F}}}, D^{n_{\m{F}}}) \\
        &=\begin{cases}
        \zz & \text{if } n=m, n=m+1\\
        0 & \text{otherwise}.
    \end{cases}
    \end{align*}
    it follows that $n = n_{\m{F}}$ for all $n \geq 0$.     
        

     By our previous discussion, to construct a natural isomorphism of functors $\phi\colon \norm_*\Rightarrow \m{F}_*$, it is enough to do so upon restriction to the full subcategory spanned by $\{\mathbb{D}^n\}_{n\geq 0}$. 
     Therefore we need to argue we 
     can choose isomorphisms $\phi_n\colon  \norm_*(\mathbb{D}^n)\cong D^n \stackrel{\cong}\rightarrow \m{F}_*(\mathbb{D}^n)$ for all $n\geq 0$ so that the diagrams 
     \begin{equation}\label{eq: uniqueness of DK}
\begin{tikzcd}
    \norm_*(\mathbb{D}^n)\cong D^n \ar{r}{\phi_n} \ar{d}[swap]{\norm_*(f)} & \m{F}_*(\mathbb{D}^n) \ar{d}{\m{F}_*(f)}\\
    \norm_*(\mathbb{D}^k)\cong D^k \ar{r}{\phi_k} & \m{F}_*(\mathbb{D}^k)
\end{tikzcd}
\end{equation}
commute for any map $f\colon \mathbb{D}^n\to \mathbb{D}^k$ in $\sAb$. Notice the only non-trivial cases are for $k=n, n+1$.

We build $\phi_n\colon D^n\to \m{F}_*(\mathbb{D}^n)$ inductively on $n\geq 0$.
    Denote by 
    $\alpha_n$ the generator of $\Hom_{\choo}(D^{n}, D^{n+1})$ for all $n\geq 0$. Denote by 
    $\overline{\alpha}_n$ a chosen generator of $\Hom_{\sAb}(\mathbb{D}^n, \mathbb{D}^{n+1}))$ so that 
    $\norm_*(\overline{\alpha}_n)=\alpha_n$.
    For $n=0$, choose a generator $\phi_0$ of $\Hom_{\choo}(D^0, \m{F}_*(\mathbb{D}^0))\cong \zz$.
    Inductively, suppose we have the isomorphism $\phi_n\colon D^n\stackrel{\cong}\rightarrow\m{F}_*(\mathbb{D}^n)$ for some $n\geq 0$, we want to define the isomorphism $\phi_{n+1}\colon D^{n+1}\stackrel{\cong}\rightarrow \m{F}_*(\mathbb{D}^{n+1})$ such that the following diagram commutes:
    \[
    \begin{tikzcd}
        D^n \ar{r}{\phi_n} \ar{d}[swap]{{\alpha}_n} & \m{F}_*(\mathbb{D}^n) \ar{d}{\m{F}_*(\overline{\alpha}_n)}\\
        D^{n+1} \ar{r}{\phi_{n+1}} & \m{F}_*(\mathbb{D}^{n+1}),
    \end{tikzcd}
    \]
    i.e.\ $\phi_{n+1}\circ {\alpha}_n=\m{F}_*(\overline{\alpha}_n)\circ \phi_n$.
    Notice the group homomorphism
    \begin{align*}
        {\alpha}_n^*\colon \Hom_{\choo}(D^{n+1}, \m{F}_*(\mathbb{D}^{n+1})) & \longrightarrow \Hom_{\choo}(D^n, \m{F}_*(\mathbb{D}^{n+1}))\\
        f & \longmapsto f\circ {\alpha}_n
    \end{align*}
is equivalent to the differential $\m{F}_{n+1}(\mathbb{D}^{n+1})\to \m{F}_n(\mathbb{D}^{n+1})$ which must be an isomorphism. 
Thus we may define the isomorphism $\phi_{n+1}$ as
\[
\phi_{n+1}=({\alpha}_n^*)^{-1}(\m{F}_*(\overline{\alpha}_n) \circ \phi_n).
\]
By construction, we see that our choice of $\phi_n$ forces the diagrams of \cref{eq: uniqueness of DK} to commute.
This builds the natural isomorphism $\phi\colon \norm_* \Rightarrow \m{F}_*$ on the indecomposable projective generators, and hence on all objects of $\sAb$ and $\choo$.
\end{proof}

\begin{remark}
    Let $\mathrm{Eqv}(\sAb,\choo)$ denote the groupoid whose objects are equivalences of categories and whose morphisms are natural isomorphisms of functors. 
    The above result shows that this groupoid is connected, which implies we have equivalences  
    \[
    \mathrm{Eqv}(\sAb,\choo) \simeq \mathrm{Eqv}(\choo,\choo) \simeq B\mathrm{Aut}(\mathrm{id}_{\choo}) \simeq B\zz/2
    \]
    The last equivalence follows from the fact that there are exactly two automorphisms of the abelian group $\zz$. 
\end{remark}

We now relate the skeleton filtration with the normalization.
Following our conversation in \cref{subsec: filtered obj=chains}, given a stable $\infty$-category $\cE$, the skeleton filtration induces an equivalence:
\[
    \sk_*^\cE\colon \Fun(\Delta^\op, \cE) \stackrel{\simeq}\longrightarrow \Fun(\nn, \cE)\simeq \chp(\cE)
    \]
by \cref{theorem: Lurie DK-theorem}, see also Joyal's approach in \cite[35.3]{Joyal}.
Recall the notion of $t$-structure of a stable $\infty$-category $\cE$, which is notably determined by a pair of full subcategories $\cE_{\geq 0}$ and $\cE_{\leq 0}$, see \cite[1.2.1.4]{HA}. A $t$-structure also determines subcategories $\cE_{\geq n}$ and $\cE_{\leq n}$ after $n$ amount of shifts. For the case $\cE=\Sp$, we will use the standard $t$-structure on spectra \cite[1.4.3.6]{HA}.
Suppose now that $\cE$ is a complete stable $\infty$-category with a $t$-structure. 
Then there is a $t$-structure on $\Ch(\cE)$ given levelwise, for which $\Ch(\cE)_{\geq 0}$ is spanned by objects in $\Ch(\cE_{\geq 0})$, and $\Ch(\cE)_{\leq 0}$ is spanned  by objects in $\Ch(\cE_{\leq 0})$ \cite[5.1]{stefano}.
One can similarly define a $t$-structure on non-negative chain complexes $\chp(\cE)$.

\begin{warning}
    It is important to not confuse the stable $\infty$-category $\chp(\cE)$ of non-negative chain complexes in $\cE$ with the (non-stable) $\infty$-category $\Ch(\cE)_{\geq 0}$ of chain complexes in $\cE_{\geq 0}$.
\end{warning}

Given certain assumptions on the $t$-structure of $\cE$, there is a $t$-structure on filtered objects $\Fun(\zz, \cE)$ called the \textit{Beilinson $t$-structure}, for which $\Fun(\zz, \cE)_{\geq 0}$ is spanned by filtered objects $F_*$ such that $F_n/F_{n-1}$ is in $\cE_{\geq n}$, while $\Fun(\zz, \cE)_{\leq 0}$ is spanned by filtered objects such that $F_i$ is in $\cE_{\leq i}$. 
This $t$-structure on filtered objects is compatible with the $t$-structure on $\Ch(\cE)$, see \cite[5.5, 5.12]{stefano}. The observation can be carried over non-negative chain complexes so that there exists a $t$-structure on $\Fil(\nn, \cE)$ that is equivalent to the one on $\chp(\cE)$, which in particular induces an equivalence on the hearts \cite[1.2.11]{HA}: $\Fil(\cE)^\heartsuit\simeq \Ch(\cE^\heartsuit)$, see \cite[5.4]{stefano}.
One can also define a levelwise $t$-structure on simplicial objects $\Fun(\Delta^\op, \cE)$, and can check the equivalence $\Fun(\Delta^\op, \cE)\simeq \Fun(\nn, \cE)$ is compatible with the choice of $t$-structures.

\begin{lemma}\label{lemma: skeletal filtration is t-exact for beilinson}
    The skeletal filtration 
    \(
    \sk_\ast \colon \Fun(\Delta^{\op},\Sp) \rightarrow \Fun(\nn,\Sp)
    \)
    is $t$-exact with respect to the pointwise $t$-structure on the source and the Beilinson $t$-structure on the target. 
\end{lemma}

\begin{proof}
    Let $X \in \Fun(\Delta^{\op},\Sp)_{\geq 0} \simeq \Fun(\Delta^{\op},\Sp_{\geq 0})$. 
    The fact that $\sk_{\ast}(X)$ belongs to $\Fun(\nn,\Sp)_{\geq 0}$ follows from the cofiber sequences 
    \[
    \sk_{n-1}(X) \rightarrow \sk_{n}(X) \rightarrow \Sigma^{n}X_{n}/L_{n}(X)
    \]
    where $L_n(X)$ denotes the usual $n$-th latching object of $X$,
    together with the fact that any small colimit of connective spectra remains connective; therefore, we have $\Sigma^{n}X_{n}/L_{n}(X) \in \Sp_{\geq n}$. 

    Now, let $Y \in \Fun(\Delta^{\op},\Sp)_{\leq0} \simeq \Fun(\Delta^{\op},\Sp_{\leq 0})$. 
    For each $n \geq 0$, let $P_{\neq \varnothing}([n])$ denote the set of nonempty $S \subseteq [n]$, with the partial order given by inclusion. 
    There is a functor $P_{\neq \varnothing}([n]) \rightarrow \Delta_{\sleq{n}}$ given by sending some subset $S=\{s_{0}<\cdots <s_{k}\}$ to $[k]$ and by sending an inclusion $S\setminus\{s_i\} \subseteq S$ to the relevant face operator $[k-1] \rightarrow [k]$; this functor is final, so that $\tau\colon P_{\neq \varnothing}([n])^{\op} \rightarrow \Delta_{\sleq{n}}^{\op}$ is cofinal.
    Therefore, for each $n\geq 0$, the spectrum $\sk_{n}(Y)$ is a cubical colimit of the diagram $\tau^{\ast}Y$ and can be calculated iteratively over smaller cubical diagrams. 
    Additionally, recall that the cofiber of a map of $k$-coconnective spectra, is a $(k+1)$-coconnective spectrum.
    Combining these observations, we see that $\sk_{n}(Y) \in \Sp_{\leq n}$ for all $n\geq 0$, completing the proof.    
\end{proof}

In particular, as the heart of $\Fun(\Delta^\op, \cE)$ is $\Fun(\Delta^\op, \cE^\heartsuit)$, we obtain that the skeleton filtration induces an equivalence on the hearts:
\[
\sk^{\cE,\heartsuit}_*\colon \Fun(\Delta^\op, \cE^\heartsuit) \stackrel{\simeq}\longrightarrow \Fun(\nn, \cE)^\heartsuit\simeq \chp(\cE^\heartsuit).
\]
Choosing $\cE$ to be the \category $\Sp$ of spectra with its usual $t$-structure, its heart $\Sp^\heartsuit$ is precisely the category $\mathrm{Ab}$ of abelian groups. 
Combining with \cref{lemma: uniqueness of DoldKan Equivalence},
this observation leads to the following result, see also \cite[1.2.4.3]{HA}.

\begin{proposition}
    The equivalence $\sk_*\colon \Fun(\Delta^\op, \Sp)\stackrel{\simeq}\rightarrow \chp(\Sp)$ of \categories recovers on its heart the usual Dold--Kan correspondence given by the normalization functor seen as an equivalence of categories
    \[
   \norm_*\colon \sAb\stackrel{\simeq}\longrightarrow \choo.
    \]
\end{proposition}

\subsection{The classical Eilenberg--Zilber homomorphism}
The normalization functor $\norm_*\colon\sAb\rightarrow \choo$ is a unital lax symmetric monoidal functor via the Eilenberg--Zilber homomorphism that we now recall, see more details in \cite[8.5.4]{weibel} or \cite[\href{https://kerodon.net/tag/00RF}{Subsection 00RF}]{kerodon}.
\begin{definition}
    A \textit{$(p,q)$-shuffle} $\sigma\colon [p+q]\to [p]\times [q]$ is a strictly increasing map of partially ordered sets, i.e.\ a permutation of the set $\{0,\ldots , p + q- 1\}$ which leaves the first $p$ elements and the last $q$ elements in their natural order. It is useful to think of $(p,q)$-shuffles as the non-degenerate simplices of $\Delta^p\times \Delta^q$.
We shall denote $\sigma=(\sigma_{-}, \sigma_+)$ and we let $I_{-}$ be the set of integers $1\leq i \leq p+q$ such that $\sigma_{-}(i-1)<\sigma_-(i)$ and $I_+$ the set of integers $1\leq i \leq p+q$ such that $\sigma_+(i-1)<\sigma_+(i)$.
We let $S_{pq}$ be the set of all $(p,q)$-shuffles.
We define the sign of $\sigma$ as:
\[
\mathrm{sgn}(\sigma)=\prod_{(i,j)\in I_-\times I_+} \begin{cases}
    1 & \text{if }i<j\\
    -1 & \text{if }i>j.
\end{cases}
\]
\end{definition}

For all $p,q\geq 0$, define the \textit{unnormalized shuffle product} as:
\begin{align*}
    \nabla_{pq}\colon X_p\otimes Y_q& \longrightarrow  X_{p+q}\otimes Y_{p+q}\\
    x\otimes y & \longmapsto \sum_{\sigma\in S_{pq}} \mathrm{sgn}(\sigma)\sigma^*_-(x)\otimes \sigma^*_+(y),
\end{align*}
where given a $(p,q)$-shuffle $\sigma=(\sigma_-, \sigma_+)$,  we wrote $\sigma^*_-\colon X_p\to X_{p+q}$ and $\sigma^*_+\colon Y_q\to Y_{p+q}$ the induced simplicial group homomorphisms.
The unnormalized shuffle product map preserves the subcomplex of degenerate simplices and thus induces a unique homomorphism $\nabla\colon \norm_p(X)\otimes \norm_q(Y)\to \norm_{p+q}(X\otimes Y)$ called the \textit{shuffle product}.
The shuffle product is natural in $X$ and $Y$, and induces a chain homomorphism \[\nabla\colon \norm_*(X)\otimes \norm_*(Y)\longrightarrow \norm_*(X\otimes Y)\]
that we call the \textit{Eilenberg--Zilber homomorphism}, and is the desired unital lax symmetric monoidal transformation on $\norm_*$ introduced in \cite{EZ}. 

It is a classical fact that the shuffle product is unique: there is a unique collection of homomorphisms $\nabla\colon \norm_p(X)\otimes \norm_q(Y)\to \norm_{p+q}(X\otimes Y)$ that are natural in $X$ and $Y$, satisfying the Leibniz rule (i.e. induces a chain homomorphism), and such that $\nabla(x\otimes y)=x\otimes y$ for all $x\in X_0$ and $y\in  Y_0$, see for instance \cite[\href{https://kerodon.net/tag/00RG}{Proposition 00RG}]{kerodon}.
We strengthen this observation building on \cref{lemma: uniqueness of DoldKan Equivalence}. 

\begin{proposition}\label{prop:unicity of classical EZ}
Suppose $\m{F}_*\colon \sAb\to \choo$ is an equivalence of categories that is unital lax symmetric monoidal.
Then $\m{F}_*$ is uniquely naturally isomorphic to the normalization $\norm_*\colon \sAb\to \choo$ so that its lax symmetric monoidal structure is equivalent to the Eilenberg--Zilber homomorphism. 
\end{proposition}

\begin{proof}
Let $D^n$ and $\mathbb{D}^n$ be as in \cref{lemma: uniqueness of DoldKan Equivalence}.
Notice that $D^0$ and $\mathbb{D}^0$ are the monoidal units of $\choo$ and $\sAb$ respectively.
Let us denote by $\mu\colon \m{F}_{\ast}(-)\otimes \m{F}_{\ast}(-)\to \m{F}_{\ast}(- \otimes -)$ and $\eta\colon D^0 \to \m{F}_*(\mathbb{D}^0)$, the multiplication and unit maps which exhibit $\m{F}_{\ast}$ as lax symmetric monoidal. 

Recall from the proof of \cref{lemma: uniqueness of DoldKan Equivalence} there is a natural isomorphism $\phi\colon \norm_*\Rightarrow \m{F}_*$, which was uniquely determined by the choice of a homomorphism $\phi_0\colon D^0\to \m{F}_*(\mathbb{D}^0)$. 
We now show that we can choose exactly one such $\phi$ so that it becomes a unital lax symmetric monoidal natural isomorphism from $\norm_\ast$ to $\m{F}_{\ast}$.

Notice by unitality, we must choose $\phi_0=\eta$.
    We now need to verify that $\phi$ is compatible with $\mu$ and $\nabla$. 
    By exactness, it is enough to check that for all $p,q\geq 0$, we have the commutative diagram:
    \[
    \begin{tikzcd}
    D^p\otimes D^q  \ar[phantom, "\cong"]{r}\ar{d}[swap]{\phi_p\otimes \phi_q}& [-3em] \norm_*(\mathbb{D}^p)\otimes \norm_*(\mathbb{D}^q)  \ar{r}{\nabla} & \norm_*(\mathbb{D}^p\otimes \mathbb{D}^q)\ar{d}{\phi}\\
     \m{F}_*(\mathbb{D}^p)\otimes \m{F}_*(\mathbb{D}^q) \ar{rr}{\mu} & & \m{F}_*(\mathbb{D}^p\otimes \mathbb{D}^q).
    \end{tikzcd}
    \]
    We denote by $\theta_{p,q}\colon D^p\otimes D^q\to \m{F}_*(\mathbb{D}^p\otimes \mathbb{D}^q)$ the chain map obtained as $\phi\circ \nabla$ as above, and $\zeta_{p,q}\colon D^p\otimes D^q\to \m{F}_*(\mathbb{D}^p\otimes \mathbb{D}^q)$ the one obtained as $\mu\circ(\phi_p\otimes \phi_q)$. 
    We must show $\theta_{p,q}=\zeta_{p,q}$ for all $p,q\geq 0$. 
    
    We do this by induction on $p+q$. If $p+q=0$, then $p=0$ and $q=0$, and  $\theta_{0,0}=\zeta_{0,0}$ follows from our choice of $\phi_0$.
    Suppose we have shown $\theta_{p,q}=\zeta_{p,q}$ for all $0\leq p+q\leq n-1$, for some $n\geq 1$. Suppose now $p+q=n$.
    Suppose either $p=0$ or $q=0$, then the desired equality follows again by our choice of $\phi_0$, so let us assume $p,q>0$. 

    Recall that the chain complex $D^p\otimes D^q$ is precisely
     \[
    \begin{tikzcd}
        \cdots \ar{r} & 0 \ar{r} & \mathbb{Z} \ar{r}{(1,(-1)^{p})}& [2em]\zz^{\oplus 2} \ar{r}{(1,(-1)^{p+1}) } & [3em] \zz \ar{r} & 0 \ar{r} & \cdots
    \end{tikzcd}
    \]
    where $\zz^{\oplus 2}$ is in degree $p+q-1$.
    Given a chain complex $C_*$, notice a chain map $f_{p,q}\colon D^p\otimes D^q\to C_*$ is entirely determined by a choice of elements $x\in C_{p+q}$, $y,z\in C_{p+q-1}$ such that $d(x)=y+(-1)^pz$ and $d(y)=(-1)^{p+1}d(z)$. Denote $f_{p,q}^x\coloneqq x$, $f_{p,q}^y\coloneqq y$, and $f_{p,q}^z\coloneqq z$.
    Suppose we are given chain maps $f_{p,q},g_{p,q}\colon D^p\otimes D^q\to C_*$ for all $p,q \geq 0$,
    such that for all $p\geq 1$: \[f_{p,q}\circ (\alpha_{p-1}\otimes 1)=g_{p,q} \circ(\alpha_{p-1}\otimes 1)\]
    where $\alpha_{p-1}\otimes 1\colon D^{p-1}\otimes D^q\to D^p\otimes D^q$ is defined from the inclusion $\alpha_{p-1}\colon D^{p-1}\to D^q$, see our choice in \cref{lemma: uniqueness of DoldKan Equivalence}. Then in that case, we obtain $f_{p,q}^y=g_{p,q}^y$. 
    Similarly, we obtain $f_{p,q}^z=g_{p,q}^z$ when the chain maps equalize after pre-composing with $1\otimes \alpha_{q-1}\colon D^p\otimes D^{q-1}\to D^p\otimes D^q$.

    We apply this discussion for the maps $\theta_{p,q}$ and $\zeta_{p,q}$.
    From the inductive definition of $\phi$ in \cref{lemma: uniqueness of DoldKan Equivalence}, we obtain:
    \[
    \theta_{p,q} \circ (\alpha_{p-1}\otimes 1) = F(\overline{\alpha}_{p-1}\otimes 1)\circ \theta_{p-1, q}, \quad \quad \zeta_{p,q}\circ (\alpha_{p-1}\otimes 1) = F(\overline{\alpha}_{p-1}\otimes 1) \circ \zeta_{p-1, q}
    \]
    where $F(\overline{\alpha}_{p-1}\otimes 1)\colon \m{F}_*(\mathbb{D}^{p-1}\otimes \mathbb{D}^q)\to \m{F}_*(\mathbb{D}^p\otimes \mathbb{D}^q)$, following notation of \cref{lemma: uniqueness of DoldKan Equivalence}.
    By induction, we know that $\theta_{p-1, q}=\zeta_{p-1, q}$. 
    Therefore:
    \[
    \theta_{p,q}\circ (\alpha_{p-1}\otimes 1)=\zeta_{p,q} \circ(\alpha_{p-1}\otimes 1).
    \]
    By our discussion, this implies $\theta_{p,q}^y=\zeta_{p,q}^y$. Similarly, we obtain $\theta_{p,q}^z=\zeta_{p,q}^z$.    
    It only remains to show that $\theta_{p,q}^x=\zeta_{p,q}^x$.
    Observe that given any chain complex $P_*$ and $Q_*$, we obtain the isomorphism:
    \[
    \norm_k(\Gamma(P_*)\otimes \Gamma(Q_*))\cong \bigoplus_{S^k_{p,q}} P_p\otimes Q_q
    \]
    where $S^k_{p,q}$ denotes the set of strictly increasing maps $[k]\to [p]\times [q]$ of partially ordered sets where composing with either projection is surjective (when $k=p+q$, this is the set of $(p,q)$-shuffles).
    Applying for $P=D^p$ and $Q=D^q$, we obtain in particular that the level $p+q+1$ of the chain complex $\norm_*(\mathbb{D}^p\otimes \mathbb{D}^q)\cong \m{F}_*(\mathbb{D}^p\otimes \mathbb{D}^q)$ is zero because there are no strictly increasing maps $[p+q+1]\to [i]\times [j]$ for $i\in \{p-1, p\}$ and $j\in \{q-1, q\}$.
    Moreover, as $D^p\otimes D^q$ is acyclic, then so is $\norm_*(\mathbb{D}^p\otimes \mathbb{D}^q)$ by the Eilenberg--Zilber theorem, thus the differential $d\colon \norm_{p+q}(\mathbb{D}^p\otimes \mathbb{D}^q)\to \norm_{p+q-1}(\mathbb{D}^p\otimes \mathbb{D}^q)$ must be injective.
    Notice we must have
    \[
d(\theta_{p,q}^x)=\theta_{p,q}^y+(-1)^p \theta_{p,q}^z = \zeta_{p,q}^y+(-1)^p\zeta_{p,q}^z=d(\zeta_{p,q}^x)
    \]
    where $d$ denotes the differential $\m{F}_{p+q}(\mathbb{D}^p\otimes \mathbb{D}^q)\to \m{F}_{p+q-1}(\mathbb{D}^p\otimes \mathbb{D}^q)$, which must also be injective, thus $\theta_{p,q}^x=\zeta_{p,q}^x$. 
    Thus $\theta_{p,q}=\zeta_{p,q}$ for all $p,q \geq 0$.
\end{proof}

\section{Eilenberg--Zilber monoidal structures}\label{section: EZ map on skeleton}

In this section, we apply the technical results from \cref{section: Day convolution and props} and \cref{section: promonoidal localization} to establish \cref{introtheorem: EZ structure,introtheorem: products of skeleta} from the introduction. 
This is accomplished by isolating specific criteria under which a filtered \category admits an $\nn$-promonoidal structure, which we apply to the special case of $\nnDelta$. 

\subsection{Pointwise and filtered tensor products}
Let $\m{C}$ be a small \category and let $\m{E}$ be a symmetric monoidal \category.
By \cite[2.1.3.4]{HA}, the category $\Fun(\m{C},\m{E})$ admits a symmetric monoidal structure, given by the coCartesian fibration
\[
\Fun(\m{C},\m{E})^{\otimes}_{\rm pt} = \Fun(\m{C},\m{E}^\otimes) \times_{\Fun(\m{C},\Fin_\ast)} \Fin_\ast \rightarrow \Fin_\ast
\]
whose associated symmetric monoidal product $F \otimes G \colon \m{C} \rightarrow \m{E}$, is given by $(F\otimes G)(c) \simeq F(c)\otimes G(c)$.  
In particular, the Cartesian symmetric monoidal structure on the \category of spaces induces a pointwise symmetric monoidal structure 
\[
\Fun(\m{C},\Spaces)^{\times} = \Fun(\m{C},\Spaces)^{\otimes}_{\rm pt} \rightarrow \Fin_\ast. 
\]
See \cite[2.4.1, 4.8.1]{HA} for more details. 

\begin{definition}\label{def: cartesian promonoidal structure}
    Let $\m{C}$ be a small \category. 
    We define the \textit{pointwise symmetric promonoidal structure on $\m{C}$}, denoted by
    \(
    \cC^{\rm pt} \rightarrow \Fin_{\ast}
    \)
    to be the symmetric promonoidal \category determined by the presentably symmetric monoidal \category, $\Fun(\m{C},\Spaces)^{\times}$, in virtue of  \cref{proposition: closed O-monoidal structures are O-promonoidal}.
    Furthermore,  by \cref{proposition: Day conv recovers presentably O-monoidal str}, we have an equivalence of presentably symmetric monoidal \categories 
    \[
    \Fun(\m{C},\Spaces)^{\times} \xrightarrow{\sim} \Fun(\m{C},\Spaces)^{\otimes}
    \] 
\end{definition}

\begin{remark}\label{remark: multimorphism spaces in pointwise promonoidal}
    It follows directly from the definitions that for any finite set $I$, we can identify 
    \[
    \Mul_{\m{C}}(\{c_i\}_{i\in I},-) \simeq \prod_{i\in I} \Map_{\m{C}}(c_i ,-) \colon \m{C} \rightarrow \Spaces
    \]
    and that this identification is natural in $\{c_i\}_{i\in I}$. 
    Furthermore, it is not hard to see that $\m{C}^{\rm pt}$ is equivalent to the coCartesian \operad, $\m{C}^{\sqcup}$, from \cite[2.4.3.3]{HA}, though we will not need this here. 
    In other words, the coCartesian \operad arises from the Cartesian symmetric monoidal structure on $\Fun(\m{C},\Spaces)$. 
\end{remark}

We now show that pointwise symmetric promonoidal structures recover the pointwise tensor product on $\Fun(\m{C},\m{E})$ for \textit{any} symmetric monoidal \category $\m{E}$. 

\begin{proposition}\label{prop: Day convolution gives pointwise products}
    Let $\m{C}$ be a small \category, let $\m{C}^{\rm pt}$ be the pointwise symmetric promonoidal structure on $\m{C}$, let $q\colon \m{E}^{\otimes} \rightarrow \Fin_{\ast}$ be a symmetric monoidal \category. 
    \begin{enumerate}
    \item For each map of finite sets $\phi \colon I \rightarrow \ast$, the \category $\Ar^{\phi}(\m{C})$ has a terminal object. 
    This implies that $\m{C}^{\rm pt}$ is symmetric promonoidally $\kappa$-small for all infinite regular cardinals $\kappa$
    \item The functor $\Fun(\m{C},\m{E})^{\otimes} \rightarrow \Fin_\ast$ is a coCartesian fibration whose symmetric monoidal product and unit are given by the formulae 
    \[
    (F\otimes G)(c) \simeq F(c)\otimes G(c), 
    \]
    and 
    \[
    \unit_{\Fun(\m{C},\m{E})}(c) \simeq \unit_{\m{E}}, 
    \]
    which are both natural in $c$.   
    \end{enumerate}
\end{proposition}

\begin{proof}
Let $\phi \colon I \rightarrow \ast$ be the unique map of finite sets. 
To establish the first claim, note that for any $c \in \m{C}$, the right fibration $\Ar^{\phi}(\m{C})_{c} \rightarrow \m{C}^{\times I}$ is equivalent to the right fibration $(\m{C}_{/c})^{\times I} \rightarrow \m{C}^{\times I}$, and that $(\m{C}_{/c})^{\times I}$ has a terminal object. 
Therefore, $\m{C}^{\rm pt}$ is symmetric promonoidally $\kappa$-small for all infinite regular cardinals $\kappa$. 
The second claim follows directly from \cref{lemma: Day conv is O-corep} and \cref{prop: existence of Day convolution product}. 
\end{proof}

\begin{proposition}
    Let $f\colon \m{C} \rightarrow \m{D}$ be a functor of small \categories. 
    Then, there is an induced lax symmetric promonoidal functor 
    \(
    f^{\rm pt}\colon \m{C}^{\rm pt} \rightarrow \m{D}^{\rm pt}. 
    \)
    Furthermore, for any symmetric monoidal \category $\m{E}^{\otimes}$, restriction along $f^{\rm pt}$ induces a symmetric monoidal functor 
    \[
    (f^{\rm pt})^{\ast} \colon \Fun(\m{D},\m{E})^{\otimes} \rightarrow \Fun(\m{C},\m{E})^{\otimes}
    \]
    which is given by $f^{\ast}\colon \Fun(\m{D},\m{E}) \rightarrow \Fun(\m{C},\m{E})$ on underlying \categories. 
\end{proposition}

\begin{proof}
    To construct $f^{\rm pt}$, first note that restriction along $f$ induces a symmetric monoidal functor 
    \[
    f^{\ast}\colon \Fun(\m{D},\Spaces)^\times \rightarrow \Fun(\m{C},\Spaces)^{\times}
    \]
    between presentable fibrations, which is a symmetric monoidal right adjoint, thereby admitting an oplax symmetric monoidal left adjoint given by $f_{!}$. 
    Forming the coCartesian dual, the induced functor 
    \[
    (f^{\ast})^{\vee} \colon \Fun(\m{D},\Spaces)^{\times,\vee} \rightarrow \Fun(\m{C},\Spaces)^{\times,\vee}
    \]
    remains symmetric monoidal but becomes a left adjoint, so that the functor 
    \[
    (f_{!})^{\vee} \colon \Fun(\m{C},\Spaces)^{\times,\vee} \rightarrow \Fun(\m{D},\Spaces)^{\times,\vee}
    \]
    is a lax symmetric monoidal right adjoint. 
    As $(f_{!})^{\vee}$ carries $\m{C}^{\rm pt}$ into $\m{D}^{\rm pt}$, this completes the construction of $f^{\rm pt}$. 
    It is clear that $(f^{\rm pt})^{\ast}$ is given by $f^{\ast}$ on underlying \categories, and the fact that $(f^{\rm pt})^{\ast}$ is symmetric follows by a check of the definitions.     
\end{proof}

We now recall results of a similar flavor, replacing $\m{C}^{\rm pt}$ with the symmetric monoidal \category determined by the commutative monoid $(\nn,+)$. 

\begin{definition}
    Let $\nn^{\otimes} \rightarrow \Fin_{\ast}$ denote the symmetric monoidal \category associated to the commutative monoid $(\nn,+)$.
    Explicitly, $\nn^{\otimes}$ can be identified with the nerve of the $1$-category described as follows.
    \begin{enumerate}
        \item An object is a tuple $\{a_{i}\}_{i\in I}$ for $I$ a finite set.
        \item A morphism $\{a_{i}\}_{i\in I} \rightarrow \{b_{j}\}_{j\in J}$ is defined to be a map of finite pointed sets $\alpha \colon I_+ \rightarrow J_+$ plus the relations $\sum_{i\in \alpha^{-1}\{j\}} a_{i} \leq b_{j}$ for each $j \in J$.
        \item The composition rule is the obvious one inherited from $\Fin_{\ast}$ and $\nn$. 
    \end{enumerate}
\end{definition}

\begin{remark}\label{rem: whay is an N-monoidal cat}
    Unwinding the definitions, an $\nn$-monoidal \category $\m{C}^{\otimes} \rightarrow \nn^{\otimes}$ encodes the following structure:
    \begin{enumerate}
        \item An $\nn$-shaped diagram in $\Cat_\infty$, 
        \[
        \m{C}_{0} \rightarrow \m{C}_{1}\rightarrow \m{C}_{2} \rightarrow \cdots 
        \]
        which we refer to as a \textit{filtered \category}. 
        \item A distinguished object $\ast \xrightarrow{\unit} \m{C}_{0}$, which we call the unit. 
        \item For any multimorphism $\phi \colon \{a_i\}_{i\in I} \rightarrow b$ in $\nn^{\otimes}$ (which amounts to the relation $\sum_{i} a_i \leq b$), there is a functor 
        \[
        \bigotimes^{\phi}_{i\in I} \colon \prod_{i \in I} \m{C}_{a_i}\rightarrow  \m{C}_{b}. 
        \]
        \item A huge amount of coherence data, encoding diagrams involving each of the functors $\bigotimes^{\phi}_{i\in I}$ which express symmetry, associativity, and unitality. 
        For instance, the following diagram of \categories is required to commute
        \[
        \begin{tikzcd}
        \m{C}_{a} \times \m{C}_{b} \times \m{C}_{c} \arrow[r] \arrow[d] & \m{C}_{a} \times \m{C}_{b+c} \arrow[d] \\
        \m{C}_{a+b} \times \m{C}_{c} \arrow[r] & \m{C}_{a+b+c}
        \end{tikzcd}
        \]
        for all $a,b,c \in \nn$. 
    \end{enumerate}
\end{remark}

We also have the following well-known result, whose proof we include for completeness. 

\begin{proposition}\label{proposition: filtered objects omega small}
    The \category $\nn^{\otimes}$ is promonoidally $\omega$-small, so that whenever $\m{E}^{\otimes}$ is is a finitely cocomplete symmetric monoidal \category, $\Fun(\nn,\m{E})^{\otimes} \rightarrow \Fin_\ast$ is a symmetric monoidal \category. 
    Furthermore, we have
    \[
    (F\otimes G)(c) \simeq \colim_{a+b\leq c} F(a)\otimes G(b), 
    \]
    and 
    \[
    \unit_{\Fun(\nn,\m{E})}(c) \simeq \unit_{\m{E}}, 
    \]
    naturally in $c$. 
\end{proposition}

\begin{proof}
    It is easy to see that $\nn^{\otimes}$ is promonoidally $\omega$-small, as the \category $\Ar^{\phi}(\nn)_{b}$ is finite for any map of finite sets $\phi\colon I \rightarrow \ast$ and for any $b \in \nn$.
    In fact, for each $b \in \nn$, we have an equivalence of \categories 
    \[
    \Ar^{\phi}(\nn)_{b} \simeq \left(\prod_{i\in I} \nn \right)\times_{\nn} \nn_{/b}, 
    \]
    where $\prod_{i\in I} \nn \rightarrow \nn$ is given by addition of natural numbers. 
    As in the case of \cref{prop: Day convolution gives pointwise products}, the formulae above follow from from \cref{lemma: Day conv is O-corep} and \cref{prop: existence of Day convolution product}. 
\end{proof}

\subsection{Skeletally promonoidal filtrations}

\begin{definition}
    Let $\m{C}$ be a small \category equipped with an $\nn$-indexed filtration by full subcategories
    \[
    \m{C}_{\sleq{\ast}} = \m{C}_{\sleq{0}} \rightarrow \m{C}_{\sleq{1}} \rightarrow \cdots \rightarrow \m{C}_{\sleq{a}} \rightarrow \cdots 
    \]
    such that $\colim_{a \in \nn} \m{C}_{\sleq{a}} \simeq \m{C}$. 
    For each $a \in \nn$,
    we let $i_{a} \colon \m{C}_{\sleq{a}}
    \rightarrow \m{C}$ denote the inclusion. 
    Define 
    \[
    \nn\ltimes \m{C}_{\sleq \ast} \subseteq \nn\times \m{C}
    \]
    to be the full subcategory spanned by objects $(a,c)$ where $c \in \m{C}_{\sleq{a}}$. 
    It is straightforward to see that the projection onto $\nn$ yields a coCartesian fibration 
    \(
    p\colon \nn\ltimes\m{C}_{\sleq \ast} \rightarrow \nn
    \)
    classified by the functor $\m{C}_{\sleq{\ast}}\colon \nn \rightarrow \Cat_\infty$. 
    We denote the fiberwise groupoid completion of $p$ by 
    \[
    \nn\ltimes\gpd{\m{C}_{\sleq{\ast}}} = (\nn\ltimes\m{C}_{\sleq \ast})[W_{p}^{-1}] \rightarrow \nn.
    \]
\end{definition}

\begin{definition}
    Let $(\nn\times \m{C})^{\otimes}$ denote the $\nn$-monoidal \category determined by the pullback diagram
    \[
    \begin{tikzcd}
    (\nn\times \m{C})^{\otimes} \arrow[r] \arrow[d] & \m{C}_{\rm pt}^{\otimes} \arrow[d] \\
    \nn^{\otimes} \arrow[r] & \Fin_{\ast}
    \end{tikzcd}
    \]
    We define $(\nn\ltimes\m{C}_{\sleq{\ast}})^{\otimes}$ to be the full suboperad of $(\nn\times\m{C})^{\otimes}$ spanned by the objects of $\nn\ltimes\m{C}_{\sleq{\ast}}$. 
\end{definition}

\begin{remark}\label{remark: formula for Mul}
    Let $I \in \Fin$ and let $\{(a_i,c_i)\}_{i\in I} \in (\nn\times\m{C})^{\otimes}_{I}$.
    It follows directly from the definition that there is an identification of functors $\nn\times \m{C} \rightarrow \Spaces$
    \[
    \Mul_{\nn\times\m{C}}(\{(a_i,c_i)\}_{i\in I},-) \simeq \Map_{\nn}(\sum_{i} a_i,-) \times \prod_{i\in I} \Map_{\m{C}}(c_i ,-).
    \]
    Consequently, given a multimorphism, $\phi\colon \{a_i\}_{i\in I} \rightarrow b$, in $\nn^{\otimes}$, we have an equivalence of functors $\m{C} \rightarrow \Spaces$
    \[
    \Mul^{\phi}_{\nn\times\m{C}}(\{(a_i,c_i)\}_{i\in I},(b,-)) \simeq \Map_{\nn}(\sum_{i} a_i,b) \times \prod_{i\in I} \Map_{\m{C}}(c_i ,-) \simeq \prod_{i\in I} \Map_{\m{C}}(c_i ,-).
    \]
    Therefore, given $\{(a_i,c_i)\}_{i\in I} \in (\nn\ltimes\m{C}_{\sleq{\ast}})^{\otimes}_{I} \simeq \prod_{i\in I}\m{C}_{\sleq{a_i}}$, the $\phi$-multimorphism space
    \[
    \Mul^{\phi}_{\nn\ltimes\m{C}_{\sleq{\ast}}}(\{(a_i,c_i)\}_{i\in I},(b,-))\colon \m{C}_{\sleq{b}} \rightarrow \Spaces
    \]
    can be computed as 
    \[
    \Mul^{\phi}_{\nn\ltimes\m{C}_{\sleq{\ast}}}(\{(a_i,c_i)\}_{i\in I},(b,-)) \simeq i_{b}^{\ast}\left(\prod_{i\in I} \Map_{\m{C}}(c_i , -)\right).
    \]    
\end{remark}

\begin{lemma}\label{lemma: existence of EZ promonoidal structure}
    The map of \operads 
    \((\nn\ltimes\m{C}_{\sleq{\ast}})^{\otimes} \rightarrow \nn^{\otimes}
    \)
    exhibits the source as $\nn$-promonoidal, if and only if, for each active morphism $\phi\colon \{a_i\}_{i\in I} \rightarrow b$ in $\nn^{\otimes}$ and for each $(c_i)_{i\in I} \in \prod_{i\in I} \m{C}_{\sleq{a_i}}$, the functor 
    \[
    \prod_{i\in I} \Map_{\m{C}}(c_i,-) \colon \m{C} \rightarrow \Spaces
    \]
    is left Kan extended from $\m{C}_{\sleq{b}}$.
\end{lemma}

\begin{proof}
    By \cref{proposition: sub operad of promonoidal is promonoidal} and \cref{remark: formula for Mul}, the map of \operads $(\nn\ltimes\m{C}_{\sleq{\ast}})^{\otimes} \rightarrow \nn^{\otimes}$ determines an $\nn$-promonoidal structure, if and only if, for each active morphism $\phi\colon \{a_i\}_{i\in I} \rightarrow b$ in $\nn^{\otimes}$ and for each $(c_i)_{i\in I} \in \prod_{i\in I} \m{C}_{\sleq{a_i}}$ the canonical transformation
    \[
    (i_{b})_{!}i_{b}^{\ast}\left(\prod_{i\in I} \Map_{\m{C}}(c_i,-)\right) \rightarrow \prod_{i\in I} \Map_{\m{C}}(c_i,-)
    \]
    is a natural equivalence, whence the claim. 
\end{proof}

\begin{definition}
    Whenever $\m{C}_{\sleq{\ast}}$ satisfies the hypothesis of \cref{lemma: existence of EZ promonoidal structure}, we call the resulting $\nn$-promonoidal \category
    \((\nn\ltimes\m{C}_{\sleq{\ast}})^{\otimes} \rightarrow \nn^{\otimes}
    \)
    the \textit{Eilenberg--Zilber structure on $\m{C}_{\sleq{\ast}}$}.
\end{definition}

\begin{lemma}
    Let $\m{C}_{\sleq{\ast}}$ be a filtration satisfying the hypotheses of \cref{lemma: existence of EZ promonoidal structure}. 
    Further assume that for each multimorphism $\phi\colon \{a_i\}_{i\in I} \rightarrow b$ in $\nn^{\otimes}$, and for every tuple $(f_i)_{i\in I} \in \prod_{i\in I}\Map_{\m{C}_{\sleq{a_i}}}(c_i,c'_{i})$, the induced transformation of functors $\m{C}_{\sleq{b}} \rightarrow \Spaces$
    \[
    \prod_{i\in I} \Map_{\m{C}}(c'_i,-) \rightarrow \prod_{i\in I} \Map_{\m{C}}(c_{i},-)
    \]
    is an equivalence after realization. 
    Then, the $\nn$-promonoidal groupoid of the Eilenberg--Zilber structure, which we denote by $(\nn\ltimes\gpd{\m{C}_{\sleq{\ast}}})^{\otimes}$, exists, and there is an $\nn$-promonoidal functor 
    \[
    (\nn\ltimes \m{C}_{\sleq{\ast}})^{\otimes}\rightarrow (\nn\ltimes\gpd{\m{C}_{\sleq{\ast}}})^{\otimes}
    \]
    extending the fiberwise groupoid completion functor.   
\end{lemma}

\begin{proof}
    It is enough to verify that $(\nn\ltimes\m{C}_{\sleq{\ast}})^{\otimes}$ and the saturated collection of arrows $W_{p}$ in $\nn\ltimes\m{C}_{\sleq{\ast}}$ satisfy the hypothesis of \cref{theorem: promonoidal structures exist}.
    However, by \cref{remark: promonoidal loc existence criteria} and \cref{remark: formula for Mul}, this will be the case exactly when the transformation 
    \[
    \prod_{i\in I} \Map_{\m{C}}(c'_i,-) \rightarrow \prod_{i\in I} \Map_{\m{C}}(c_{i},-)
    \]
    is an equivalence after left Kan extension along $\m{C}_{\leq b} \rightarrow \gpd{\m{C}_{\leq b}}$, i.e., after realization. 
\end{proof}

\subsection{The simplicial skeletal filtration} 
We now prove our main results from the introduction by carefully studying the category $\nnDelta$. 
Our analysis is powered by a simple fact about products of simplicial spaces, which we now explain. 

Let $\mathrm{PoSet}$ denote the $1$-category of partially ordered sets and monotone maps between them, of which $\Delta$ is a full subcategory; denote this inclusion by $\varepsilon \colon \Delta \rightarrow \PoSet$
While the category $\Delta$ does not admit finite products, the $1$-category $\PoSet$ does admit such limits, using the product poset structure. 

\begin{construction}\label{construction: product of simplicies}
    Given $[n_1],\dots,[n_{k}] \in \Delta$, let $P$ denote the product poset $\prod_{i} [n_i]$. 
    Define
    \[
    \Delta_{/P} = \Delta\times_{\PoSet} \PoSet_{/P},
    \]
    which can be identified with the full subcategory of $\PoSet_{/P}$ spanned by maps of posets of the form $[m]\rightarrow P$. 
    This admits a further full subcategory 
    \[
    \Delta^{\mathrm{inj}}_{/P} \hookrightarrow \Delta_{/P}
    \]
    consisting of injective monotone functions $\sigma\colon[m] \rightarrow P$ (equivalently, \textit{strictly} monotone functions). 
    This yields a functor 
    \[
    \Delta^{\mathrm{inj}}_{/P} \rightarrow \Delta \rightarrow \Fun(\Delta^{\op},\Spaces)
    \]
    which we denote by $\sigma \mapsto \Delta^{\sigma}$. 
    By definition, we obtain a natural transformation 
    \[
    \colim_{\sigma \in \Delta^{\mathrm{inj}}_{/P}} \Delta^{\sigma} \rightarrow \varepsilon^{\ast} \Map_{\PoSet}(-,P),
    \]
    and as $\PoSet$ admits finite products, we have a canonical equivalence 
    \[
    \varepsilon^{\ast} \Map_{\PoSet}(-,P) \simeq \prod_{i} \Delta^{n_i}.
    \]
\end{construction}

\begin{lemma}\label{lemma: product of simplices is skeletal}
    The natural transformation of simplicial spaces 
    \[
    \colim_{\sigma \in \Delta^{\mathrm{inj}}_{/P}} \Delta^{\sigma} \rightarrow  \prod_{i} \Delta^{n_i}. 
    \]
    from \cref{construction: product of simplicies} is an equivalence.     
    Consequently, the simplicial space $\prod_{i} \Delta^{n_i}$ is $\sum_{i}n_i$-skeletal. 
\end{lemma}

\begin{proof}
    To prove the claim, we show that the induced map 
    \[
    \Map_{\m{P}(\Delta)}(\prod_{i\in I}\Delta^{n_i},X) \rightarrow \Map_{\m{P}(\Delta)}(\colim_{\sigma \in \Delta^{\mathrm{inj}}_{/P}} \Delta^{\sigma}, X)
    \]
    is an equivalence for all simplicial presheaves $X$. 
    Under the identification 
    \[
    \varepsilon^{\ast}\Map_{\PoSet}(-,P) \simeq \prod_{i\in I}\Delta^{n_i}
    \]
    it is enough to prove the canonical map 
    \(
    \varepsilon_{\ast}X(P) \rightarrow \lim_{\sigma \in (\Delta^{\mathrm{inj}}_{/P})^{\op}} X(\sigma) 
    \)
    is an equivalence for all $X$, where $\varepsilon_{\ast}$ denotes right Kan extension along $\varepsilon$. 
    Let $\Delta^{\op}_{P/} = \Delta^{\op}\times_{\PoSet^{\op}}(\PoSet^{\op})_{P/}$, so by the formula for right Kan extension we have
    \[
    \varepsilon_{\ast}X(P) \simeq \lim_{\tau \in \Delta^{\op}_{P/}} X(\tau). 
    \]
    Therefore, it will be enough to prove that
    \(
    \Delta^{\mathrm{inj}}_{/P} \rightarrow \Delta_{/P}
    \)
    is cofinal, as we have $(\Delta^{\op}_{P/})^{\op} \simeq \Delta_{/P}$.
    To this end, fix $\tau \colon [k] \rightarrow P$ in $\Delta_{/P}$ and consider the product category 
    \[
    (\Delta^{\mathrm{inj}}_{/P})_{\tau/} = \Delta^{\mathrm{inj}}_{/P} \times_{\Delta_{/P}}(\Delta_{/P})_{\tau/}. 
    \]
    The map of posets $\tau\colon [k] \rightarrow P$ can be uniquely factored as $[k] \rightarrow [\ell] \hookrightarrow P$, so that $[\ell]$ is isomorphic to the subposet given by ${\rm im}(\tau)$. 
    We see that $[\ell] \hookrightarrow P$ determines an initial object of $(\Delta^{\mathrm{inj}}_{/P})_{\tau/}$, meaning this category has contractible realization. 
    By Quillen's Theorem A for \categories \cite[4.1.3.1]{HTT}, the claim follows. 
\end{proof}

\begin{remark}
    One can even do a little bit better in \cref{lemma: product of simplices is skeletal}. 
    Writing $N = \sum_{i} n_{i}$, the subcategory of strictly monotone maps $[m] \rightarrow P$, where $N-1 \leq m \leq N$ is also cofinal. 
    In the case where $P = [n] \times [m]$, this cofinal subcategory amounts to finding all the nondegenerate $(n+m)$-simplices, and how they are glued together. 
\end{remark}

Recall that by \cref{prop: Day convolution gives pointwise products}, the pointwise promonoidal structure $(\Delta^{\op})^{\otimes}_{\rm pt}$ classifies the pointwise Cartesian product on $\Fun(\Delta^{\op},\Spaces)$, so that 
\[
\Fun(\Delta^{\op},\Spaces)^{\times} \simeq \Fun(\Delta^{\op},\Spaces)^{\otimes},
\]
where the latter presentably symmetric monoidal \category is the one arising from Day convolution. 
Due to the combinatorial nature of $\Delta^{\op}$, there is a concrete formula for the multimorphism spaces of $(\Delta^{\op})_{\rm pt}^{\otimes}$.

\begin{remark}\label{remark: shuffles and multimorphisms}
    For $I$ a finite set, we have by \cref{lemma: product of simplices is skeletal}
    \[
    \Mul_{\Delta^\op}(\{[n_i]\}_{i\in I},[m]) \simeq \Map_{\m{P}(\Delta)}(\Delta^{m},\prod_{i\in I}\Delta^{n_i}) \simeq \Map_{\mathrm{PoSet}}([m],\prod_{i\in I}[n_i]).     
    \]
    By the discussion in \cref{lemma: product of simplices is skeletal}, each morphism $[m] \rightarrow \prod_{i\in I}[n_i]$ factors through a strictly monotone map $[n] \rightarrow \prod_{i\in I}[n_i]$ which yields a nondegenerate simplex in the product. 
    Therefore, any multimorphism in $(\Delta^{\op})_{\rm pt}^{\otimes}$ is ``generated'' by one of these nondegenerate simplices. 
\end{remark}

Similarly, we can give a concrete formula for the multimorphism spaces of $(\nnDelta)^{\otimes}$. 

\begin{remark}
    Therefore, given $\{[n_i]\}_{i\in I} \in \prod_{i} \Delta^{\op}_{\sleq{a_i}}$, the $\phi$-multimorphism space
    \[
    \Mul^{\phi}_{\nnDelta}(\{[n_i])\}_{i\in I},-)\colon \Delta^{\op}_{\sleq{b}} \rightarrow \Spaces
    \]
    is given by  
    \[
    \Mul^{\phi}_{\nnDelta}(\{[n_i])\}_{i\in I},-) \simeq i_{b}^{\ast}\left(\prod_{i\in I} \Map_{\Delta^{\op}}([n_i], -)\right).
    \]    
\end{remark}

\begin{proposition}\label{proposition: nnDelta main prop}
    Let $\iota^{\otimes} \colon (\nnDelta)^{\otimes} \rightarrow (\nn\times \Delta^{\op})^{\otimes}$ denote the inclusion of \operads. 
    \begin{enumerate}    
        \item The commutative diagram 
        \[
        \begin{tikzcd}
        (\nnDelta)^{\otimes} \arrow[rr] \arrow[dr] & & (\nn\times \Delta^{\op})^{\otimes} \arrow[dl] \\
        & \nn^{\otimes} &
        \end{tikzcd}
        \]
        exhibits $(\nnDelta)^{\otimes}$ as a $\nn$-promonoidal subcategory of $(\nn\times\Delta^\op)^{\otimes}$. 
        \item The $\nn$-promonoidal groupoid completion of $(\nnDelta)^{\otimes}$ exists, and there is an $\nn$-promonoidal functor 
        \[
        (\nnDelta)^{\otimes} \rightarrow (\nn\ltimes\gpd{\Delta^{\op}_{\sleq{\ast}}})^{\otimes}.     
        \]
        Furthermore, the induced map 
        \((\nn\ltimes\gpd{\Delta^{\op}_{\sleq{\ast}}})^{\otimes} \rightarrow \nn^{\otimes}
        \)
        is an equivalence of $\nn$-promonoidal \categories. 
        \item The $\nn$-promonoidal \categories, $(\nnDelta)^{\otimes}$ and $(\nn\ltimes\gpd{\Delta^{\op}_{\sleq{\ast}}})^{\otimes}$, are $\nn$-promonoidally finite.
        Furthermore, when viewed as symmetric promonoidal \categories, they are also symmetric promonoidally finite. 
    \end{enumerate}
\end{proposition}

\begin{proof}
    The first claim follows by combining \cref{lemma: existence of EZ promonoidal structure} with \cref{lemma: product of simplices is skeletal} and \cref{remark: shuffles and multimorphisms}. 
    To prove the second claim, let $\phi\colon \{a_i\}_{i\in I} \rightarrow b$ in $\nn^{\otimes}$, and let $(f_i)_{i\in I} \in \prod_{i\in I}\Map_{\Delta^{\op}_{\sleq{a_i}}}([n_i],[m_i])$. 
    These maps induce a map of functors $\Delta^{\op}_{\sleq{b}} \rightarrow \Spaces$
    \[
    \prod_{i\in I} \Map_{\Delta^{\op}}([m_i],-) \rightarrow \prod_{i\in I} \Map_{\Delta^{\op}}([n_i],-). 
    \]
    However, as both the source and target are $b$-skeletal by \cref{lemma: product of simplices is skeletal}, this natural transformation realizes to a map between contractible spaces, which must be an equivalence.  
    That the map to $\nn^{\otimes}$ is an equivalence of $\nn$-promonoidal \categories follows from \cref{example: fiberwise gpd of nnDelta}. 
    To establish the third claim, fix a multimorphism $\phi\colon\{a_i\}_{i\in I} \rightarrow b$ in $\nn^{\otimes}$, let $[n] \in \Delta^{\op}_{\sleq{b}}$, and consider the \category 
    \(
    \Ar^{\phi}(\nnDelta)_{(b,[n])}.
    \)
    Unwinding the definitions, this is equivalent to 
    \[
    \prod_{i\in I} \Delta_{\sleq{a_i}}^{\op}\times_{\Deltaleq{b}^{\op}} (\Delta_{\sleq{b}}^{\op})_{/[n]},
    \]
    hence a finite \category. 
    The claim for the groupoid completion follows from the fact $\nn^{\otimes}$ is symmetric promonoidally small \cref{proposition: filtered objects omega small}, together with the $\nn$-promonoidal equivalence
    \((\nn\ltimes\gpd{\Delta^{\op}_{\sleq{\ast}}})^{\otimes} \rightarrow \nn^{\otimes}
    \)
    established above.
\end{proof}

\begin{theorem}\label{theorem: products of skeletal objects are skeletal}
    Let $\m{E}$ be a finitely cocomplete symmetric monoidal \category, and let $\m{E}^{\otimes}_{\nn} = \nn^{\otimes} \times_{\Fin_\ast}\m{E}^{\otimes}$.
    Then 
    \[
    \Fun_{\nn}(\nnDelta,\m{E}_{\nn})^{\otimes} \ \  \text{and} \ \ \Fun_{\nn}(\nn\ltimes \gpd{\Deltaleq{\ast}^{\op}},\m{E}_{\nn})^{\otimes}
    \]
    are $\nn$-monoidal \categories.
    Furthermore, left Kan extension along $\nn$-monoidal groupoid completion induces an $\nn$-monoidal functor 
    \[
    \Fun_{\nn}(\nnDelta,\m{E}_{\nn})^{\otimes} \rightarrow \Fun_{\nn}(\nn\ltimes \gpd{\Deltaleq{\ast}^{\op}},\m{E}_{\nn})^{\otimes}.
    \]
\end{theorem}

\begin{proof}
    The theorem follows from \cref{proposition: nnDelta main prop} combined with \cref{prop: existence of Day convolution product} and \cref{prop: left kan extension along strong promonoidal is strong monoidal}. 
\end{proof}

\begin{remark}
    The $\nn$-monoidal \category, $\Fun_{\nn}(\nnDelta,\m{E}_{\nn})^{\otimes}$, endows
    \[
    \Fun(\Deltaleq{\ast}^{\op},\m{E})\colon \nn \rightarrow \Cat_\infty
    \]
    with a lax symmetric monoidal structure. 
    Informally, this lax symmetric monoidal structure expresses that for any natural numbers $a,b,c \in \nn$, with $a+b\leq c$,  there is an essentially unique factorization of the pointwise tensor product on $\Fun(\Delta^{\op},\m{E})$:
    \[\begin{tikzcd}
	{\Fun(\Deltaleq{a}^{\op},\m{E})\times \Fun(\Deltaleq{b}^{\op},\m{E})} & [1em] {\Fun(\Deltaleq{c}^{\op},\m{E})} \\
	{\Fun(\Delta^\op,\m{E})\times\Fun(\Delta^{\op},\m{E})} & {\Fun(\Delta^{\op},\m{E}).}
	\arrow["{\otimes_{a+b\leq c}}", dashed, from=1-1, to=1-2]
	\arrow[hook, from=1-1, to=2-1]
	\arrow[hook, from=1-2, to=2-2]
	\arrow["\otimes"', from=2-1, to=2-2]
    \end{tikzcd}\]
    Here, $\otimes_{a+b\leq c}$ denotes the functor associated to the multimorphism $\{a,b\} \rightarrow c$ in $\nn^{\otimes}$. 
    In other words, the tensor product of an $a$-skeletal object with a $b$-skeletal object being $(a+b)$-skeletal is a completely generic phenomenon. 
\end{remark}

We can now deduce \cref{introtheorem: EZ structure} from the introduction. 

\begin{theorem}\label{theorem: main theorem on EZ for sk}
    Let $\m{E}$ be a finitely cocomplete symmetric monoidal \category. 
    The composition of lax symmetric monoidal functors
    \[
    \Fun(\Delta^{\op},\m{E})^{\otimes} \xrightarrow{q^{\ast}}\Fun(\nnDelta,\m{E})^{\otimes} \xrightarrow{q_{!}} \Fun(\nn\ltimes \gpd{\Deltaleq{\ast}^{\op}},\m{E})^{\otimes} \simeq \Fun(\nn,\m{E})^{\otimes}
    \]
    endows $\sk_{\ast}$ with a canonical lax symmetric monoidal structure. 
\end{theorem}

\begin{proof}
    The functor $q^{\ast}$ is lax symmetric monoidal by \cref{prop: left kan extension along strong promonoidal is strong monoidal}. 
    That the functor $p_{!}$ is lax symmetric monoidal follows from \cref{theorem: products of skeletal objects are skeletal} and \cref{lemma: sections and O-monoidal structures} by applying $\Fun(\nn,-)^{\otimes}$ to the $\nn$-monoidal functor 
    \[
    \Fun_{\nn}(\nnDelta,\m{E}_{\nn})^{\otimes} \rightarrow \Fun_{\nn}(\nn\ltimes \gpd{\Deltaleq{\ast}^{\op}},\m{E}_{\nn})^{\otimes}
    \]
    from \cref{theorem: products of skeletal objects are skeletal}.
\end{proof}

\begin{remark}
    The methods in this section apply directly to other kinds of generalized Reedy categories, such as $\Lambda^{\op}_{\leq \ast}$ and $\Lambda^{\op}_{\infty,\leq \ast}$, for example, which is the subject of ongoing work by the authors. 
\end{remark}

A characterizing property of the Eilenberg--Zilber map on chain complexes is that the map $\nabla\colon \norm_0(A)\otimes \norm_0(B)\to \norm_0(A\otimes B)$ is the identity, for any simplicial abelian groups $A$ and $B$. 
The following corollary is a categorification of this property; the skeletal filtration recovers the identity functor on $\m{E}$ upon restriction to the full subcategory of $0$-skeletal simplicial objects. 

\begin{proposition}\label{proposition: sk is unital}
    Let $\m{E}$ be a finitely cocomplete symmetric monoidal \category, and let 
    \[
    (i_{0})_{!} \colon \m{E} \rightarrow \Fun(\Delta^{\op},\m{E}) \ \ \text{and} \ \ (j_{0})_{!} \colon \m{E} \rightarrow \Fun(\nn,\m{E}) 
    \]
    denote the functors of left Kan extension along $i_{0}\colon \Delta^{\op}_{\sleq 0} \rightarrow \Delta^{\op}$ and $j_{0}\colon \{0\} \rightarrow \nn$. 
    Both of these functors are fully faithful and symmetric monoidal, and there is a canonical commutative diagram of symmetric monoidal \categories 
    \[
    \begin{tikzcd}
    & \m{E}^{\otimes} \arrow[dl,"(i_{0})_{!}"'] \arrow[dr,"(j_{0})_{!}"] &  \\
    \Fun(\Delta^{\op},\m{E})^{\otimes} \arrow[rr,"\sk_\ast"'] & & \Fun(\nn,\m{E})^{\otimes} 
    \end{tikzcd}
    \]   
    Consequently, $\sk_{\ast}$ is unital, and, for each $X,Y \in \Fun(\Delta^{\op},\m{E})$, the Eilenberg--Zilber map induces an equivalence 
    \[
    \left(\sk_{\ast}(X)\otimes \sk_{\ast}(Y)\right)_{0} \xrightarrow{\sim} \sk_{0}(X\otimes Y). 
    \]
\end{proposition}

\begin{proof}
    Consider the following commutative diagram of promonoidal \categories and lax promonoidal functors: 
    \[\begin{tikzcd}
	{(\Delta^{\op}_{\sleq 0})^{\otimes}_{\rm pt}} & {(\Delta^{\op}_{\sleq 0})^{\otimes}_{\rm pt}} & {(\gpd{\Delta^{\op}_{\sleq 0}})^{\otimes}_{\rm pt} } \\
	{(\Delta^{\op})^{\otimes}_{\rm pt}} & {(\nnDelta)^{\otimes}} & {\nn^{\otimes}}
	\arrow[from=1-1, to=2-1]
	\arrow[equals, from=1-2, to=1-1]
	\arrow["\sim", from=1-2, to=1-3]
	\arrow[from=1-2, to=2-2]
	\arrow[from=1-3, to=2-3]
	\arrow[from=2-2, to=2-1]
	\arrow[from=2-2, to=2-3]
    \end{tikzcd}\]
    By applying $\Fun(-,\m{E})^{\otimes}$, we obtain a commutative diagram of symmetric monoidal \categories and lax symmetric monoidal functors:
    \[\begin{tikzcd}
	{\Fun(\Delta^{\op}_{\sleq 0},\m{E})^{\otimes}} & {\Fun(\Delta^{\op}_{\sleq 0},\m{E})^{\otimes}} & {\Fun(\gpd{\Delta^{\op}_{\sleq 0}},\m{E})^{\otimes}} \\
	{\Fun(\Delta^{\op},\m{E})^{\otimes}} & {\Fun(\nnDelta,\m{E})^{\otimes}} & {\Fun(\nn^{\otimes},\m{E})^{\otimes}}
	\arrow[equals, from=1-2, to=1-1]
	\arrow["\sim"', from=1-3, to=1-2]
	\arrow[from=2-1, to=1-1]
	\arrow[from=2-1, to=2-2]
	\arrow[from=2-2, to=1-2]
	\arrow[from=2-3, to=1-3]
	\arrow[from=2-3, to=2-2]
    \end{tikzcd}\]
    in virtue of \cref{prop: existence of Day convolution product} and \cref{prop: left kan extension along strong promonoidal is strong monoidal}.
    Again by \cref{prop: left kan extension along strong promonoidal is strong monoidal}, the lax monoidal functors in the rightmost square all admit symmetric monoidal left adjoints. 
    To complete the proof, we note that the square 
    \[\begin{tikzcd}
	{(\Delta^{\op}_{\sleq{0}})^{\otimes}_{\rm pt}} & {(\Delta^{\op}_{\sleq{0}})^{\otimes}_{\rm pt}} \\
	{(\Delta^{\op})^{\otimes}_{\rm pt}} & {(\nnDelta)^{\otimes}}
	\arrow["\sim", from=1-1, to=1-2]
	\arrow[from=1-1, to=2-1]
	\arrow[from=1-2, to=2-2]
	\arrow[from=2-1, to=2-2]
    \end{tikzcd}\]
    is a pullback, and that the vertical maps in this square are promonoidal functors. 
    Combining these observations, we obtain a commutative diagram 
    \[\begin{tikzcd}
	{\Fun(\Delta^{\op}_{\sleq 0},\m{E})^{\otimes}} && {\Fun(\gpd{\Delta^{\op}_{\sleq 0}},\m{E})^{\otimes}} \\
	{\Fun(\Delta^{\op},\m{E})^{\otimes}} && {\Fun(\nn^{\otimes},\m{E})^{\otimes}}
	\arrow["{\mathrm{id}}", from=1-1, to=1-3]
	\arrow["{(i_{0})_{!}}"', hook, from=1-1, to=2-1]
	\arrow["{(j_{0})_{!}}"', hook, from=1-3, to=2-3]
	\arrow["{\sk_{\ast}}"', from=2-1, to=2-3]
    \end{tikzcd}\]
    as desired. 
    The remaining claims follow immediately from the symmetric monoidal equivalence of  symmetric monoidal functors $(j_{0})_{!} \simeq \sk_{\ast} \circ (i_{0})_{!}$. 
\end{proof}

\begin{proposition}
    The skeletal filtration $\sk_\ast \colon \Fun(\Delta^{\op},\Sp) \rightarrow \Fun(\nn,\Sp)$ induces a lax symmetric monoidal functor
    \[
    \sk_{\ast}^{\heartsuit}\colon \Fun(\Delta^{\op},\mathrm{Ab}) \simeq \Fun(\Delta^{\op},\Sp)^{\heartsuit} \rightarrow \Fun(\nn,\Sp)^{\heartsuit} \simeq \mathrm{Ch}_{\geq 0}(\mathrm{Ab}),
    \]
    which is an equivalence of categories.  
    Furthermore, for any simplicial abelian groups $A,B$, the induced map of chain complexes 
    \(
    \sk_{\ast}^{\heartsuit}(A) \otimes \sk^{\heartsuit}_{\ast}(B) \rightarrow \sk_{\ast}^{\heartsuit}(A\otimes B)
    \)
    is an isomorphism in degree $0$.
    Consequently, there is a canonical lax symmetric monoidal equivalence of lax symmetric monoidal functors 
    \(
    \sk_{\ast}^{\heartsuit} \simeq \norm_{\ast}.
    \)
\end{proposition}

\begin{proof}
    As both the pointwise and Beilinson t-structures are compatible with the relevant symmetric monoidal structures, the pointwise tensor product and Day convolution product endow the hearts with symmetric monoidal structures \cite[6.5, 6.14]{stefano}; we denote these by $\Fun(\Delta^{\op},\Sp)^{\heartsuit,\otimes}$ and $\Fun(\nn,\Sp)^{\heartsuit,\otimes}$. 
    Furthermore, by \cref{lemma: skeletal filtration is t-exact for beilinson}, $\sk_\ast$ restricts to a lax symmetric monoidal functor 
    \[
    \sk_{\ast}^{\heartsuit}\colon  \Fun(\Delta^{\op},\Sp)^{\heartsuit,\otimes} \rightarrow \Fun(\nn,\Sp)^{\heartsuit,\otimes}.
    \]
    That the map 
    $
    \sk_{\ast}^{\heartsuit}(A) \otimes \sk^{\heartsuit}_{\ast}(B) \rightarrow \sk_{\ast}^{\heartsuit}(A\otimes B)
    $
    is an isomorphism follows directly from \cref{proposition: sk is unital}, and the lax symmetric monoidal equivalence follows from \cref{prop:unicity of classical EZ}.
\end{proof}

In her thesis \cite[II.3.2]{hedenlund-thesis}, Hedenlund discusses joint work with Krause and Nikolaus, which endows $\mathrm{SSEQ}$, the $1$-category of bigraded spectral sequences of abelian groups, with the structure of an \operad; the key point is that \textit{pairings} of spectral sequences comprise the multimorphisms for the \operad structure. 
Furthermore, in \cite[Theorem II.3.5]{hedenlund-thesis}, it is shown that the construction which takes a filtered spectrum to its associated spectral sequence can be refined to a map of \operads:
\[
\mathrm{Fil}(\Sp)^{\otimes} = \Fun(\zz^{\op},\Sp)^\otimes \rightarrow \mathrm{SSEQ}^{\otimes}. 
\]
Here, we follow the convention in \cite{hedenlund-thesis} so that the $E_{1}$-page of the associated spectral sequence is the $E_{2}$-page of the spectral sequence considered in \cite[1.2.2]{HA}. 
This sort of reindexing convention goes back to work of Maunder on the Atiyah--Hirzebruch spectral sequence \cite{maunder}. 
It is the understanding of the authors, based on \cite{hedenlund-thesis}, and the summary  in \cite{BM-poitou}, that these results will appear in forthcoming work by Hedenlund--Krause--Nikolaus \cite{hedenlund-krause-nikolaus}. 
These results neatly package all the relevant structure involved in manipulating spectral sequences, so that certain well-known facts (which tend to be tediously checked by hand) can be deduced immediately by formal argumentation. 
For example, it is well-known that differentials in the spectral sequence associated to a filtered $\mathbb{E}_\infty$-ring spectrum satisfy the Leibniz rule, and this property holds as soon as there is a map of \operads $\Fin_\ast \rightarrow {\rm SSEQ}^{\otimes}$. 
Combining \cref{introtheorem: EZ structure} with the discussion above, we immediately obtain the following result.  

\begin{corollary}
    There is a map of \operads 
    \(\Fun(\Delta^{\op},\Sp)^{\otimes} \rightarrow\mathrm{SSEQ}^{\otimes}. 
    \)
    Consequently, for any $0 \leq n \leq \infty$, a simplicial $\mathbb{E}_{n}$-ring spectrum determines an $\mathbb{E}_{n}$-algebra in $\mathrm{SSEQ}^{\otimes}$.
    In particular, the differentials in the spectral sequence associated to a simplicial $\ee{\infty}$-ring spectrum satisfy the Leibniz rule. 
\end{corollary}

\begin{example}
    Let $\mathrm{B}^{\mathrm{cyc}}(A)$ denote the cyclic bar construction of an $\mathbb{E}_{\infty}$-ring spectrum which is an $\mathbb{E}_{\infty}$-algebra in $\Fun(\Delta^{\op},\Sp)$; consequently, $\sk_\ast \mathrm{B}^{\mathrm{cyc}}(A)$ is an $\ee{\infty}$-algebra in $\nn$-filtered spectra. 
    As the base-change functor $\mathbb{F}_{p} \otimes(-) \colon \Sp \rightarrow \Mod_{\mathbb{F}_{p}}$ is symmetric monoidal and colimit-preserving
    \(
    \mathbb{F}_{p} \otimes \sk_\ast \mathrm{B}^{\mathrm{cyc}}(A)
    \)
    is a filtered $\ee{\infty}$-$\mathbb{F}_{p}$-algebra. 
    Therefore, the B\"{o}kstedt spectral sequence defines a commutative algebra in $\mathrm{SSEQ}^{\otimes}$. 
    While the results of our work are able to recover this aforementioned multiplicative structure, the B\"{o}kstedt spectral sequence carries a great deal more structure, such as a $\sigma$-operator and a Hopf algebra structure; see \cite{angeltveit-rognes} for details. 
    The $\sigma$-operator is in fact a shadow of a cyclic analogue of the skeletal filtration, which encodes how the circle action shifts the filtration degree of the (simplicial) skeletal filtration; see \cite[Section 2]{malkiewich-cyclotomic} for a nice account of the cyclic skeletal filtration.  
    Categorified versions of these results are the subject of forthcoming work by the authors.
\end{example}

\appendix
\section{Base-change and Kan extensions}\label{section: base-change and Kan extensions}
In this appendix, we explain two technical results \cref{proposition: exchange transformation} and \cref{lem: projection formula} which are critically used in \cref{section: promonoidal cat} and \cref{section: Day convolution and props}. 
Both of these are well-known, but we were unable to find the specific generality of \cref{proposition: exchange transformation} in the literature. 
While \cref{lem: projection formula} does appear in the literature, we have opted to include a proof it in this appendix for the sake of completeness.

\subsection{Adjointability}
Given a pullback square of \categories 
\begin{equation*}
\begin{tikzcd}
	X \pull & Y \\
	S & T
	\arrow["\psi", from=1-1, to=1-2]
	\arrow["p"', from=1-1, to=2-1]
	\arrow["q", from=1-2, to=2-2]
	\arrow["\varphi"', from=2-1, to=2-2]
\end{tikzcd}
\end{equation*}
and an \category $\m{C}$, we have an induced commutative square, 
\begin{equation}\label{equation: left adjointable square}
    \begin{tikzcd}
	{\Fun(T,\m{C})} & {\Fun(S,\m{C})} \\
	{\Fun(Y,\m{C})} & {\Fun(X,\m{C})}
	\arrow["{\varphi^{\ast}}", from=1-1, to=1-2]
	\arrow["{q^{\ast}}"', from=1-1, to=2-1]
	\arrow["{p^{\ast}}", from=1-2, to=2-2]
	\arrow["{\psi^{\ast}}"', from=2-1, to=2-2]
\end{tikzcd}
\end{equation}
which we denote by $\sigma$. 
In many situations, the functors $\varphi^{\ast}$ and $\psi^{\ast}$ admit left adjoints, denoted by $\varphi_{!}$ and $\psi_{!}$, respectively.
The (co)units of these adjunctions combined with $\sigma$ induce an \textit{exchange transformation}\footnote{This is also called a Beck--Chevalley transformation.} 
\[
\r{Ex}_{\sigma} \colon \psi_{!}p^{\ast} \rightarrow q^{\ast}\varphi_{!}.
\]
In the situation where $\r{Ex}_{\sigma}$ is an equivalence, the square $\sigma$ is said to be \textit{left adjointable} (\cite[4.7.4.13]{HA}). 
If $\m{C}$ is presentable and $q$ is a coCartesian fibration, it is well-known that $\sigma$ is a left adjointable square. 
In fact, the same result is true when $q$ is only locally coCartesian \cite[Lemma 3.2.6]{ayala-francis-fibrations}.

Alternatively, if $\varphi^{\ast}$ and $\psi^{\ast}$ admit right adjoints, there is an entirely dual story in which $\sigma$ is \textit{right adjointable} whenever $q$ is a locally Cartesian fibration. 
Even more generally, Clausen--Jansen prove that the square $\sigma$ is right adjointable when $q$ is a \textit{proper functor} \cite[Definition 2.22, Theorem 2.27]{clausen-jansen}; in particular, locally Cartesian fibrations between \categories are always proper.

The main result of this appendix is to identify a class of pullback diagrams which induce left adjointable squares of the form \cref{equation: left adjointable square}. 
We now recall a well-known fact to be used in the the course of our arguments, the proof of which appears a \cite[Lemma 3.2.6]{ayala-francis-fibrations}.

\begin{lemma}\label{appendix A: lemma: local fib adjoints}
    Let $p\colon X \rightarrow S$ be a functor of \categories. 
    \begin{enumerate}
        \item If $p$ is a locally coCartesian fibration, then the inclusion 
        \(
        X_{s} \rightarrow X_{/s} = X\times_{S} S_{/s}
        \)
        is a right adjoint. 
        If $(x,s' \xrightarrow{\alpha} s) \in X_{/s}$, the value of this right adjoint is $\alpha_{!}(x) \in X_{s}$, where $\alpha_{!} \colon X_{s'} \rightarrow X_{s}$ denotes the covariant transport of $\alpha$. 
        \item If $p$ is a locally Cartesian fibration, then the inclusion 
        \(
        X_{s} \rightarrow X_{s/} = X\times_{S} S_{s/}
        \)
        is a left adjoint. 
        If $(x,s \xrightarrow{\alpha} s') \in X_{/s}$, the value of this right adjoint is equivalent to $\alpha_{\ast}(x) \in X_{s}$, where $\alpha_{\ast} \colon X_{s} \rightarrow X_{s'}$ denotes the functor induced by $\alpha$. 
    \end{enumerate}
\end{lemma}

\begin{lemma}\label{lemma: pullback of right fibrations}
Let
\begin{equation*}
\begin{tikzcd}
	X \pull & Y \\
	S & T
	\arrow["\psi", from=1-1, to=1-2]
	\arrow["p"', from=1-1, to=2-1]
	\arrow["q", from=1-2, to=2-2]
	\arrow["\varphi"', from=2-1, to=2-2]
\end{tikzcd}
\end{equation*}
be a pullback square of \categories where $q$ is a locally Cartesian fibration. 
Then, for any $y \in Y$, the following induced map is cofinal:
\[
(\varphi,\psi) \colon X \pullb_{Y} Y_{/y} \longrightarrow S \pullb_{T} T_{/q(y)}.
\] 
\end{lemma}

\begin{proof}
As above, we will write $X_{/y} = X \pullb_{Y} Y_{/y}$ and $S_{/q(y)} = S \pullb_{T} T_{/q(y)}$.
By Joyal's \categorical Quillen's Theorem A \cite[4.1.3.1]{HTT}, we need to verify that for any object $(s,\varphi(s)\xrightarrow{\alpha} q(y)) \in S_{/q(y)}$ the geometric realization of the \category 
\[
(X_{/y})_{(s,\varphi(s)\rightarrow q(y))/} = X_{/y} \pullb_{S_{/q(y)}} (S_{/q(y)})_{(s,\varphi(s)\rightarrow q(y))/}
\]
is weakly contractible. 
Unwinding the definitions, we have equivalences of \categories 
\[
(X_{/y})_{(s,\varphi(s)\rightarrow q(y))/} \simeq X_{s/} \pullb_{Y_{\varphi(s)/}} (Y_{\varphi(s)/})_{/(y,\varphi(s) \rightarrow q(y))}
\]
and 
\[
(Y_{\varphi(s)})_{/\alpha_{\ast}(y)} \xrightarrow{\sim} X_{s} \pullb_{Y_{\varphi(s)}} (Y_{\varphi(s)})_{/\alpha_{\ast}(y)}, 
\]
from which we can conclude the inclusion 
\[
(Y_{\varphi(s)})_{/\alpha_{\ast}(y)} \rightarrow (X_{/y})_{(s,\varphi(s)\rightarrow q(y))/} 
\]
is a left adjoint by \cref{appendix A: lemma: local fib adjoints}; as the source has a terminal object, the claim follows. 
\end{proof}

\begin{proposition}\label{proposition: exchange transformation}
Let 
\begin{equation*}
\begin{tikzcd}
	X \pull & Y \\
	S & T
	\arrow["\psi", from=1-1, to=1-2]
	\arrow["p"', from=1-1, to=2-1]
	\arrow["q", from=1-2, to=2-2]
	\arrow["\varphi"', from=2-1, to=2-2]
\end{tikzcd}
\end{equation*}
be a pullback square of \categories with $q$ a locally Cartesian fibration
We denote this square by $\sigma$. 
Let $\m{C}$ be an \category which which admits $S\times_{T}T_{/t}$-indexed colimits for all $t \in T$. 
Then, the induced square
\[\begin{tikzcd}
	{\Fun(T,\m{C})} & {\Fun(S,\m{C})} \\
	{\Fun(Y,\m{C})} & {\Fun(X,\m{C})}
	\arrow["{\varphi^{\ast}}", from=1-1, to=1-2]
	\arrow["{q^{\ast}}"', from=1-1, to=2-1]
	\arrow["{p^{\ast}}", from=1-2, to=2-2]
	\arrow["{\psi^{\ast}}"', from=2-1, to=2-2]
\end{tikzcd}\]
is left adjointable. 
That is, $\varphi^{\ast}$ and $\psi^{\ast}$ admit left adjoints, denoted by $\varphi_{!}$ and $\psi_{!}$, respectively, and the following exchange transformation is an equivalence:
\[
\r{Ex}_{\sigma} \colon \psi_{!}p^{\ast} \rightarrow q^{\ast} \varphi_{!}.
\]
\end{proposition}

\begin{proof}
Our assumption about the existence of $S\times_{T}T_{/t}$-shaped colimits in $\m{C}$ guarantees the existence of $\varphi_{!}$ and $\psi_{!}$. 
To see that $\r{Ex}_{\sigma}$ is an equivalence, note that for any $F \colon S \rightarrow \m{C}$ and any $y \in Y$, the induced map 
\[
(\psi_{!}p^{\ast}F)(y) \simeq \colim_{X\times_{Y}Y_{/y}} (F \circ p)|_{X\times_{Y}Y_{/y}} \rightarrow \colim_{S\times_{T}T_{/q(y)}} F|_{S\times_{T}T_{/q(y)}} \simeq (q^{\ast}\varphi_{!}F)(y),
\]
given by colimit interchange, is an equivalence by \cref{lemma: pullback of right fibrations}.
\end{proof}

\begin{remark}
    After writing the arguments above, we learned an alternative proof of \cref{proposition: exchange transformation} from Peter Haine in the case where $\m{C}$ is presentable.     
    The square $\sigma$ in \cref{proposition: exchange transformation} is the transposition of a square to which we can apply \cite[Theorem 2.27]{clausen-jansen}.
    By \cite[4.7.4.14]{HA} the left adjointability of $\sigma$ is tantamount to the right adjointability of the aforementioned transposed square.     
\end{remark}

\subsection{A projection formula for coends}
One interesting consequence of the preceding results is a well-known projection formula for weighted colimits. 
Before stating this result, we recall some of the basic theory of weighted colimits. 

\begin{definition}
    Let $W \colon S^\op \rightarrow \Spaces$ and let $F \colon S \rightarrow \m{C}$ be functors of \categories, where $W$ is classified by the right fibration $q \colon X \rightarrow S$.
    Then, the \textit{$W$-weighted colimit of $F$}, if it exists, is defined as
    \[
    \colim_{S,W} F = \colim_{X} F \circ q.
    \]
\end{definition}

In favorable cases, such as when $\m{C}$ is a presentable \category, weighted colimits can be described in terms of coends. 
If $\m{C}$ admits colimits indexed by spaces, we let $X \odot C$ denote the colimit of the constant diagram $X \rightarrow \{C\} \rightarrow \m{C}$. 
Additionally, if $\m{C}$ is presentable, we will use the same notation for the tensor of $\m{C}$ over the \category of spaces. 
The following lemma is standard. 

\begin{lemma}\label{lemma: coend formula for weighted colimit}
    Let $W \colon S^\op \rightarrow \Spaces$ and let $F \colon S \rightarrow \m{C}$ be functors of \categories, where $W$ is classified by the right fibration $q \colon X \rightarrow T$. 
    Furthermore, assume that $\m{C}$ admits colimits indexed by spaces. 
    Then, if the $W$-weighted colimit of $F$ exists, there is a natural equivalence
    \[
    \colim_{S,W} F \simeq \int^{s \in S} W(s) \odot F(s),  
    \]
    where $\int^{s \in S} W(s) \odot F(s)$ is the coend of the functor $W \odot F \colon S^{\op} \times S \rightarrow \m{C}$. 
\end{lemma}

\begin{proof}
    The dual of this claim is proved in \cite[Proposition 4.3]{haugseng-coends}, for example.
\end{proof}

\begin{lemma}\label{lem: projection formula}
    Let $S$ and $T$ be small \categories and let $\m{C}$ be a presentable \category. 
    Let $F \colon S \rightarrow \m{C}$, $\varphi\colon S \rightarrow T$, and $W \colon T^{\op} \rightarrow \Spaces$ be functors, where $W$ is classified by the right fibration $q \colon Y \rightarrow T$.
    Then, we have a natural equivalence 
    \[
    \int^{s\in S} (\varphi^{\op})^{\ast}W(s) \odot F(s) \simeq \colim_{S,W\circ \varphi^{\op}} F \stackrel{\simeq}\longrightarrow \colim_{T,W} \varphi_{!}F \simeq \int^{t\in T} W(t) \odot \varphi_{!}F(t).    
    \]
\end{lemma}

\begin{proof}
    This follows immediately from \cref{proposition: exchange transformation} by observing that the right fibration $S\times_{T} Y \rightarrow S$, where the pullback is taken along $\varphi$, classifies the functor $W \circ \varphi^{\op}$.
\end{proof}

\section{The Day convolution on a monoidal model category}\label{sec: Appendix Day convo agree}

We show here that the derived Day convolution product in a monoidal model category agrees with the $\infty$-categorical Day convolution reviewed in \cref{section: Day convolution and props}.

First, we show that the notions of Day convolutions in ordinary categories  and in $\infty$-categories agree. 
Recall from \cite[\href{https://kerodon.net/tag/002Z}{Proposition 002Z}]{kerodon} that for $\cC$ a small category and $\cD$ a category, we get an equivalence of $\infty$-categories:
\[
\ \r{N}(\Fun(\cC, \cD))\simeq\Fun(\r{N}(\cC), \r{N}(\cD)),
\]
and in fact, it is an isomorphism of simplicial sets.
Recall that if $\cD$ is a symmetric monoidal category, the nerve on its operator category defines an $\infty$-category $\r{N}(\cD^\otimes)$ whose underlying $\infty$-category is precisely $\r{N}(\cD)$, see \cite[2.1.2.21]{HA}.
Recall we proved a similar result if $\cD$ is replaced by a promonoidal category $\cC$, see \cref{cor: promonoidal nerve}.
On one hand, we can define the Day convolution monoidal structure on the ordinary category $\Fun(\cC, \cD)$ as in \cref{construction: Day convolution in ordinary categories}, and denote its category of operators  $\Fun(\cC, \cD)^\Day$, on the other hand we can also define the Day convolution on the $\infty$-category $\Fun(\r{N}(\cC), \r{N}(\cD))$ as in \cref{section: Day convolution and props}.

\begin{lemma}\label{lem: Day convolution underived}
    Let $\cC$ be a small promonoidal category.
    Let $\cD$ be a finitely cocomplete symmetric monoidal category.
    Then we obtain an equivalence of symmetric monoidal \categories:
    \[
    \r{N} \big( \Fun(\cC, \cD)^\Day\big) \simeq \Fun\left(\r{N}(\cC), \r{N}(\cD)\right)^\Day.
    \]
    In other words, we have the equivalence $\Fun(\r{N}(\cC), \r{N}(\cD))_\r{Day}\simeq \r{N}(\Fun(\cC, \cD)_\r{Day})$.
\end{lemma}

\begin{proof}
As $\r{h}\r{Fun}\left(\r{N}(\cC), \r{N}(\cD)\right)\simeq \Fun(\cC, \cD)$ as categories, we can extend this equivalence to the operator categories, using that the convolution products of functors $\r{N}(\cC)\rightarrow \r{N}(\cD)$ is equivalent to the convolution product of functors $\cC\rightarrow
 \cD$:
\[
\r{N}(F)\Day \r{N}(G)\simeq \r{N}(F\Day G)
\]
for all functors $F,G\colon \cC\rightarrow \cD$, using \cite[\href{https://kerodon.net/tag/02JD}{Example 02JD}]{kerodon}.
\end{proof}

We are now interested in the case where $\cD$ is a symmetric monoidal model category $\m{M}$, in the sense of \cite[4.2.6]{hovey}.
Provided that $\m{M}$ is cofibrantly generated, recall from \cite[11.6.1]{hir} that there is a model structure on the functor category $\Fun(\cC, \m{M})$ for any small category $\cC$, called the \textit{projective model structure}, with weak equivalences and fibrations are respectively levelwise weak equivalences and levelwise fibrations.
Recall a symmetric monoidal model category $\m{M}$ is said to be \textit{combinatorial} if it is cofibrantly generated and locally presentable (its monoidal product automatically preserves colimits in each variable).

\begin{lemma}\label{lem: Day convolution is monoidal model}
Let $\cC$ be a small promonoidal category. Let $\m{M}$ be a combinatorial symmetric monoidal model category. The functor category $\Fun(\cC, \m{M})$ is a symmetric monoidal model category when endowed with its projective model structure and Day convolution.
\end{lemma}

\begin{proof}
See \cite[2.18]{SmithPhD2021} for a proof, which is an adaptation of \cite[4.1]{Bataninn-Berger} and  \cite[5.6.35]{HHR21} in the promonoidal case.
\end{proof}

There is a variant of our previous result where we consider instead the injective model structure on $\Fun(\cC, \cM)$ where the cofibrations and acyclic cofibrations are defined levelwise, which always exists by \cite[3.4.1]{left}.

\begin{lemma}\label{lem: INJECTIVE Day convolution is monoidal model}
Let $\cC$ be a small promonoidal category. Let $\m{M}$ be a combinatorial symmetric monoidal model category. The functor category $\Fun(\cC, \m{M})$ is a symmetric monoidal model category when endowed with its injective model structure and Day convolution.
\end{lemma}

\begin{proof}
    Let $\alpha\colon F\to G$ and $\beta\colon S\to T$ be injective cofibrations in $\Fun(\cC, \cM)$.
    We need to show the induced map in $\Fun(\cC, \cM)$
    \[
    \alpha\square \beta\colon (F\Day T)\coprod_{F\Day S} (G\Day S)\longrightarrow G\Day T
    \]
    is an injective cofibration, that is acyclic whenever $\alpha$ or $\beta$ is acyclic. 
    Unwinding the definition of Day convolution product \cref{construction: Day convolution in ordinary categories}, by \cite[18.4.1]{hir}, it is enough to show the induced map in $\cM$
    \[
    (F(c_1)\otimes T(c_2))\coprod_{F(c_1)\otimes S(c_2)} (G(c_1)\otimes S(c_2)) \longrightarrow G(c_1)\otimes T(c_2)
    \]
    is a cofibration that is acyclic whenever $\alpha$ or $\beta$ is acyclic for all $c_1,c_2\in \cC$, but this follows immediately from the fact that $\cM$ is a monoidal model category.
\end{proof}

\begin{remark}
    Suppose $\cC$ is both promonoidal and a Reedy category, then it may not be true in general that $\Fun(\cC, \cM)$ is a monoidal model category when endowed with its Reedy model structure, see more details when $\cC$ is monoidal in \cite{BM11}. 
\end{remark}

Given a combinatorial symmetric monoidal model category $\m{M}$, with collection of weak equivalences $\m{W}$ on cofibrant objects $\m{M}^c$, we obtain a symmetric monoidal $\infty$-category $\mathrm{N}(\m{M}^c)\left[\m{W}^{-1}\right]$ obtained by Dwyer--Kan localization, see \cite[4.1.7.4]{HA}. Therefore, given any small promonoidal category $\cC$, we can consider the Day convolution on the $\infty$-category of $\cC$-diagrams: \[\Fun\left(\mathrm{N}(\cC),\mathrm{N}(\m{M}^{c})\left[\m{W}^{-1}\right]\right).\]
By \cite[1.3.4.25]{HA}, the above $\infty$-category is equivalent to the Dwyer-Kan localization of $\Fun(\cC, \m{M})$ when endowed with its projective model structure. Indeed, let us denote $\m{W}_\r{proj}$ the collection of weak equivalences on projective cofibrant functors $\Fun(\cC, \m{M})^c$. The natural pairing:
\begin{align*}
    \Fun(\cC, \m{M})\times \cC & \longrightarrow \m{M}\\
    (F, c) & \longmapsto F(c),
\end{align*}
induces an equivalence of $\infty$-categories:
\[
\varphi\colon \mathrm{N}\left( \Fun(\cC,\m{M})^{c}\right)[\m{W}^{-1}_\r{proj}] \stackrel{\simeq}\longrightarrow \Fun\left(\mathrm{N}(\cC),\mathrm{N}(\m{M}^{c})\left[\m{W}^{-1}\right]\right).
\]
By \cref{lem: Day convolution is monoidal model}, the Day convolution at the level of ordinary categories induces a symmetric monoidal structure on the $\infty$-category $\mathrm{N}\left( \Fun(\cC,\m{M})^{c}\right)[\m{W}^{-1}_\r{proj}]$. Therefore we obtain possibly two monoidal structures. The next result states they are equivalent.

\begin{theorem}\label{theorem: Day conv model structure agrees with oo-categorical Day conv}
Let $\m{M}$ be a combinatorial symmetric monoidal model category and $\cC$ a small promonoidal category. 
Let $\m{W}$ denote the collection of weak equivalences between cofibrant objects for the model structure on $\m{M}$ and let $\m{W}_\r{proj}$ denote the weak equivalences between cofibrant objects for the projective model structure on $\Fun(\cC,\m{M})$. 
There is 
an equivalence of symmetric monoidal \categories:
\[
\varphi^\Day\colon \mathrm{N}\left( {(\Fun(\cC,\m{M})^{c})}^\Day\right)[\m{W}^{-1}_\r{proj}] \xrightarrow{\simeq} \Fun(\mathrm{N}(\cC),\mathrm{N}(\m{M}^{c})[\m{W}^{-1}]))^{\Day},
\]
where both functor categories are equipped with the Day convolution.
\end{theorem}

\begin{proof}
Since projectively cofibrant functors are levelwise cofibrant \cite[11.6.3]{hir}, the equivalence of \cref{lem: Day convolution underived} induces a strong symmetric monoidal functor 
\[
\r{N}\left( (\Fun(\cC, \m{M})^c)^\Day\right) \longrightarrow \Fun\left( \r{N}(\cC), \r{N}(\m{M}^c)\right)^\Day.
\]   
The localization functor $\mathrm{N}(\m{M}^c)\rightarrow \mathrm{N}(\m{M}^c) [\m{W}^{-1}]$ is strong monoidal \cite[4.1.7.4]{HA}.
By \cref{lem: Dau convolution preserves maps of operads}, 
post-composing with this localization on the functor above yields a lax symmetric monoidal functor on the Day convolutions: 
\[
 \varphi^\Day\colon\r{N}\left( (\Fun(\cC, \m{M})^c)^\Day\right) \longrightarrow \Fun\left( \r{N}(\cC), \r{N}(\m{M}^c)[\m{W}^{-1}]\right)^\Day.
\]  
As projective weak equivalences are levelwise weak equivalences, the above functor sends morphisms in $\m{W}_\r{proj}$ to equivalences in $\Fun\left( \r{N}(\cC), \r{N}(\m{M}^c)[\m{W}^{-1}]\right)$.
By the universal property of the symmetric monoidal Dwyer--Kan localization \cite[A.5]{nikolaus-scholze}, the above functor uniquely determines a lax symmetric monoidal functor:
\[
\varphi^\Day\colon\mathrm{N}\left( {(\Fun(\cC,\m{M})^{c})}^\Day\right)[\m{W}^{-1}_\r{proj}] \longrightarrow \Fun(\mathrm{N}(\cC),\mathrm{N}(\m{M}^{c})[\m{W}^{-1}]))^{\Day}.
\]
Its underlying functor $\varphi\colon \r{N}(\Fun(\cC, \m{M})^c)[\m{W}^{-1}_\r{proj}]\rightarrow \Fun(\mathrm{N}(\cC),\mathrm{N}(\m{M}^{c})[\m{W}^{-1}])$ is precisely the equivalence of \cite[1.3.4.25]{HA}, as discussed above.
We are only left to show that $\varphi^\Day$ is not just lax symmetric monoidal, but in fact strong symmetric monoidal.
We shall denote the Day convolution in $\r{N}(\Fun(\cC, \m{M})^c)[\m{W}^{-1}_\r{proj}]$ by $\Day^\mathbb{L}$ with unit $\unit^\mathbb{L}$, while the Day convolution in $\Fun(\mathrm{N}(\cC),\mathrm{N}(\m{M}^{c})[\m{W}^{-1}])$ as $\Day$, with unit $\unit$.
By \cite[3.26]{cothhduality}, it is enough to show that for all projectively cofibrant functors $F,G\colon \cC\rightarrow \m{M}$, the natural maps
\[
\left(\varphi(F)\Day \varphi(G)\right)(c) \longrightarrow \varphi(F\Day^\mathbb{L} G)(c)
\] 
and $\unit(c)\longrightarrow \varphi(\unit^\mathbb{L})(c)$
are weak equivalences in $\m{M}$ for all $c\in \cC$.
By \cref{remark: Day convolution in ordinaty symmetric monoidal case} and \cref{prop: existence of Day convolution product}, we get directly that those maps are equivalences as colimits in $\mathrm{N}(\m{M}^{c})[\m{W}^{-1}]$ correspond precisely to homotopy colimits in $\m{M}$ by \cite[1.3.4.24]{HA}.
\end{proof}

\section{A model explicit construction of the Eilenberg--Zilber map}\label{section: appendix EZ model cat}

We prove in this appendix that the skeleton filtration $\sk_*^\cE$ of a presentable $\infty$-category $\cE$ can be described entirely using model categories. This allows us to construct and describe the Eilenberg--Zilber lax monoidal structure explicitly using point-set methods.

\subsection{The skeleton filtration of a bisimplicial set}
We show here that the skeleton filtration on space $\sk_*^\cS\colon\Fun(\Delta^\op, \cS)\to\Fun(\nn, \cS)$ for the $\infty$-category of spaces is lax symmetric monoidal, using model categories.
We begin to compare the notion of skeleton filtration in \categories  reviewed in \cref{const: skeleton functor and the I category} with the usual skeleton filtration in ordinary categories.

\begin{construction}\label{const: the usual n-skeleton in simplicial set}
Let $\cM$ be a finitely cocomplete category.  Let $\iota_n\colon \Delta_{\leq n}\hookrightarrow \Delta$ be the full subcategory as in \cref{const: skeleton functor and the I category}. 
The inclusion defines an adjunction:
\[
\begin{tikzcd}
\Fun(\Delta^\op_{\leq n}, \cM) \ar[bend left]{r}{(\iota_n)_!} \ar[phantom, description, xshift=0ex]{r}{\perp}& \ar[bend left]{l}{\iota_n^*}\Fun(\Delta^\op, \cM).
\end{tikzcd}
\]
The \textit{$n$-th skeleton functor} is the induced comonad on $\Fun(\Delta^\op, \cM)$:
\[
\begin{tikzcd}
\usk_n^\cM\colon \Fun(\Delta^\op, \cM) \ar{r}{\iota^*_n}& \Fun(\Delta^\op_{\leq n}, \cM) \ar{r}{(\iota_n)_!} & \Fun(\Delta^\op, \cM).
\end{tikzcd}
\]
Explicitly, given a simplicial object $X$ in $\cM$, the $k$-simplices of its $n$-th skeleton $\usk_n^\cM (X)$ are given as a colimit in $\cM$:
\[
\left(\usk_n^\cM (X)\right)_k= \int^{[i]\in \Delta^\op_{\leq n}} \Delta(k, i)\odot X_i\cong \colim_{\substack{ [k]\to [i]\\ i\leq n}} X_i.
\]
We denoted $-\odot-\colon \Set\times \cM\to\cM$ the tensoring of $\cM$ over the category $\Set$ of sets.
It extends to a functor $-\odot-\colon \sSet\times \cM\to \Fun(\Delta^\op, \cM)$ where given a simplicial set $S_\bullet$ and object $M$ in $\cM$, the simplicial object $S_\bullet\odot M$ is defined such that $(S_\bullet\odot M)_k=S_k\odot M$ in $\cM$.
Then we can write $\usk^M(X)$ as a coend in $\Fun(\Delta^\op, \cM)$:
\[
\usk_n^\cM(X)=\int^{[i]\in \Delta^\op_{\leq n}} \Delta^i\odot X_i.
\]
It is useful to think of $\usk_n^\cM (X)$ as the simplicial object generated by the simplices of $X$ of degree less or equal to $n$. 
Therefore, we obtain the classical filtration in $\cM$:
\[
X\cong \colim \left(\begin{tikzcd}
    \usk_0^\cM(X)\ar{r} & \usk_1^\cM(X) \ar{r} & \cdots \ar{r} & \usk_n^\cM(X) \ar{r} & \cdots
\end{tikzcd}\right)
\]
and a functor $\usk_*^\cM\colon \Fun(\Delta^\op, \cM)\rightarrow \Fun(\nn, \Fun(\Delta^\op, \cM))$.
\end{construction}

One can show that the functor $\usk^\cM_*$ is well-behaved homotopically for any combinatorial model structure on $\cM$.
We focus on the case where $\cM$ equals the  category $\sSet$ of simplicial sets, endowed with its Kan--Quillen model structure, in which the weak equivalences are the weak homotopy equivalences, fibrations are Kan fibrations, and cofibrations are the monomorphisms \cite[3.6.5]{hovey}. We shall denote $\usk^\sSet_*$ simply as $\usk_*$.

We denote $\ssSet=\Fun(\Delta^\op, \sSet)$ the category of bisimplicial sets, endowed with its Reedy model structure: a weak equivalence is a levelwise weak homotopy equivalence, and a Reedy cofibration is precisely a levelwise monomorphism \cite[15.8.7]{hir}. In particular, every object is cofibrant. This model structure is in fact equal to the injective model structure.

Given a combinatorial model category $\cM$, we denote $\Fil(\cM)=\Fun(\nn, \cM)$ the category of filtered objects in $\cM$ endowed with its projective model structure: a weak equivalence is a levelwise weak equivalence in $\cM$, and a fibration is a levelwise fibration in $\cM$ \cite[11.6.1]{hir}. A map $F_*\to G_*$ in $\Fil(\cM)$ is a cofibration if $F_0\to G_0$ is a cofibration in $\cM$, and the natural map
\(
F_{n}\coprod_{F_{n-1}}G_{n-1}\longrightarrow G_{n}
\)
is a cofibration in $\cM$ for all $n\geq 1$. Such a map is referred as a \textit{projective cofibration}. We shall consider both $\Fil(\sSet)$ and $\Fil(\ssSet)$ subsequently.

\begin{proposition}
    The skeleton filtration $\usk_*\colon \ssSet\to \Fil(\ssSet)$ on bisimplicial sets is a left Quillen functor. 
\end{proposition}

\begin{proof}
    We prove that $\usk_*$ preserves cofibrations and weak equivalences.
    The functor $\usk_n\colon \ssSet\to \ssSet$ is a left Quillen functor for all $n\geq 0$, see \cite[6.4]{BM11}. Therefore $\usk_n$ preserves weak equivalences between cofibrant objects, but every object is cofibrant in $\ssSet$.
    In particular, given a weak equivalence $X\to Y$ in $\ssSet$, then $\usk_n(X)\to \usk_n(Y)$ is a weak equivalence in $\ssSet$. Thus $\usk_*(X)\to \usk_*(Y)$ is a weak equivalence in $\Fil(\ssSet)$. 

    Suppose $f\colon X\to Y$ is a cofibration in $\ssSet$. Let us show $\usk_*(X)\to \usk_*(Y)$ is a projective cofibration. We know $\usk_n(X)\to \usk_n(Y)$ must also be a cofibration in $\ssSet$ since $\usk_n$ is left Quillen. In particular, $\usk_0(X)\to \usk_0(Y)$ is a cofibration in $\ssSet$.
    It remains to show that the natural map
    \[
    \usk_{n}(X)\coprod_{\usk_{n-1}(X)}\usk_{n-1}(Y)\longrightarrow \usk_{n}(Y)
    \]
    is a cofibration in $\ssSet$ for all $n\geq 1$.
    We view the bisimplicial sets $X$ and $Y$ as functors $\Delta^\op\times \Delta^\op\to \Set$ and we say a $(k,r)$-simplex $x\in X_{kr}$ is degenerate if it is a degenerate simplex in the simplicial set $X_{\bullet r}$.
    As pushouts in $\ssSet$ are computed levelwise, we check that the natural map:
    \[
    \usk_{n}(X)_{kr}\coprod_{\usk_{n-1}(X)_{kr}}\usk_{n-1}(Y)_{kr}\longrightarrow \usk_{n}(Y)_{kr}
    \]
     is an injective function for all $k,r\geq 0$.
     When $k\leq n-1$, the function is the identity and thus injective, so suppose $k\geq n$. 
     In this case, the function sends either a degenerate $(k, r)$-simplex $x$ of $X$ which is not in the image of a degeneracy from a $(\ell, r)$-simplex of $X$ for $\ell\leq n-1$ to its corresponding degenerate $(k,r)$-simplex $f(x)$ in $Y$; or it sends a degenerate $(k, r)$-simplex of $Y$ to itself, which is in the image of a degeneracy on a $(\ell, r)$-simplex of $Y$ for $\ell\leq n-1$. Therefore it is injective.
\end{proof}

Recall that the geometric realization $|-|\colon \ssSet\to \sSet$ of a bisimplicial set $X$ is defined as the coend in $\sSet$
\[
|X|\coloneq \int^{[k]\in \Delta^\op} \Delta^k\times X_k
\]
see \cite[15.11.1]{hir}. Recall also that the resulting simplicial set $|X|$ is isomorphic to its diagonal
\cite[15.11.6]{hir}, i.e.\ $|X|_n=X_{nn}$, where the faces and degeneracies are the composition vertical and horizontal faces and degeneracies of a bisimplicial set $X$, see precise definition in \cite[15.11.3]{hir}. 
The geometric realization is a left Quillen functor \cite[14.3.10]{Riehl} with our choice of model structures. 

It lifts to a functor $|-|\colon \Fil(\ssSet)\to \Fil(\sSet)$ which sends a filtered bisimplicial set $F_*$ to a filtered simplicial set $|F_*|$. This functor remains a left Quillen functor \cite[11.6.5]{hir}. 
This leads to a functor $|\usk_*|\colon \ssSet\to \Fil(\sSet)$ which provides a filtration of the geometric realization of any bisimplicial set $X$:
\[
\begin{tikzcd}
    {|}\usk_0(X) {|}\ar[hook]{r} &  {|}\usk_1(X) {|} \ar[hook]{r} & \cdots \ar{r} &  {|}\usk_n(X) {|} \ar[hook]{r} & \cdots \ar[hook]{r} & {|}X{|}.
\end{tikzcd}
\]
Our previous results leads to the following observation.

\begin{corollary}\label{cor: geom(usk) is left quillen}
    The functor $|\usk_*|\colon \ssSet\to \Fil(\sSet)$ is a left Quillen functor.
\end{corollary}

Recall we denoted $\cS$ the \category of spaces, which can be viewed as the Dwyer--Kan localization of $\sSet$.
By \cite[1.3.4.25]{HA}, the Dwyer--Kan localization of $\ssSet$ is precisely the \category $\Fun(\Delta^\op, \cS)$ of simplicial spaces, and the Dwyer--Kan localization of $\Fil(\sSet)$ is the \category $\Fun(\nn, \cS)$ of filtered spaces. 
Consequently, the total left derived functor of $|\usk_*|\colon \ssSet\to \Fil(\sSet)$ is precisely a functor $\Fun(\Delta^\op, \cS)\rightarrow\Fun(\zz_{\geq 0}, \cS)$ by \cite[1.5.1]{hindk}.

Unwinding the definition of $|\usk_n(X)|$, we see that we obtain the natural identification for a bisimplicial set $X$:
\[
|\usk_n(X)| \cong \int^{[i]\in \Delta_{\leq n}^\op} \Delta^i\times X_i
\]
Here we used the Yoneda lemma to identify $  \Delta^i\cong\int^{[k]\in \Delta^\op} \Delta(k,i)\odot \Delta^k$.
Notice $\Delta^\bullet\colon \Delta^\op\to \sSet$
    is a Reedy cofibrant object in the Reedy model structure of $\Fun(\Delta, \sSet)$, see \cite[15.9.11]{hir}.
    Moreover, $\Delta^i$ is a contractible simplicial set, for all $i\geq 0$.
    As $X$ must be Reedy cofibrant in $\ssSet$, we obtain natural weak homotopy equivalences in $\sSet$
    \[
|\usk_n(X)|\cong \int^{[i]\in \Delta^\op_{\leq n}} \Delta^i\times X_i\simeq \hocolim_{\Delta^\op_{\leq n}} \iota^*_n(X) \simeq\sk_n^\cS(X)
    \]
    by \cite[18.4.16]{hir}. 
    We record our observation in the following corollary.

\begin{corollary}
    The skeleton filtration $\sk_*^\cS\colon \Fun(\Delta^\op, \cS)\to \Fun(\nn, \cS)$ of \cref{const: skeleton functor and the I category} is equivalent to the total left derived functor of $|\usk_*|\colon \ssSet\to \Fil(\sSet)$.
    In particular, given a bisimplicial set $X$, we obtain a natural equivalence $\sk_*^\cS(X)\simeq |\usk_*(X)|$ in $\Fun(\nn, \cS)$.
\end{corollary}

Endow the category $\sSet$ of simplicial sets with its Cartesian monoidal structure, this is a monoidal model category \cite[4.2.8]{hovey}. 
The levelwise monoidal structure on $\ssSet=\Fun(\Delta^\op, \sSet)$ is also its Cartesian monoidal structure. Then $\ssSet$ is a monoidal model category with its Cartesian monoidal structure and its Reedy model structure (which is equal to its injective model structure). This is well-known but one can recover that result by \cref{lem: INJECTIVE Day convolution is monoidal model} and using the fact that the promonoidal structure on $\Delta^\op$ defines the pointwise monoidal structure by \cref{prop: Day convolution gives pointwise products}.

 Endow the filtered categories $\Fil(\sSet)$ and $\Fil(\ssSet)$ with their Day convolution: if $F_*$ and $G_*$ are filtered (bi)simplicial sets, then the Day convolution $F_*\Day G_*$ is the filtered (bi)simplicial set given by:
\[
(F_*\Day G_*)_n=\colim_{p+q=n} F_a\times G_b\cong \bigcup_{p+q=n}F_p\times G_q.
\]
using the formula appearing in \cref{remark: Day convolution in ordinaty symmetric monoidal case}. Here we used the fact that the colimit is computed in the category of sets, and therefore can be interpreted as union of (bi)simplicial sets. 
The monoidal unit are given by the constant filtered (bi)simplicial set induced by the terminal object.
It is well-known that these assemble to a monoidal model structure on $\Fil(\sSet)$ and $\Fil(\ssSet)$, but we can recover this result by \cref{lem: Day convolution is monoidal model}.

\begin{proposition}\label{prop: usk is lax mon}
    The functor $\usk_*\colon \ssSet\to \Fun(\nn, \ssSet)$ is lax symmetric monoidal.
\end{proposition}

\begin{proof}
    Given bisimplicial sets $X$ and $Y$, we build a natural map 
    \(
    \nabla\colon\usk_* X \Day \usk_* Y \longrightarrow \usk_* (X\times Y)
    \) in $\Fun(\nn, \ssSet)$.
    By definition of the Day convolution product,  we have:
    \begin{align*}
        \left( \usk_* X \Day \usk_* Y\right)_n \cong \bigcup_{p+q=n} \usk_pX \times \usk_qY
    \end{align*}
    for all $n \geq 0$.
    For $p+q=n$, notice that we have the isomorphisms of bisimplicial sets:
    \[
    \usk_n\left( \usk_pX\times \usk_q Y\right)\cong\usk_pX\times \usk_q Y.
    \]
    Indeed, we have the natural identifications
    \begin{align*}
    &\usk_n\left( \usk_pX\times \usk_q Y\right)\\
         & \cong \int^{[k]\in \Delta^\op_{\leq n}} \Delta^k \times \left( \left(\int^{[i]\in \Delta^\op_{\leq p}} \Delta(k, i)\odot X_i\right) \times \left( \int^{[j]\in \Delta^\op_{\leq q}}  \Delta(k,j)\odot Y_j\right)  \right) \\
        & \cong \int^{[k]\in \Delta^\op_{\leq n}} \int^{[i]\in \Delta^\op_{\leq p}}\int^{[j]\in \Delta^\op_{\leq q}}\Delta^k\times((\Delta(k, i)\odot X_i)\times (\Delta(k,j)\odot Y_j))\\
        & \cong \int^{[k]\in \Delta^\op_{\leq n}} \int^{[i]\in \Delta^\op_{\leq p}}\int^{[j]\in \Delta^\op_{\leq q}} (\Delta(k, i)\odot \Delta(k, j) \odot \Delta^k) \times (X_i\times Y_j)\\
        & \cong \int^{[i]\in \Delta^\op_{\leq p}}\int^{[j]\in \Delta^\op_{\leq q}} \usk^{\Set}_n(\Delta^i\times \Delta^j)\odot (X_i\times Y_j)\\
        & \cong \int^{[i]\in \Delta^\op_{\leq p}}\int^{[j]\in \Delta^\op_{\leq q}} (\Delta^i\times \Delta^j)\odot (X_i\times Y_j)\\
        & \cong \int^{[i]\in \Delta^\op_{\leq p}}\int^{[j]\in \Delta^\op_{\leq q}} (\Delta^i\odot X_i)\times (\Delta^j\odot Y_j)\\
        & \cong \left(\int^{[i]\in \Delta^\op_{\leq p}} \Delta^i\odot X_i\right) \times \left( \int^{[j]\in \Delta^\op_{\leq q}} \Delta^j\odot Y_j\right)\\
        & \cong \usk_p(X) \times \usk_q(Y).
    \end{align*}
Here, we denoted $\usk^\Set_n\colon \sSet\to \sSet$ of \cref{const: the usual n-skeleton in simplicial set}, which is not to be confused with $\usk_n\colon \ssSet\to \ssSet$, for which given a simplicial set $S_\bullet$, is defined as
\[
\usk^\Set_n(S_\bullet)= \int^{[k]\in\Delta^\op_{\leq n}} S_k\odot \Delta^k.
\]
As usual, if $S_{\bullet}$ has only non-degenerate simplices for degrees less or equal to $n$, then $\usk_n^\Set(S_\bullet)=S_\bullet$. 
Since the non-degenerate simplices of $\Delta^i\times \Delta^j$ are precisely the $(i, j)$-shuffles which are in degrees less or equal to $i+j$, and $i+j\leq p+q=n$, then $\usk_n^\Set(\Delta^i\times \Delta^j)=\Delta^i\times \Delta^j$.
    
   Therefore the desired map $\nabla\colon(\usk_* X \Day \usk_* Y)_n \longrightarrow \usk_n (X\times Y)$ is obtained as:
    \[
    \begin{tikzcd}
       \displaystyle \bigcup_{p+q=n} \usk_pX \times \usk_qY = \bigcup_{p+q=n} \usk_n\left(\usk_pX \times \usk_qY\right) \ar[hook]{r} & 
    \usk_n(X\times Y)
    \end{tikzcd}
    \]
    using the natural inclusions $\usk_p X \hookrightarrow X$  and $\usk_q Y\hookrightarrow Y$,  and then apply the folding map.

     Notice $\usk_*$ sends the terminal object of $\ssSet$ to the constant filtered object on $\Fil(\ssSet)$ induced by the terminal object of $\ssSet$, which is the monoidal unit of $\Fil(\ssSet)$.
    These observations determine the lax symmetric monoidal structure on $\usk_*$ and by construction the structure is indeed associative and unital.
\end{proof}

\begin{corollary}
    The functor $|\usk_*|\colon \ssSet\to \Fil(\sSet)$ is lax symmetric monoidal.
\end{corollary}

\begin{proof}
    The functor $|-|\colon \ssSet\to \sSet$ is strong symmetric monoidal as
    $
    (X\times Y)_{nn}=X_{nn}\times Y_{nn}
    $
    for any bisimplicial sets $X$ and $Y$, and all $n\geq 0$.
    Therefore $|-|\colon \Fil(\ssSet)\to\Fil(\sSet)$ remains strong monoidal by  \cref{lem: Dau convolution preserves maps of operads} and \cref{lem: Day convolution underived}.
    The result follows from \cref{prop: usk is lax mon}.
\end{proof}

\begin{corollary}\label{cor: sk is lax on spaces}
    The functor $\sk_*^\cS\colon \Fun(\Delta^\op, \cS)\to \Fun(\nn, \cS)$ is lax symmetric monoidal.
\end{corollary}

\begin{proof}
   By our previous results, we know that $|\usk_*|$ is lax symmetric monoidal and a left Quillen functor, and thus it induces functor on $\infty$-categories that is lax symmetric monoidal:
    \(
    \r{N}(|\usk_*|)\colon \r{N}\left( \ssSet \right) \longrightarrow \r{N}\left(\Fil(\sSet)^c\right).
    \)
    Here the upper subscript $c$ means projectively cofibrant objects in the filtered category of simplicial sets.
    Let $\mathscr{W}_{\nn}$ be the class of projective weak homotopy equivalences between projective cofibrant filtered simplicial sets.
    The localization \(\r{N}\left(\Fil(\sSet)\right) \longrightarrow \r{N}\left(\Fil(\sSet)^c\right) \left[ \mathscr{W}^{-1}_{\nn}\right]
    \)
    is strong symmetric monoidal \cite[4.1.7.4]{HA}.
    Recall $|\usk_*|$ preserves weak equivalences. 
    Thus by Hinich--Nikolaus--Scholze \cite[A.5 (v)]{nikolaus-scholze}, the induced lax symmetric monoidal functor
    \[
    \r{N}\left(\ssSet \right) \longrightarrow \r{N}\left(\Fil(\sSet)^c\right) \longrightarrow \r{N}\left(\Fil(\sSet)^c\right) \left[ \mathscr{W}^{-1}_{\nn}\right]
    \]
    uniquely determines a lax symmetric monoidal structure on the functor:
    \[
    \r{N}\left( \ssSet \right) \left[ \mathscr{W}^{-1}_\Delta\right]\longrightarrow\r{N}\left(\Fil(\sSet)^c\right) \left[ \mathscr{W}^{-1}_{\nn}\right]
    \]
    where $\mathscr{W}_\Delta$  denotes the class of weak equivalences in bisimplicial sets.
    The above functor is precisely the total left derived functor of $|\usk_*|$, i.e.\ $\sk_*^\cS\colon \Fun(\Delta^\op, \cS)\to \Fun(\nn, \cS)$.
    We can conclude as 
     \cref{theorem: Day conv model structure agrees with oo-categorical Day conv} guarantees that the monoidal structure obtained on the Dwyer--Kan localization of $\Fil(\sSet)$ is the usual Day convolution product on the $\infty$-category $\Fun(\nn, \cS)$.
\end{proof}

\subsection{The Eilenberg--Zilber map in the presentable case}
We wish to extend the lax symmetric monoidal structure on $\sk_*^\cS\colon \Fun(\Delta^\op, \cS)\to \Fun(\nn, \cS)$ to the skeleton functor $\sk^\cE_*\colon \Fun(\Delta^\op, \cE)\to \Fun(\nn, \cE)$ associated to any presentable \category $\cE$.

\begin{lemma}\label{lem: presheaf is tensoring}
    Let $\cE$ be a presentable $\infty$-category.
    Let $K$ be any simplicial set.
    Then there is a natural equivalence of $\infty$-categories:
    \(
    \cE\otimes \Fun(K, \cS)\simeq  \Fun(K, \cE).
    \)
\end{lemma}

\begin{proof}
    This follows from the universal property on the presheaf category $\Fun(K^\op, \cS)$ and the fact it is a dualizable object in $\PrL$ with dual $\Fun(K, \cS)$:
 \[
        \Fun(K, \cE)  \simeq \Fun^L(\Fun(K^\op, \cS), \cE)
         \simeq \Fun^L(\cS, \cE \otimes \Fun(K, \cS))
         \simeq \cE \otimes \Fun(K, \cS).
 \qedhere
 \]
\end{proof}

\begin{remark}\label{remark: description of the presheaf eq}
    Informally, the equivalence of \cref{lem: presheaf is tensoring} sends a pair $C\otimes F$ in $\cE\otimes \Fun(K, \cS)$ to a functor $F_C\colon K\rightarrow \cE$ where $F_C(x)=F(x)\odot C$, where $\odot\colon \cS\times \cE\rightarrow \cE$ denotes the tensoring on $\cE$ over spaces, and $x$ is a vertex of $K$. 
\end{remark}

\begin{corollary}\label{prop: skeleton is determined on spaces}
Let $\cE$ be a presentable \category.
Then the skeleton filtration is compatible with the equivalence of \cref{lem: presheaf is tensoring}, in the sense that the following diagram commutes:
\[
\begin{tikzcd}
    \Fun(\Delta^\op, \cE) \ar[leftarrow]{d}[swap]{\simeq} \ar{r}{\sk_*^\cE} & \Fun(\nn, \cE)\ar[leftarrow]{d}{\simeq}\\
    \cE\otimes \Fun(\Delta^\op, \cS) \ar{r}[swap]{\cE\otimes \sk_*^\cS} & \cE\otimes \Fun(\nn, \cE).
\end{tikzcd}
\]
In other words $\sk_*^\cE\simeq \cE\otimes \sk_*^\cS$
\end{corollary}

Suppose now that $\cE$ is a presentably symmetric monoidal $\infty$-category, i.e.\ an $\mathbb{E}_\infty$- algebra in $\PrL$. An algebra morphism in $\PrL$ is precisely a (colimit preserving) lax symmetric monoidal functor.
We can strengthen \cref{prop: skeleton is determined on spaces} as follows.

\begin{lemma}\label{lem: monoidal compatibility with presheaf and Day}
Let $\cE$ be a presentably symmetric monoidal $\infty$-category.
    Let $K$ be a promonoidal $\infty$-category.
    Then there is a natural equivalence of symmetric monoidal $\infty$-categories  (using Dav convolution products):
    \[
    \cE\otimes \Fun(K, \cS)\simeq  \Fun(K, \cE).
    \]
\end{lemma}

\begin{proof}
    The evaluation functor
$
    \mathrm{Ev}\colon \Fun(K, \cS)\times K\longrightarrow \cS
    $
    is lax promonoidal by the universal property of Day convolution \cref{remark: universal property of day convolution operad}.
    It induces a lax promonoidal functor
    \[
    \cE\times \mathrm{Ev}\colon \cE\times \Fun(K, \cE)\times K \longrightarrow \cE \times \cS
    \]
    The tensoring of $\cE$ over spaces provides a strong symmetric monoidal functor $\odot\colon \cE\times \cS\rightarrow \cE$ as $\cE$ is presentably symmetric monoidal.
    Thus we obtain a lax promonoidal functor:
    \[
    \begin{tikzcd}
        \cE\times \Fun(K, \cE)\times K \ar{r}{\cE\times \r{Ev}} & \cE\times \cS \ar{r}{\odot} & \cE.
    \end{tikzcd}
    \]
    Applying again the universal property of Day convolution, we obtain a lax symmetric monoidal functor
    \(
    \cE\times \Fun(K, \cS)\longrightarrow \Fun(K, \cE)
    \)
    with underlying functor informally described as $(C, F)\mapsto F_C$ as described in \cref{remark: description of the presheaf eq}.
    As it preserves colimits in each variable, by the universal property of the tensor product in $\PrL$, it uniquely determines a lax symmetric monoidal functor
    \(
    \cE\otimes \Fun(K, \cS)\longrightarrow \Fun(K, \cE)
    \)
    defined as $C\otimes F\mapsto F_C$, which is an equivalence of underlying \categories by \cref{lem: presheaf is tensoring}.
    Given any functor $F,G\colon K\rightarrow \cS$ and objects $C,C'\in \cE$, notice that, as tensoring of $\cE$ over $\cS$ is compatible with colimits and the monoidal product with $\cE$, we have:
    \begin{align*}
        (F\Day G)_{C\otimes C'}(x) & = (F\Day G)(x)\odot (C\otimes C')\\
        & \simeq \left(\int^{(x_1, x_2)\in K^{\times 2}} \Mul_K(\{x_1, x_2\}, x) \odot F(x_1) \otimes G(x_2)\right) \odot (C\otimes C')\\
        & \simeq \int^{(x_1, x_2)\in K^{\times 2}} \Mul_K(\{x_1, x_2\}, x) \odot (F(x_1) \odot C) \otimes (G(x_2) \odot C')\\
        & \simeq \int^{(x_1, x_2)\in K^{\times 2}} \Mul_K(\{x_1, x_2\}, x) \odot F_C(x_1)  \otimes G_{C'}(x_2)\\
        & \simeq (F_C \Day G_{C'})(x)
    \end{align*}
    for all objects $x$ in $K$.
    Moreover, if $\unit_\cE$ is the monoidal unit of $\cE$, notice that $*\odot \unit_\cE\simeq \unit_\cE$ by strong monoidality of $\odot\colon\cS\times \cE\rightarrow \cE$.
    If $\unit_{\cS^K}$ is the monoidal unit of $\Fun(K, \cS)$ and $\unit_{\cE^K}$ is the unit of $\Fun(K, \cE)$, then we obtain $(\unit_{\cS^K})_{\unit_\cE}\simeq \unit_{\cE^K}$ as desired.
    By \cite[3.26]{cothhduality}, this is enough to conclude that the lax symmetric monoidal structure on the equivalence $\cE\otimes \Fun(K, \cS)\longrightarrow \Fun(K, \cE)$ is actually strong symmetric monoidal.
\end{proof}

Combining our results above, since $\sk_*^\cS$ is lax symmetric monoidal by \cref{cor: sk is lax on spaces}, then the induced functor $\cE~\otimes~\sk_*^\cS\colon \cE~\otimes~\Fun(\Delta^\op, \cS)\to \cE\otimes \Fun(\nn,\cS)$ must be lax symmetric monoidal as well. Therefore, we obtain by \cref{prop: skeleton is determined on spaces} the following.

\begin{corollary}\label{cor: EZ for presentable case}
   Let $\cE$ be a presentably symmetric monoidal \category.
   Then the skeleton filtration $\sk_*^\cE\colon \Fun(\Delta^\op, \cE)\rightarrow \Fun(\nn, \cE)$ is lax symmetric monoidal.
\end{corollary}

\bibliographystyle{alpha}
\bibliography{references}
\end{document}